\documentclass[12pt]{amsart}

\usepackage{fancyhdr}
\usepackage{amsmath}
\usepackage{amssymb}
\usepackage{amsfonts}
\usepackage{geometry}
\usepackage{url}
\usepackage{amssymb}
\usepackage{amsfonts}
\usepackage{geometry}
\usepackage{url}
\usepackage{tikz}
\usepackage{color}
\usepackage{setspace}
\usepackage{amsthm}
\usepackage{amsmath,amscd}
\usepackage{esint}

\makeatletter
\def\@tocline#1#2#3#4#5#6#7{\relax
  \ifnum #1>\c@tocdepth 
  \else
    \par \addpenalty\@secpenalty\addvspace{#2}%
    \begingroup \hyphenpenalty\@M
    \@ifempty{#4}{%
      \@tempdima\csname r@tocindent\number#1\endcsname\relax
    }{%
      \@tempdima#4\relax
    }%
    \parindent\z@ \leftskip#3\relax \advance\leftskip\@tempdima\relax
    \rightskip\@pnumwidth plus4em \parfillskip-\@pnumwidth
    #5\leavevmode\hskip-\@tempdima
      \ifcase #1
       \or\or \hskip 1em \or \hskip 2em \else \hskip 3em \fi%
      #6\nobreak\relax
    \dotfill\hbox to\@pnumwidth{\@tocpagenum{#7}}\par
    \nobreak
    \endgroup
  \fi}
\makeatother

\usepackage{hyperref}

\newtheorem{thm}{Theorem}[section]
\newtheorem{lemma}[thm]{Lemma}
\newtheorem{proposition}[thm]{Proposition}
\newtheorem{corollary}[thm]{Corollary}
\newtheorem{fact}[thm]{Fact}

\theoremstyle{definition}
\newtheorem{definition}[thm]{Definition}
\newtheorem{example}[thm]{Example}

\newcommand{\ignore}[1]{}
\newcommand{\DS}{\displaystyle}

\renewcommand{\epsilon}{\varepsilon}
\renewcommand{\div}{\operatorname{div}}

\newcommand{\tr}{\operatorname{tr}}
\newcommand{\diag}{\operatorname{diag}}
\newcommand{\C}{\mathbb{C}}

\newcommand{\R}{\mathbb{R}}

\newcommand{\wda}{\widetilde{A}}
\renewcommand{\phi}{\varphi}
\renewcommand{\l}{\left}
\renewcommand{\r}{\right}
\newcommand{\lb}{\lbrace}
\newcommand{\rb}{\rbrace}
\newcommand{\dist}{\operatorname{dist}}
\newcommand{\diam}{\operatorname{diam}}
\newcommand{\mult}{\operatorname{mult}}
\newcommand{\Lip}{\operatorname{Lip}}
\newcommand{\iso}{\operatorname{iso}}
\newcommand{\dil}{\operatorname{dil}}
\newcommand{\hreff}[1]{\hyperref[#1]{\ref*{#1}}}

\numberwithin{equation}{section}

\begin{document}

\title{\hspace{0.75cm} Integral estimates for approximations by \newline volume preserving maps}

\author{Christopher Policastro}
\address{University of California, Berkeley, Department of Mathematics, 958 Evans Hall \#3840, Berkeley, CA 94720}
\email{cpoli@math.berkeley.edu}

\vspace{-1.4cm}

\begin{abstract} 
A quantitative Brenier decomposition shows that the deviation of a map from volume preserving is bounded by the deviation of the derivative from volume preserving. A study of the matrix nearness problem for $SL(n)$ and $Sp(2n)$ relates the estimate to incompressible deformations of elastic materials.      
\end{abstract}

\maketitle

\vspace{-2cm}

\tableofcontents

\vspace{-1.4cm}

\section{Background} \label{BACK}

A deformation $u: U \to \R^n$ of a homogeneous hyperelastic material can be understood through an energy function $W: \operatorname{Mat}^{n \times n} \to \R$ measuring the stored energy of $u(U)$. The derivative $DW$ relates the pressure of the deformation to the displacement of the material under the deformation. If $v:U \to \R^n$ satisfies $W(Dv)\equiv 0$, then the deformation does no work on the material. We can try to approximate $u$ by $v$ with a bound by $W(Du)$ on the difference. For small deformations, we have the energy function $$W_{so}(A) = \frac{1}{4} \l|A+A^T\r|^2 \ .$$ Note that $W_{so}(A) = \dist^2\l(A,so(n)\r)$ where $so(n)=\l\lb A \in \operatorname{Mat}^{n \times n} : A^T=-A \r\rb$. An esimtate of Korn implies that for  $u \in W^{1,p}(U,\R^n)$ with $1 < p < \infty$, there exist $A \in so(n)$ and $c \in \R^n$ such that 
\begin{equation} \label{BACK_Korn}
\int_U~ \l| u - \l(A x + c\r) \r|^p\ dx \leq C \int_U~ \operatorname{dist}^p\l(Du,so(n)\r)  \ dx 
\end{equation}
For large deformations, we have the energy function $$W_{SO}(A)=\l|\l(A^TA\r)^{\frac{1}{2}} - I\r|^2 \ .$$ Note that $W_{SO}(A) = \dist^2\l(A,SO(n)\r)$ for $\det A >0$ where $SO(n)=\newline \l\lb A \in \operatorname{Mat}^{n \times n} : A^T=A^{-1} \text{ and } \det A =1 \r\rb$. An estimate of Kohn \cite[p.~134]{K} and Friesecke-James-M\"{u}ller \cite[p.~1468]{FJM} implies that for $u \in W^{1,p}(U,\R^n)$ with $1 < p < \infty$, there exist $A \in SO(n)$ and $c \in \R^n$ such that 
\begin{equation} \label{BACK_FJM}
\int_U~ \left|u(x) - (A x + c)\right|^p\ dx \leq C \int_U~ \dist^p(Du, SO(n))\ dx 
\end{equation}
We can split the energy function into an  isochoric part measuring the stored energy of volume preserving deformations and a dilational part measuring the stored energy of volume changing deformations  $$W=W^{\operatorname{iso}} + \kappa\  W^{\operatorname{dil}} \ .$$ If the bulk modulus $\kappa$ is large, then $W^{\operatorname{dil}}$ contributes more to the stored energy meaning the material is nearly incompressible. We can try to approximate a compressible deformation by an incompressible deformation with a bound by $W^{\dil}$ on the difference. For small deformations, we have the energy functions $$W^{\iso}_{\operatorname{Hookean}}(A)=\l|A - \frac{1}{n} \tr A \cdot I\r|^2, \; \; W^{\dil}_{\operatorname{Hookean}}(A) = \frac{1}{n}\l| \tr A \r|^2 \ .$$ Proposition \hreff{SLG_linear} shows that for $u \in W^{1,p}\l(U,\R^n\r)$ with $1\leq p<\infty$, there exists $v \in W^{1,p}(U,\R^n)$ with $\div v = 0$ such that 
\begin{equation} \label{BACK_bdd_linear}
\int_U~ \l|u-v\r|^p\ dx \leq C \int_U~ \l|\tr Du\r|^p\ dx
\end{equation}
where $C:=C(n,p,\diam U)$. For large deformations, we have the energy functions $$W^{\iso}_{\operatorname{neo-Hookean}}(A) = \l|\frac{1}{\det^{\frac{1}{n}}A} A\r|^2 - n, \; \; W^{\dil}_{\operatorname{neo-Hookean}}(A) = \l(1-\det A\r)^2$$ where $\det A >0$.  Corollary \hreff{OTM_cor} shows that for $1<p<\infty$ and $u \in W^{1,\infty}\l(U,\R^n\r)$ injective, there exists $v \in W^{1,\infty}\l(U,\R^n\r)$ incompressible such that 
\begin{equation} \label{BACK_bdd}
\int_U~ \l|u - v \r|^p\ dx \leq C \int_U~ \l|1 - \det Du \r|^p\ dx 
\end{equation}
where $C:=C(n,p,d,\lambda,\Lambda)$. We treat the degenerate case of $p=1$ in Corollary \hreff{OTM_cor_1}. Note that $$W^{\dil}_{\operatorname{Hookean}}(A) =  \dist^2\l(A,sl(n)\r)$$ where $sl(n) := \l\lb A \in \operatorname{Mat}^{n\times n} : \operatorname{tr} A =0\r\rb$. The relation between $\dist\l(\cdot,sl(n)\r)$ and  $W^{\operatorname{dil}}_{\operatorname{Hookean}}$ suggests that the energy function $$W_{SL}(A) = \dist^2\l(A,SL(n)\r)$$ 
might satisfy an estimate comparable to (\hreff{BACK_bdd_linear}) where $SL(n) := \l\lb A \in \operatorname{Mat}^{n\times n} : \det A =1 \r\rb$. For $1\leq p<\infty$ and $u \in W^{1,p}\l( U,\R\r)$ we have 
\begin{equation} \begin{gathered} \label{BACK_fact}
\int_U~ \l| u - \l(x + c\r) \r|^p \ dx \leq \l| U \r| \underset{U}{\sup} \l|u-(x+c)\r|^p  \\ \leq  \l|U\r| \l(\int_U~ \l| 1 - u' \r|\ dx \r)^p \leq \l|U\r|^p \int_U~ \l|1-u' \r|^p\ dx \\ = \l|U\r|^p \int_U~ \dist^p\l(u',SL(1)\r)\ dx
\end{gathered} \end{equation}
where $c:=\fint_U u - x \ dx$. However,  (\hreff{BACK_fact}) does not hold for $n>1$. The energy functions $W_{SL}$ and  $W^{\operatorname{dil}}_{\operatorname{neo-Hookean}}$ are not equivalent. Example \hreff{MNP_example} shows that $W_{SL}$ is small and $W^{\operatorname{dil}}_{\operatorname{neo-Hookean}}$ is large for certain ill-conditioned matrices. Corollary \hreff{MNP_cor} provides bounds on $W_{SL}$ that suggest a modified energy function 
\begin{equation} \label{BACK_modified}
\l(n+ W^{\iso}_{\operatorname{neo-Hookean}}(A^n)\r) W_{SL}\l(A\r)\ .
\end{equation}
Here the distance to $SL(n)$ is weighted by a ratio reflecting the ill-conditioning of the matrix. Proposition \hreff{SLG_cor} shows that  (\hreff{BACK_fact}) holds for $n>1$ with the modified energy function.    

We can understand $W^{\iso}_{\operatorname{neo-Hookean}}$ as measuring stretching and shrinking along lines and $W^{\dil}_{\operatorname{neo-Hookean}}$ as measuring change of volume. We can incorporate stretching and shrinking along planes into the energy function. The Mooney-Rivlin energy function makes the modification to model more accurately deformations arising from forces in several directions (cf. Example  \hreff{NIM_Dynamic_example}). For $A \in \operatorname{Mat}^{2n \times 2n}$ and $$J = \begin{bmatrix} 0_{n \times n} & -1_{n \times n} \\ 1_{n \times n} & 0_{n \times n} \end{bmatrix}$$ the quantity $\l|A^T J A - J\r|^2$ controls the stretching and shrinking along coordinate planes $\R_{x_k} \times \R_{x_{k+n}}$. A deformation $u:U \to \R^{2n}$ that preserves the area of coordinate planes is called symplectic. These deformations are incompressible because they correspond to $Sp(2n) \subset SL(2n)$ (cf. Definition \hreff{SG_defn_sp_SP}). While incompressible deformations are symplectic for $n=1$, symplectic deformations are more rigid than incompressible deformations for $n>1$ because they neither stretch nor shrink the area of coordinate planes. Lemma \hreff{SG_approx_lemma} shows that incompressible deformations can be approximated by symplectic deformations. The observation allows us to treat (\hreff{BACK_bdd}) for symplectic deformations in Proposition \hreff{SG_prop} and to treat Corollary \hreff{SLG_cor} for $Sp(2n)$ in Corollary \hreff{SG_cor}. Proposition \hreff{SG_linearization} shows the analogue of (\hreff{BACK_bdd_linear}) for $sp(2n)$. 

We apply the estimates for $0 \ll \kappa$ to understand the trend of compressible deformations to incompressible deformations. Consider an energy function $W=W^{\iso} + \kappa W^{\dil}$. Firstly, we study deformations subject to boundary conditions. For boundary condition $v: \partial U \to \R^n$, we want to minimize $\int_U~ W\l(Du\r)\ dx$ subject to $u\mid_{\partial U} = v$. Assume that $v$ extends to $U$ with $\l| v\l(U\r)\r| = \l| U \r|$. Suppose $$u_{\kappa} = \operatorname{argmin} \l\lb \int_U~ W\l(Du\r)\ dx : u\mid_{\partial U} = v \r\rb$$ and $$u_{\infty} = \operatorname{argmin}\l\lb \int_U~ W^{\iso}\l(Du\r)\ dx : u\mid_{\partial U} = v \text{ and } \det Du \equiv 1 \r\rb \ .$$
We expect the minimizer $u_{\kappa}$ to relate to the minimizer $u_{\infty}$. Indeed, Dacorogna et al. \cite{Dacorogna} showed $u_{\kappa} \underset{\kappa \to \infty}{\to} u_{\infty}$ for energy functions with certain convexity and coercivity properties. We apply (\hreff{BACK_bdd}) in Proposition \hreff{NIM_prop_static} to determine an incompressible deformation $s_{\kappa}: U \to \R^n$ such that $$\int_U~ \l|u_{\kappa} -s_{\kappa} \r|^2\ dx \leq \frac{C}{\kappa}$$ where $C:=C\l(U,v,W\r)$. 

Secondly, we study the dynamics of deformations over time. For initial conditions $v(x;\kappa),~\widetilde{v}(x;\kappa):U \to \R^n$, the equations of motion are 
\begin{equation} \begin{gathered} \label{BACK_eqn_motion}
\begin{cases} 
u_{tt} = \div DW\l(Du\r) = \div DW^{\iso}\l(Du\r) + \kappa\ \div DW^{\dil}\l(Du\r) & \\ u(x,0;\kappa) = v(x;\kappa) \text{ and } u_t(x,0;\kappa) = \widetilde{v}(x;\kappa)  
\end{cases} 
\end{gathered} \end{equation}
If $v(x;\kappa)$ is incompressible, then we expect the compressible dynamics $\l\lb u(x,t;\kappa) \r\rb_{t \in [0,T]}$ to relate to the incompressible dynamics $\l\lb u(x,t;\infty) \r\rb_{t \in [0,T]}$. Indeed, Schochet \cite{Schochet} used observations about singular limits of hyperbolic systems of equations to show convergence $u(x,t;\kappa) \underset{\kappa \to \infty}{\to} u(x,t;\infty)$. We apply (\hreff{BACK_bdd}) in Proposition \hreff{NIM_prop} to determine incompressible deformations $\l\lb s(x,t;\kappa) \r\rb_{t \in [0,T]}$ such that for $0 \ll \kappa < \infty$ $$ \underset{t \in [0,T]}{\sup}\ \int_{U}\l|u(x,t;\kappa) - s(x,t;\kappa)\r|^2\ dx \leq \frac{C}{\kappa}$$ where $C:=C\l(U,v,\widetilde{v},W\r)$. 

The approach to (\hreff{BACK_bdd}) should be compared with the approach to (\hreff{BACK_FJM}). Heuristically, Friesecke-James-M\"{u}ller decompose $u:U \to \R^n$ as $u=w+v$ where $$\begin{cases} \Delta w = 0 & \text{ in } U \\ w=u & \text{ on } \partial U \end{cases} \; , \; \begin{cases} \Delta v = \div u & \text{ in } U \\ v = 0 & \text{ on } \partial U \end{cases} \ .$$ Here $w$ minimizes $\int_U~ \l|Dz\r|^2\ dx$ subject to a boundary constraint involving $u$. They remove $Dv$ because average distance of $Du$ to $SO(n)$ controls $Dv$. They determine $A \in SO(n)$ nearest to $Dw$ on average among distance preserving maps. This implies $\dist\l(Dw,SO(n)\r) \approx \l|Dw-A\r| \approx \dist\l(A^{-1} Dw-I,so(n)\r)$. The estimate reduces to (\hreff{BACK_Korn}) for $A^{-1}Dw-I$. 

Heuristically, we use the Brenier decomposition to express $u:U \to \R^n$ as $u=D\psi \circ v$ where $$\begin{cases} \det D^2 \psi\l(y\r) = \det D u\l(u^{-1}\l(D\psi(y)\r)\r) & \text{for } y \in u(U) \\ D\psi\l(u(U)\r) \subset u(U) & \end{cases} \; , \; \begin{cases} \det Dv(x)  = 1 & \text{for } x\in U \\ v(U) \subset u(U) & \end{cases}$$ in a weak sense \cite[p.~3]{Brenier2}. Here $v$ is nearest to $u$ on average among incompressible maps. This implies that $\psi$ minimizes $\int_{u\l(U\r)} \l|D\varphi(y) - y \r|^2\ dy$ subject to a determinant constraint involving $u$. We remove $v$ because $\int_U~ |u - v|^2\ dx = \int_{u(U)}~|D\psi - y|^2\ dy$. We bound $\int_{u(U)}~|D\psi - y|^2\ dy$ by $\int_{u(U)}~|w - y|^2\ dy$ where $$\begin{cases} \det Dw(y) =\det Du\l(u^{-1}\l(w(y)\r)\r) & \text{ for } y \in u(U) \\ w\l(u(U)\r) \subset u(U) & \end{cases} \ .$$ Here $w$ is constructed from the flow of a vector field $z:u(U) \to \R^n$. We can control $w$ by $z \approx w - y$. The estimate reduces to (\hreff{BACK_bdd_linear}) for $z$. 

The differences between the approach to (\hreff{BACK_bdd}) and the approach to (\hreff{BACK_FJM}) reflect the differences between the Laplace equation and the Monge-Amp\`{e}re equation. Firstly, solutions to the Monge-Amp\`{e}re equation are not unique because a solution can be composed with a volume preserving map to yield another solution. While distorted solutions can be treated by restricting to normalized domains \cite[p.~138]{V}, we must treat the distortion through a bound on the diameter of the image. Secondly, the Monge-Amp\`{e}re equation is degenerate elliptic. The treatment of existence and regularity of solutions requires control on the degeneracy through control on the determinant. We need bounds on the determinant for existence of weak solutions and $L^p$ estimates on the weak solutions.     

\textbf{Notation.} Throughout maps denoted with letters in the Greek alphabet are scalar valued and maps denoted with letters in the Latin alphabet are vector valued. The set $U \subset \R^n$ denotes an open, bounded, connected region with Lipschitz boundary \cite[p.~12]{Grisvard}. A function $u \in W^{1,\infty}(U,\R^n)$ is identified with its continuous representative. Take $\Lip(u):=\l|\l|Du\r|\r|_{L^{\infty}\l(U,\R^n\r)}$. Further notation is explained in Definition \hreff{OTM_defn_measure_pushforward}, Definition \hreff{OTM_defn_mult_K}, Definition \hreff{MNP_defn_norm_sl_SL}, and Definition \hreff{SG_defn_sp_SP}.  

\vspace{-0.15cm}

\section{Optimal transport maps} \label{OTM}

The main results of Section \hreff{OTM} are Proposition \hreff{OTM_prop} for $1<p<\infty$ and Proposition \hreff{OTM_prop_1} for $p=1$. For the degenerate case $p=1$, we must use a different approach. Restricting to injective maps gives Corollary \hreff{OTM_cor} and Corollary \hreff{OTM_cor_1} stated in (\hreff{BACK_bdd}). Throughout, we must distinguish between different notions of volume preserving.

\begin{definition} \label{OTM_defn_measure_pushforward}
\textbf{(1)} Let $A \subset \R^n$ be measurable and $s \in L^{\infty}\l(A,\R^n\r)$. If there exists $B \subset \R^n$ measurable with $s_{\#}\l(\mathcal{H}^n \llcorner A\r) = \mathcal{H}^n \llcorner B$ \cite[p.~2]{EG}, then call $s$ measure preserving. We can characterize $s$ as measure preserving from the property $$\int_A~ v \circ s\ dx = \int_B~ v\ dy $$ for all $v \in L^1\l(B\r)$. 

\textbf{(2)} Let $A \subset \R^n$ be measurable. Take $s \in L^{\infty}\l(A,\R^n\r)$ differentiable at $x$ \cite[p.~81]{EG} for a.e. $x \in A$. If $\det Ds(x) = 1$ for a.e. $x \in A$, then call the Jacobian equal to one.

\textbf{(3)} Take $A \subset \R^n$ and $u:A \to \R^n$. For $y \in \R^n$ set $$\omega_u(y):=\begin{cases} 0 & \text{ for } P_u(y) \text{ empty } \\ \sum_{x \in P_u(y)} \frac{1}{\det Du(x)} & \text{ for } P_u(y) \text{ nonempty, countable} \\ \infty & \text{ for } P_u(y) \text{ uncountable }  \end{cases}$$ where $P_u(y) := \l\lb x \in A : u(x) = y \text{ and } Du(x) \text{ exists with } \det Du(x) >0  \r\rb$. 
\end{definition}

For $s \in W^{1,\infty}\l(A,\R^n\r)$ injective, the area formula \cite[p.~99]{EG} implies that $s$ is measure preserving if and only if the Jacobian of $s$ is equal to one. The equivalence may fail without the the assumption of injectivity. 

\begin{example} \label{OTM_example}
\textbf{(1)} The map expressible as $$(0,1) \times S^{n-1} \ni (r,\underline{\theta}) \to (2^{1-n} r, 2 \underline{\theta}) \in \R_{>0} \times S^{n-1}$$ in polar coordinates is not measure preserving. However, the Jacobian is equal to one. 

\textbf{(2)} The map $$(-1,1)^n \ni (x_i)_{i=1}^n \to \l( \frac{1}{2} \l|x_i\r| \r)_{i=1}^n \in \R^n$$ is measure preserving. However, the Jacobian is not equal to one.

\textbf{(3)} The injective map $$B_1(0) \ni x \to \frac{x}{|x|} \l(2^n -1 + |x|^n \r)^{1/n} \in \R^n$$ lies in $W^{1,p}\l(B_1(0),\R^n\r)$ for $1 \leq p < n$. It is measure preserving. Its Jacobian is a.e. equal to one. However, the image contains a cavity.  
\end{example}

We should think of measure preserving maps as globally volume preserving, and maps with Jacobian identically equal to one as locally volume preserving. Nonetheless, we can relate the different notions with observations of Brenier-Gangbo \cite{BG}. 

\begin{lemma} \label{OTM_approx_lemma}
 Take $U \subset \R^n$ with $n>1$ and $s \in L^{\infty}(U,\R^n)$ measure preserving. Let $1\leq p<\infty$. For any $\epsilon >0$, there exists $S \in C^{\infty}_{\operatorname{diff}}(\R^n,\R^n)$ with $\det DS \equiv 1$ such that $$\int_U~ \l|s-S\r|^p\ dx \leq \epsilon  \ .$$  
\end{lemma}
\begin{proof}
 Since $s$ is measure preserving, there exists $V \subset \R^n$ measurable such that \newline $s_{\#}\l(\mathcal{H}^n \llcorner U\r) = \mathcal{H}^n \llcorner V$. Note that $V$ is contained in the essential image of $s$. For $N \subset \R^n$ with $|N|=0$, we have $$\l|s^{-1}(N) \r| =\l| N \cap V \r| \leq |N| = 0 \ .$$ Since $s \in L^{\infty}(U,\R^n)$, this implies the existence of $s_1: U \to U$ measure preserving and $\psi: \R^n \to \R$ convex such that $s(x) = D\psi \circ s_1(x)$ for a.e. $x \in U$ \cite[p.~120]{V}. Take $\psi^*(y):= \sup_{x\in \R^n} x \cdot y - \psi(x)$. Note $D \psi^*(V) \subset \overline{U}$ with $$D\psi^* \l(D\psi(x) \r) = x ~ \text{ for a.e. } ~ x\in U$$ and $$D\psi \l(D\psi^*(y) \r) = y ~ \text{ for a.e. } ~ y\in V$$ \cite[p.~66]{V}. This implies that for $A \subset U$ measurable, we have $\l|\l(D\psi^*\r)^{-1}\l(A\r)\r| = \l|D\psi(A)\r|$ and $\l| D\psi^{-1}\l(  D\psi(A) \r) \r| = |A|$. Therefore 
\begin{equation}
 \begin{gathered} \label{OTM_approx_lemma_eqn}
  \l|\l(D\psi^*\r)^{-1}\l(A\r)\r| = \l|D\psi(A)\r| \\ \underset{s \text{ measure preserving }}{=} \l|s^{-1}\l(D\psi(A) \r)\r| \\ = \l|s_1^{-1}\l( D\psi^{-1}\l(  D\psi(A) \r) \r)\r| \\ \underset{s_1 \text{ measure preserving }}{=} \l| D\psi^{-1}\l(  D\psi(A) \r) \r| = |A|
 \end{gathered}
 \end{equation}
 Set $c:=\l( ||s||_{L^{\infty}} + \underset{x \in \overline{U}}{\sup}~ |x| \r)_{i=1}^n$. Note $U \cap \l(c + V\r) = \emptyset$. Let $Q \subset \R^n$ be an open cube containing $\overline{U}$ and $c+V$. We can take $Q \subset B_r(0)$ for $r:=\l(1+\sqrt{n}\r)\l(\l|\l| s\r|\r|_{L^{\infty}} + \underset{x \in \overline{U}}{\sup}~ \l|x\r| \r)$. For $x \in Q$ take $$s_2(x) := \begin{cases} x & x \not\in U \DS \cup \l(c+V\r) \\ c + s(x) & x \in U \\ D\psi^*(x - c) & x \in c+V \end{cases} \ .$$ Note that $s_2:Q \to Q$ is measure preserving because $\l|\l(D\psi^*\r)^{-1}\l(A\r)\r| = |A|$ by (\hreff{OTM_approx_lemma_eqn}). There exists $s_{\epsilon} \in C^{\infty}(Q,Q)$ such that $$\int_Q~ \l|s_2 - s_{\epsilon} \r|^2 \ dx \leq \epsilon$$ with $\det Ds_{\epsilon} \equiv 1$ \cite[p.~156]{BG}. Here we use $n>1$. By construction $s_{\epsilon}(x) = x$ in a neighborhood of $\partial Q$. Extend to $s_{\epsilon} \in C^{\infty}(\R^n,\R^n)$. Note $$\int_U~ \l|s - \l(s_{\epsilon} -c\r) \r|^2 \ dx= \int_U~ \l|s_2 - s_{\epsilon} \r|^2 \ dx  \leq \epsilon \ .$$ Set $S:=s_{\epsilon} - c$. Note $S(x) - x \equiv -c$ for $x \not\in Q$.  
\end{proof}

We extend Lemma \hreff{OTM_approx_lemma} in Lemma \hreff{SG_approx_lemma} for $n$ even.

\begin{definition} \label{OTM_defn_mult_K}
\textbf{(1)} Take $B \subset A \subset \R^n$ and $u: A \to \R^n$. If the mulitplicity function $$\R^n \ni y \to \mathcal{H}^0\l(B \cap u^{-1}\l\lb y \r\rb\r) \in [0,\infty]$$ is essentially bounded, then denote its $L^{\infty}$-norm by $\operatorname{mult}_B(u)$. If $\operatorname{mult}_B(u) = 1$, then call $u$ essentially injective on $B$. \newline \textbf{(2)} For $A \in \operatorname{Mat}^{n \times n}$ with $\det A > 0$, set $K(A) := \frac{\l|A\r|^n}{\det A}$.
\end{definition}

Firstly, we note that for $u \in W^{1,\infty}\l(U,\R^n\r)$ the multiplicity function is measurable \cite[p.~92]{EG}. Moreover, if $\underset{U}{\inf}~ \det Du > 0$, then we can define $\operatorname{mult}_B(u)$ for any $B \subset \subset U$ because $\underset{U}{\sup}~ \frac{\l| Du \r|^n}{\det Du} = \underset{U}{\sup}~ K(Du) < \infty$ \cite[p.~57]{HK}. Note that for $u$ essentially injective, the area formula \cite[p.~99]{EG} implies that $u$ is measure preserving if and only if the Jacobian of $u$ is equal to one. Secondly, we note that 
\begin{equation} \label{OTM_defn_mult_K_eqn}
1 \leq K(A) \leq C\l(n+ W^{\iso}_{\operatorname{neo-Hookean}}\l(A^n\r)\r)^{\frac{1}{2}}
\end{equation}
because $$0<\frac{1}{C} \leq \underset{A \in \operatorname{Mat}^{n \times n} \text{ nonzero} }{\inf}~ \l|\l(\frac{1}{\l|A\r|} A \r)^n \r|$$ where $C:=C(n)$. Therefore we state Proposition \hreff{SLG_cor} in terms of $K$ rather than the energy function (\hreff{BACK_modified}).

\begin{proposition} \label{OTM_prop}
Let $1<p<\infty$. Take $u \in W^{1,\infty}\l( U,\R^n\r)$ with $\operatorname{diam}\l(u(U)\r) \leq d$, $\mult_U(u) \leq m$, and $0 < \lambda \leq \det Du(x) \leq \Lambda$ for a.e. $x \in U$. There exists $s \in W^{1,\infty}\l(U,\R^{n}\r)$ measure preserving such that 
\begin{gather*}
\int_U~ \l|u - \l(\frac{\l|u(U)\r|}{\l|U\r|}\r)^{\frac{1}{n}}s \r|^p~dx \leq C  \int_U~ \l|\frac{\l|u(U)\r|}{\l|U\r|} - \frac{1}{\omega_u\l(u(x)\r)} \r|^p\ dx 
\end{gather*}
where $C=C_0~ m^{3p-2} \l(\frac{d}{\lambda}\r)^p \l(\frac{\Lambda}{\lambda}\r)^{2p-2}$ for a constant $C_0:=C_0(n,p)$.
\end{proposition}

\begin{proof}
Step (1) rephrases the estimate. Step (2) treats $n=1$. Step (3) through Step (7) treat $n>1$. Throughout $W_p$ denotes Wasserstein distance \cite[p.~151]{AGS} and $\mathcal{P}_p$ denotes probability measures on $\R^n$ with finite $p$th moments \cite[p.~106]{AGS}. Set $q := \frac{p}{p-1}$.

\textbf{(1)} Note that $u \in W^{1,\infty}\l(U,\R^n\r)$ and $0 < \lambda \leq \det Du(x)$ for a.e. $x \in U$ imply that $K\l(Du(x)\r) \leq \frac{\l|\l|Du\r|\r|_{L^{\infty}}^n}{\lambda}$ for a.e. $x \in U$. Therefore $u(U) \subset \R^n$ is open \cite[p.~43]{HK}. We have that $u(U)$ is open, bounded, connected with $\diam u(U) >0$ and $\l|u(U)\r| >0$. Note that $\l|u(U)\r| \leq \int_U~ \det Du\ dx \leq \Lambda \l|U\r|$ and $\frac{\lambda}{m} \l|U\r| \leq \int_U~ \frac{1}{m} \det Du\ dx \leq \l|u(U)\r|$ \cite[p.~99]{EG}. This implies that
\begin{equation} \begin{gathered} \label{OTM_prop_bdd1}
\frac{\lambda}{m} \leq \frac{\l|u(U)\r|}{\l| U\r|} \leq \Lambda
\end{gathered} \end{equation}
Take $v:=\l(\frac{\l|u(U)\r|}{\l|U \r|}\r)^{-\frac{1}{n}} u$. Note that $\operatorname{mult}_U(v)=\operatorname{mult}_U(u) \leq m$. Note that (\hreff{OTM_prop_bdd1}) implies 
\begin{equation} \label{OTM_prop_bdd1_5}
\begin{gathered}
\frac{\lambda}{\Lambda} \leq \det Dv(x) \leq \frac{m \Lambda}{\lambda} 
\end{gathered}
\end{equation}
for a.e. $x \in U$. Take $V:=v(U)$. We have that $V$ is open, bounded, connected with $0< \diam V = \l(\frac{\l|u(U)\r|}{\l|U \r|}\r)^{-\frac{1}{n}} d$ and $\l|V\r| >0$. For $N \subset U$ with $\l|N\r| =0$, we have $v(N) \subset V$ measurable with $\l|v(N)\r| \leq \l|\l|Dv\r|\r|_{L^{\infty}}^n \l|N\r| = 0$ \cite[p.~92]{EG}. Taking $N:=\l\lb x \in U~:~(\hreff{OTM_prop_bdd1_5}) \text{ does not hold at } x \r\rb$ shows that $P_v(y) = v^{-1}(y)$ for a.e. $y \in V$. This implies that $0< \mathcal{H}^0\l( P_v(y) \r) \leq m$ for a.e. $y \in V$. By (\hreff{OTM_prop_bdd1_5})  
\begin{equation}
\begin{gathered} \label{OTM_prop_bdd2}
\frac{\lambda}{m\Lambda} \leq \omega_v(y) := \sum_{x \in P_v(y)} \frac{1}{\det Dv(x)} \leq \frac{m\Lambda}{\lambda} 
\end{gathered}
\end{equation}
for a.e. $y \in V$. Note $v_{\#}\l(\mathcal{H}^n \llcorner U \r) \ll \mathcal{H}^n \llcorner V$ with density $\omega_v$ \cite[p.~81, p.~99]{EG}. This implies that
\begin{equation} \label{OTM_prop_bdd7}
\begin{gathered}
\l|v^{-1}\l(N \r) \r| = \int_N~ \omega_v\ dy \leq \frac{m\Lambda}{\lambda} \l|N\r| =0  
\end{gathered}
\end{equation}
for all $N \subset V$ with $\l|N\r|=0$. Therefore $\omega_v(v) \in L^{\infty}\l(U\r)$ and $\frac{1}{\omega_v(v)} \in L^{\infty}\l(U\r)$. Note 
\begin{equation} \label{OTM_prop_bdd4_5}
\begin{gathered}
\int_V~ \omega_v\ dy =  \l|U\r| =\l|u(U)\r| \l(\frac{\l|u(U)\r|}{\l|U\r|}\r)^{-1}  = \l|V\r|  \ .
\end{gathered}
\end{equation}
Set $$\delta :=  C_0
\l(\frac{\l|u(U)\r|}{\l|U \r|}\r)^{p-\frac{p}{n}} m^{3p-2} \l(\frac{d}{\lambda}\r)^p \l(\frac{\Lambda}{\lambda}\r)^{2p-2}
\int_U~ \l|1 - \frac{1}{\omega_v\l(v(x)\r)} \r|^p\ dx$$ where $C_0:=C_0(n,p)$ specified in Step (7). Note $\l( \frac{\l|u(U)\r|}{\l|U \r|} \r) \omega_u\l(u(x)\r) = \omega_v\l(v(x)\r)$ for all $x \in U$. Therefore it suffices to show that
\begin{equation} \begin{gathered} \label{OTM_prop_eqn1}
\int_U~ \l|v - s \r|^p\ dx \leq \delta 
\end{gathered} \end{equation}
for some $s \in W^{1,\infty}\l(U,\R^n\r)$ measure preserving. For $y_0 \in V$, we can replace $v$ with $v-v(y_0)$ and we can replace $s$ with $s+v(y_0)$ in (\hreff{OTM_prop_eqn1}). Therefore we can assume that $0 \in V$.

Note $\delta = 0$ if and only if $\omega_v(v(x)) = 1$ for a.e. $x \in U$. If $\delta=0$, then set $s=v$. Otherwise, we can assume that $\delta >0$. 

\textbf{(2)} Assume $n=1$. Note $\lambda \leq \det Du(x)$ for a.e. $x \in U$ means that $\lambda \leq u'(x)$ for a.e. $x \in U$. We have 
\begin{gather*}
u(y) = u(x) +\int_x^y u'(s)\ ds \\ \geq u(x) + \lambda(y-x)
\end{gather*}
for all $x < y$ in $U$. This shows that  $u:U \to \R$ is injective. Therefore $m=1$ and $\frac{1}{\omega_v(v(x))} = v'(x)$ for a.e. $x \in U$. We have  
\begin{gather*}
\int_U~ \l|v - (x+c) \r|^p\ dx \underset{(\hreff{BACK_fact})}{\leq} \l|U\r|^p \int_U~ \l|1-v' \r|^p\ dx \\ \underset{(\hreff{OTM_prop_bdd4_5})}{=} \l|V\r|^p \int_U~ \l|1-\frac{1}{\omega_v(v)} \r|^p\ dx = \diam^p V \int_U~ \l|1-\frac{1}{\omega_v(v)} \r|^p\ dx
\end{gather*}
for $c=\fint_U v- x\ dx$. Note (\hreff{OTM_prop_bdd1}) implies that $\l(\frac{\l|u(U)\r|}{\l|U\r|}\r)^p \frac{1}{\lambda^p} \geq \frac{1}{m^p} = 1$. Since $\lambda \leq \Lambda$, we have 
\begin{gather*}
 \diam^p V = \l(\frac{\l|u(U)\r|}{\l|U\r|}\r)^{-\frac{p}{n}} d^p \\ \leq \l(\frac{\l|u(U)\r|}{\l|U\r|}\r)^{-\frac{p}{n}} d^p \l(\frac{\l|u(U)\r|}{\l|U\r|}\r)^p \frac{1}{\lambda^p} = \l(\frac{\l|u(U)\r|}{\l|U\r|}\r)^{p-\frac{p}{n}} \l(\frac{d}{\lambda}\r)^p \\ \leq \l(\frac{\l|u(U)\r|}{\l|U\r|}\r)^{p-\frac{p}{n}} \l( 1 \r)^{3p-2} \l(\frac{d}{\lambda}\r)^p  \l(\frac{\Lambda}{\lambda}\r)^{2p-2}   
\end{gather*}
Since $1< C_0$, we have $\int_U~ \l|v - (x+c) \r|^p\ dx \leq \delta$. This shows (\hreff{OTM_prop_eqn1}) for $n=1$.

\textbf{(3)} Assume $n>1$. Note that $\omega_v \mathcal{H}^n \llcorner V$ and $\mathcal{H}^n \llcorner V$ are compactly supported. For $N \subset V$ with $\l|N\r| =0$, we have $\int_N~ \omega_v\ dy \underset{(\hreff{OTM_prop_bdd7})}{=}0$. Note that $\R^n \ni x \to \l|x \r|^p \in \R$ is strictly convex because $1<p<\infty$. Therefore there exists $T:\R^n \to \R^n$ measurable with $T_{\#}\l(\omega_v \mathcal{H}^n \llcorner V \r) = \mathcal{H}^n \llcorner V$ and $$\int_V~\l|y - T(y) \r|^p\ \omega_v(y)dy = W_p^p\l(\omega_v \mathcal{H}^n \llcorner V, \mathcal{H}^n \llcorner V \r)$$ \cite[p.~141]{AGS}. The composition of $v$ and $T$ is defined a.e. in $U$. Set $\widetilde{s}=T\circ v$. 
Note that $\widetilde{s}_{\#}\l(\mathcal{H}^n \llcorner U \r) = T_{\#}\l( \omega_v \mathcal{H}^n \llcorner U \r) =\mathcal{H}^n \llcorner U$. Therefore $\widetilde{s}:U \to \R^n$ is measure preserving. We have
\begin{equation} \begin{gathered} \label{OTM_prop_eqn2}
 \int_U~\l|v - \widetilde{s}\r|^p\ dx =  \int_U~\l|v - T \circ v\r|^p\ dx \\ = \int_V~\l|y - T(y)\r|^p\ \omega_v(y)dy = W_p^p\l(\omega_v \mathcal{H}^n \llcorner V, \mathcal{H}^n \llcorner V \r)
\end{gathered} \end{equation}

\textbf{(4)} We bound (\hreff{OTM_prop_eqn2}) in Step (6) using an estimate of Benamou-Brenier. Step (4) and Step (5) are needed to apply the estimate of Benamou-Brenier.  

Set $\dot{W}^{1,2}\l(B_{2 \diam V}(0)\r)= \l\lb \psi \in W^{1,2}\l(B_{2 \diam V}(0)\r) : \int_{B_{2 \diam V}(0)}~ \psi \ dy = 0 \r\rb$. Note that $\dot{W}^{1,2}\l(B_{2 \diam V}(0)\r)$ is a Hilbert space with inner product $$\int_{B_{2 \diam V}(0)}~ \psi~ \eta \ dy + \int_{B_{2 \diam V}(0)}~ D\psi \cdot D\eta \ dy \ .$$ For $\psi \in \dot{W}^{1,2}\l(B_{2 \diam V}(0)\r)$ we have $$\int_{B_{2 \diam V}(0)}~\l|\psi\r|^2\ dy \leq 2^{2n+2} \diam^2 V \int_{B_{2 \diam V}(0)}~ \l|D\psi\r|^2\ dy$$ because $\fint_{B_{2 \diam V}(0)} \psi\ dx = 0$ \cite[p.~164]{GT}. This implies that the pairing $$\dot{W}^{1,2}\l(B_{2 \diam V}(0)\r) \times \dot{W}^{1,2}\l(B_{2 \diam V}(0)\r) \ni (\psi,\eta) \to \int_{B_{2 \diam V}(0)}~ D\psi \cdot D\eta\ dx \in \R$$ is coercive. Note that the pairing is continuous. The map $$\dot{W}^{1,2}\l(B_{2 \diam V}(0)\r) \ni \psi \to \int_V~ (1-\omega_v) \psi\ dx \in \R$$ is bounded and linear because $\int_V~\l|1-\omega_v\r|^2\ dx \leq \l|V\r| \l|1 + \frac{m\Lambda}{\lambda} \r|^2 < \infty$ by (\hreff{OTM_prop_bdd2}). Therefore Lax-Milgram implies the existence of a unique $\phi \in \dot{W}^{1,2}\l(B_{2 \diam V}(0)\r)$ such that $$\int_V~ (1-\omega_v) \eta\ dx = \int_{B_{2 \diam V}(0)}~ D\phi \cdot D\eta\ dx$$ for all $\eta \in \dot{W}^{1,2}\l(B_{2 \diam V}(0)\r)$. Note that $\int_V~ 1-\omega_v\ dy \underset{(\hreff{OTM_prop_bdd4_5})}{=} 0$ implies 
\begin{equation} \begin{gathered} \label{OTM_prop_cont_eqn}
 \int_V~ \l(1-\omega_v\r) \eta\ dx = \int_V~ \l(1-\omega_v \r)\l(\eta - \fint_{B_{2 \diam V}(0)} \eta\ dy \r)\ dx \\ = \int_{B_{2 \diam V}(0)}~ D\phi \cdot D\l(\eta - \fint_{B_{2 \diam V}(0)} \eta\ dy\r)\ dx \\= \int_{B_{2 \diam V}(0)}~ D\phi \cdot D\eta \ dx 
\end{gathered} \end{equation}
for all $\eta \in W^{1,2}\l(B_{2 \diam V}(0)\r)$. 

Recall that $\omega_v(y) = 0$ for $y \not\in V$ by Definition \hreff{OTM_defn_measure_pushforward}, and $\omega_v \in L^{\infty}\l(V\r)$ by (\hreff{OTM_prop_bdd2}). This implies that $\l(1-\omega_v\r) \mathbf{1}_V \in L^{\infty}(B_{2 \diam V}(0))$. For any $1<r<\infty$, the Neumann problem 
\begin{gather*}
\begin{cases}
 -\Delta \psi = \l(1-\omega_v\r) \mathbf{1}_V & \text{ in } B_{2 \diam V}(0) \\ \frac{\partial \psi}{\partial n} = 0 & \text{ on } \partial B_{2 \diam V}(0) \\ \psi \in W^{1,r}\l(B_{2 \diam V}(0)\r) & \text{ with } \int_{B_{2 \diam V}(0)}~\psi~dy = 0
\end{cases}
\end{gather*}
has a unique weak solution \cite[p.~2149]{GS}. Here we use $n>1$. Therefore (\hreff{OTM_prop_cont_eqn}) implies that $\phi \in W^{1,p}(B_{2 \diam V}(0))$. Set $\widetilde{V}:= \frac{1}{2 \diam V} V$, $$\widetilde{\phi} : B_1\l(0\r) \ni y \to  \phi\l(2\diam V \cdot y\r) \in \R$$ and $$\widetilde{\omega_v}(y) : \R^n \ni y \to  \omega_v\l(2\diam V \cdot y\r) \in \R  \ .$$ Note that 
\begin{equation} \label{OTM_prop_bdd2_5}
\begin{gathered}
\int_{B_1\l(0\r)} D\widetilde{\phi}\cdot D\eta\ dx \underset{(\hreff{OTM_prop_cont_eqn})}{=} \l( 2\diam V \r)^2~\int_{\widetilde{V}}\l(1-\widetilde{\omega_v}\r) \eta\ dx 
\end{gathered}
\end{equation}
for all $\eta \in W^{1,2}\l(B_1\l(0\r)\r)$. Observe that for any $z \in L^q\l(B_1\l(0\r),\R^n\r)$, there exists $Z \in L^q\l(B_1\l(0\r),\R^n\r)$ and $\zeta \in W^{1,q}\l(B_1\l(0\r)\r)$ such that $z = Z + D\zeta$ with 
\begin{equation} \begin{gathered} \label{OTM_prop_constant}
\l|\l|D\zeta\r|\r|_{L^q\l(B_1\l(0\r)\r)} \leq C_1 \l|\l|z\r|\r|_{L^q\l(B_1\l(0\r)\r)}~,~ \int_{B_1\l(0\r)} \zeta\ dy = 0 
\end{gathered} \end{equation}
for constant $C_1:=C_1(n,p)$ and $$\int_{B_1\l(0\r)} Z \cdot D\eta\ dy = 0$$ for all $\eta \in W^{1,p}\l(B_1\l(0\r)\r)$ \cite[p.~2150]{GS}. Here we use $n>1$. For any $\epsilon >0$, we have
\begin{gather*}
 \l|\l|D\widetilde{\phi}\r|\r|_{L^p\l(B_1\l(0\r)\r)} = \underset{\l\lb z \in L^q\l(B_1\l(0\r),\R^n\r)~:~\l|\l|z\r|\r|_{L^q} \leq 1 \r\rb}{\sup} \int_{B_1\l(0\r)} D\widetilde{\phi} \cdot z\ dy \\ = \underset{\l\lb z \in L^q\l(B_1\l(0\r),\R^n\r)~:~\l|\l|z\r|\r|_{L^q} \leq 1 \r\rb}{\sup} \int_{B_1\l(0\r)} D\widetilde{\phi} \cdot \l(Z + D\zeta\r)\ dy \\ = \underset{\l\lb z \in L^q\l(B_1\l(0\r),\R^n\r)~:~\l|\l|z\r|\r|_{L^q} \leq 1 \r\rb}{\sup} \int_{B_1\l(0\r)} D\widetilde{\phi} \cdot D\zeta\ dy 
 \end{gather*}
 \begin{gather*}
 \underset{\text{(\hreff{OTM_prop_constant})}}{\leq} \underset{\l\lb \zeta \in W^{1,q}\l(B_1\l(0\r)\r)~:~\l|\l|D\zeta\r|\r|_{L^q\l(B_1\l(0\r)\r)} \leq C_1 \r\rb}{\sup} \int_{B_1\l(0\r)} D\widetilde{\phi} \cdot D\zeta\ dy \\ \underset{\text{\cite[p.~127]{EG}}}{\leq}  \underset{\l\lb \zeta \in C^{\infty}\l(\overline{B_1}\l(0\r)\r)~:~ \l|\l|D\zeta\r|\r|_{L^q\l(B_1\l(0\r)\r)} \leq C_1 + \epsilon \r\rb}{\sup} \int_{B_1\l(0\r)} D\widetilde{\phi} \cdot D\zeta\ dy \\ \underset{\text{(\hreff{OTM_prop_bdd2_5})}}{=} \l(2\diam V\r)^2 \underset{\l\lb \zeta \in C^{\infty}\l(\overline{B_1}\l(0\r)\r) ~:~\l|\l|D\zeta\r|\r|_{L^q\l(B_1\l(0\r)\r)} \leq C_1 + \epsilon \r\rb}{\sup} \int_{\widetilde{V}} \l(1-\widetilde{\omega_v}\r) \zeta\ dy \\ \leq \l(2\diam V\r)^2 \underset{\l\lb \zeta \in C^{\infty}\l(\overline{B_1}\l(0\r)\r)  ~:~\l|\l|D\zeta\r|\r|_{L^q\l(B_1\l(0\r)\r)} \leq C_1 + \epsilon \r\rb}{\sup} \l|\l|1-\widetilde{\omega_v}\r|\r|_{L^p\l(\widetilde{V}\r)} \l|\l|\zeta\r|\r|_{L^q\l(B_1\l(0\r)\r)} \\ \underset{(\hreff{OTM_prop_constant})}{=} \l(2\diam V\r)^2 \underset{\l\lb \zeta \in C^{\infty}\l(\overline{B_1}\l(0\r)\r) ~:~ \l|\l|D\zeta\r|\r|_{L^q\l(B_1\l(0\r)\r)} \leq C_1 + \epsilon \r\rb}{\sup} \l|\l|1-\widetilde{\omega_v}\r|\r|_{L^p\l(\widetilde{V}\r)} \l|\l|\zeta- \fint_{B_1\l(0\r)}\zeta dy \r|\r|_{L^q\l(B_1\l(0\r)\r)} \\ \underset{\text{\cite[p.~164]{GT}}}{\leq} \l(2\diam V\r)^2 \l|\l|1-\widetilde{\omega_v}\r|\r|_{L^p\l(\widetilde{V}\r)} \underset{\l\lb \zeta \in C^{\infty}\l(\overline{B_1}\l(0\r)\r)~:~\l|\l|D\zeta\r|\r|_{L^q\l(B_1\l(0\r)\r)} \leq C_1 + \epsilon \r\rb}{\sup} 2^n \l|\l|D\zeta\r|\r|_{L^q\l(B_1\l(0\r)\r)} \\ \leq 2^n\l(2\diam V\r)^2\l(C_1 + \epsilon\r) \l|\l|1-\widetilde{\omega_v}\r|\r|_{L^p\l(\widetilde{V}\r)}
\end{gather*}
Therefore  
\begin{equation} \label{OTM_prop_bdd3_5}
\begin{gathered}
\l|\l|D\phi\r|\r|_{L^p(B_{2\diam V}\l(0\r))} = \frac{1}{2 \diam V}\l|\l|D\widetilde{\phi}\r|\r|_{L^p\l(B_1\l(0\r)\r)} \\ \leq 2^n\l(C_1 + \epsilon\r) \frac{\l(2\diam V\r)^2}{2 \diam V} \l|\l|1-\widetilde{\omega_v}\r|\r|_{L^p\l(\widetilde{V}\r)} \\ =  2^n \l(C_1 + \epsilon\r) \l(2\diam V\r) \l|\l|1-\omega_v\r|\r|_{L^p(V)}
\end{gathered}
\end{equation}
Taking $\epsilon \to 0$ shows $\l|\l|D\phi\r|\r|_{L^p(B_{2\diam V}\l(0\r))} \leq 2^{n+1} C_1 \diam V \l|\l|1-\omega_v\r|\r|_{L^p(V)}$. 

\textbf{(5)} Set $$\rho(y,t):=\begin{cases} (1-t)\omega_v(y) + t & \text{ in } V \times [0,1] \\ 1 & \text{ in } \l(B_{2 \diam V}(0)-  V\r) \times [0,1] \\ 0 & \text{ in } \l(\R^n - B_{2 \diam V}(0)\r) \times [0,1]  \end{cases}$$
$\mu_t:= \rho(\cdot,t)~ \mathcal{H}^n$, and 
$$w(y,t):= \begin{cases} \frac{D\phi(x)}{\rho(x,t)} & \text{ in } B_{2 \diam V}(0) \times [0,1] \\ 0 & \text{ in } \l(\R^n - B_{2 \diam V}(0)\r) \times [0,1] \end{cases} \ .$$ Since $1\leq m$ and $\lambda \leq \Lambda$, (\hreff{OTM_prop_bdd2}) implies
\begin{equation}
\begin{gathered} \label{OTM_prop_bdd3}
 \frac{\lambda}{m\Lambda} = \operatorname{min}\l\lb 1,\frac{\lambda}{m\Lambda} \r\rb \leq  \rho \leq \operatorname{max}\l\lb 1, \frac{m\Lambda}{\lambda} \r\rb = \frac{m\Lambda}{\lambda}
\end{gathered} \end{equation}
off a set $N \times [0,1] \subset B_{2\diam V}(0) \times [0,1]$ with $\l|N\r|=0$. This implies that 
\begin{gather*}
\int_{\R^n} \l|x\r|^p\ \mu_t(dx) \leq  \mu_t\l(B_{2 \diam V}(0)\r) \l( 2\diam V\r)^p \\ \leq  2^{p+n} \l|B_1(0)\r| ~ \frac{m\Lambda}{\lambda} \diam^{p+n} V < \infty 
\end{gather*}
for all $t \in [0,1]$. Note that for any $\eta \in C^0\l(\R^n\r) \cap L^{\infty}\l(\R^n\r)$, we have 
\begin{gather*}
\l|\int_{\R^n}\eta(x)\ \mu_t(dx) - \int_{\R^n}\eta(x)\ \mu_s(dx)\r| \leq \int_V~\l|\eta(x)\l(\rho(x,s) - \rho(x,t)\r) \r|\ dx \\ \leq \int_V~ \l|\eta(x)\l(s-t\r)\l(1-\omega_v(x)\r)\r|\ dx \\ \underset{(\hreff{OTM_prop_bdd3})}{\leq} \l(1 + \frac{m\Lambda}{\lambda}\r)\l|t-s\r| \int_V~ \l|\eta(x)\r|\ dx 
\end{gather*}
for $s,t \in [0,1]$. Therefore $$[0,1] \ni t \to \frac{1}{\mu_t\l(B_{2 \diam V}(0) \r)}~\mu_t \in \mathcal{P}_p$$ is continuous with respect to integration against continuous, bounded functions. Note that $w: \R^n \to \R^n$ is measurable. We have
\begin{equation} \begin{gathered} \label{OTM_prop_bdd4}
\int_0^1 \int_{\R^n} \l|w(y,t)\r|^p \rho(y,t)\ dydt \underset{(\hreff{OTM_prop_bdd3})}{\leq} \l(\frac{m\Lambda}{\lambda}\r)^{p-1} \int_0^1 \int_{B_{2 \diam V}(0)} \l|D\phi(y)\r|^p\ dydt \\ = \l(\frac{m\Lambda}{\lambda}\r)^{p-1} \int_{B_{2 \diam V}(0)} \l|D\phi(y)\r|^p\ dy \\ \underset{(\hreff{OTM_prop_bdd3_5})}{\leq} 2^{pn+p}~C_1^p ~\diam^p V ~\l(\frac{m\Lambda}{\lambda}\r)^{p-1} \int_{V} \l|1-\omega_v\r|^p\ dy 
\end{gathered} \end{equation}
Set $C_2:=2^{pn+p} ~ C_1^p$. Note $C_2:=C_2(n,p)$. Observe 
\begin{gather*}
\frac{d}{dt} \rho(y,t) = \begin{cases} 1-\omega_v & \text{ in } V \times [0,1] \\ 0 & \text{ in } \l(\R^n - V \r) \times [0,1]  \end{cases} \ . 
\end{gather*}
Since $\rho(y,t) = 0$ for $y \not\in B_{2\diam V}(0)$, integration by parts gives  
\begin{equation} \begin{gathered} \label{OTM_prop_cont_eqn2}
 \int_0^1\int_{\R^n}~\l(\eta_t + w \cdot D\eta \r)\ \mu_t\l(dy\r)dt = \int_0^1 \int_{B_{2\diam V}(0)}\l(\eta_t + w\cdot D\eta\r)\rho(y,t)\ dydt \\ =-\int_0^1 \int_{B_{2\diam V}(0)}~\eta~\rho_t(y,t)\ dydt + \int_0^1 \int_{B_{2\diam V}(0)}~w\cdot D\eta~\rho(y,t)\ dydt \\ =\int_0^1\int_{V}~\l(\omega_v -1\r)\eta\ dydt + \int_0^1\int_{B_{2\diam V}(0)} w \cdot D\eta~ \rho(y,t) \ dydt \\ = \int_0^1\l(\int_{V}~\l(\omega_v -1\r)\eta~dy + \int_{B_{2\diam V}(0)} D\phi \cdot D\eta\ dy\r)dt  \underset{\text{(\hreff{OTM_prop_cont_eqn})}}{=} 0 
\end{gathered} \end{equation}
for all $\eta \in C^{\infty}_{\operatorname{cpt}}\l(\R^n \times (0,1) \r)$. In summary, we have $(0,1) \ni t \to \frac{1}{\mu_t\l(B_{2\diam V}(0) \r)}~\mu_t \in \mathcal{P}_p$ is continuous with $w: \R^n \times (0,1) \to \R^n$ satisfying $$\int_0^1 \int_{\R^n} \l|w\r|^p\ \mu_t(dy)dt < \infty $$ by (\hreff{OTM_prop_bdd4}) and 
\begin{equation} \label{OTM_prop_cont_eqn_ref}
\frac{d}{dt} \mu_t + \div\l(\mu_t w\r) = 0 
\end{equation}
in a weak sense by (\hreff{OTM_prop_cont_eqn2}). This implies that $$W_p\l(\mu_{r},\mu_{1-r}\r) \leq \int_{r}^{1-r} \l(\int_{\R^n}\l|w\r|^p \mu_t(dy) \r)^{\frac{1}{p}}\ dt$$ for all $0<r<\frac{1}{2}$ \cite[p.~183]{AGS}. Since $\underset{r \to 0}{\lim}~ \mu_{r} = \omega_v \mathcal{H}^n \llcorner V$ and $\underset{r \to 0}{\lim}~ \mu_{1-r} = \mathcal{H}^n \llcorner V$ with respect to integration against continuous, bounded functions, we have $$W_p\l(\omega_v \mathcal{H}^n \llcorner V, \mathcal{H}^n \llcorner V\r) \leq \underset{r\to 0}{\operatorname{liminf}}~ W_p\l(\mu_{r},\mu_{1-r}\r)$$ \cite[p.~153]{AGS}. Therefore 
\begin{equation} \begin{gathered} \label{OTM_prop_eqn3}
W_p\l(\omega_v \mathcal{H}^n \llcorner V, \mathcal{H}^n \llcorner V \r) \leq \int_0^1 \l(\int_{\R^n} \l|w\r|^p\ \mu_t(dy) \r)^{\frac{1}{p}}dt 
\end{gathered} \end{equation} 
\textbf{(6)} Note that 
\begin{equation} \begin{gathered} \label{OTM_prop_Muckenhoupt}
 \int_V~\l|1-\omega_v\r|^p\ dy = \int_V~ \frac{\omega_v^{p-1}}{\omega_v^{p-1}} \l|1-\omega_v\r|^p\ dy \\ \underset{\text{(\hreff{OTM_prop_bdd2})}}{\leq} \l(\frac{m\Lambda}{\lambda}\r)^{p-1} \int_V~ \l|1 - \frac{1}{\omega_v} \r|^p\ \omega_vdy \\ =\l(\frac{m\Lambda}{\lambda}\r)^{p-1} \int_U~ \l|1-\frac{1}{\omega_v\l(v(x)\r)}\r|^p\ dx
\end{gathered} \end{equation}
This implies that  
\begin{equation} \label{OTM_prop_bdd8}
\begin{gathered}
 \int_0^1 \int_{\R^n} \l|w\r|^p\ \mu_t(dy)dt \underset{\text{(\hreff{OTM_prop_bdd4})}}{\leq} C_2 \diam^p V \l(\frac{m\Lambda}{\lambda}\r)^{p-1} \int_V~\l|1-\omega_v\r|^p\ dy \\ = C_2 \l(\frac{\l|u(U)\r|}{\l|U\r|} \r)^{-\frac{p}{n}} d^p \l(\frac{m\Lambda}{\lambda}\r)^{p-1} \int_V~\l|1-\omega_v\r|^p\ dy \\ \underset{\text{(\hreff{OTM_prop_Muckenhoupt})}}{\leq} C_2\l(\frac{\l|u(U)\r|}{\l|U\r|} \r)^{-\frac{p}{n}} d^p \l(\frac{m\Lambda}{\lambda}\r)^{2p-2} \int_U~\l|1-\frac{1}{\omega_v\l(v(x)\r)}\r|^p\ dx \\ \underset{\text{(\hreff{OTM_prop_bdd1})}}{\leq} C_2\l(\frac{\l|u(U)\r|}{\l|U\r|} \r)^{p-\frac{p}{n}} \l(\frac{m}{\lambda}\r)^p d^p \l(\frac{m\Lambda}{\lambda}\r)^{2p-2} \int_U~\l|1-\frac{1}{\omega_v\l(v(x)\r)}\r|^p\ dx  
\end{gathered}
\end{equation}
Therefore
\begin{equation} \begin{gathered} \label{OTM_prop_eqn4}
\int_U~\l|v - \widetilde{s}\r|^p\ dx \underset{(\hreff{OTM_prop_eqn2})}{=}   W_p^p\l(\omega_v \mathcal{H}^n \llcorner V, \mathcal{H}^n \llcorner V \r) \\ \underset{(\hreff{OTM_prop_eqn3})}{\leq} \l(\int_0^1 \l(\int_{\R^n} \l|w\r|^p \mu_t(dy) \r)^{\frac{1}{p}}dt \r)^p \\ \leq \int_0^1 \int_{\R^n} \l|w\r|^p \mu_t(dy)dt \\ \underset{(\hreff{OTM_prop_bdd8})}{\leq}  C_2\l(\frac{\l|u(U)\r|}{\l|U\r|} \r)^{p-\frac{p}{n}} m^{3p-2} \l(\frac{d}{\lambda}\r)^p \l(\frac{\Lambda}{\lambda}\r)^{2p-2} \int_U~\l|1-\frac{1}{\omega_v\l(v(x)\r)}\r|^p\ dx
\end{gathered} \end{equation}

\textbf{(7)} Set $C_0:=1+2^{p+1}C_2$. Note $C_0:=C_0(n,p)$. By Lemma \hreff{OTM_approx_lemma} there exists $s \in C^{\infty}_{\operatorname{diff}}\l(\R^n,\R^n\r)$ with $\det Ds \equiv 1$ such that $\int_U~\l|s - \widetilde{s}\r|^p\ dx \leq \frac{1}{2^{p+1}}\delta$. Here we use $1<n$ and $0<\delta$. We have \begin{gather*}
\int_U~\l|v - s\r|^p\ dx \leq 2^p \int_U~\l|v - \widetilde{s} \r|^p\ dx + 2^p \int_U~\l|s-\widetilde{s}\r|^p\ dy \\ \underset{(\hreff{OTM_prop_eqn4})}{\leq} 2^p \l( \frac{1}{2^{p+1}}\delta \r) + 2^p \int_U~\l|s-\widetilde{s}\r|^p\ dy \\ \leq \frac{1}{2} \delta + 2^p \l(\frac{1}{2^{p+1}}\delta\r) = \delta                                                                                                                                                                                                                                                                                                                                                                                  \end{gather*}
This shows (\hreff{OTM_prop_eqn1}) for $n>1$. 
\end{proof} 

\begin{corollary} \label{OTM_cor}
Let $1<p<\infty$. Take $u \in W^{1,\infty}\l(U,\R^n\r)$ essentially injective with $\operatorname{diam}\l(u(U)\r) \leq d$ and $0 < \lambda \leq \det Du(x) \leq \Lambda$ for a.e. $x \in U$. There exists $S \in W^{1,\infty}\l(U,\R^{n}\r)$ essentially injective, measure preserving such that 
\begin{gather*}
\int_U~ \l|u - S \r|^p\ dx \leq C \int_U~ \l|1 - \det Du \r|^p\ dx 
\end{gather*}
where $C=C_3 \l(\frac{d}{\lambda^{\frac{1}{n}}}\r)^p\l(1 + \frac{1}{\lambda^p} \l(\frac{\Lambda}{\lambda}\r)^{2p-2}\r)$ for a constant $C_3:=C_3(n,p)$.
\end{corollary}

\begin{proof}
Following Step (1) of Proposition \hreff{OTM_prop}, we have 
\begin{gather*}
 u^{-1}(y) = v^{-1}\l( \l(\frac{\l|u(U)\r|}{\l|U\r|} \r)^{-\frac{1}{n}}~ y \r) \underset{(\hreff{OTM_prop_bdd2})}{=} P_v\l( \l(\frac{\l|u(U)\r|}{\l|U\r|} \r)^{-\frac{1}{n}}~ y \r) = P_u(y) 
\end{gather*}
for a.e. $y \in u(U)$. Note $\mathcal{H}^0\l( u^{-1}(y) \r) =1$ for a.e. $y \in u(U)$ by essential injectivity. For $N \subset u(U)$ with $\l|N\r|=0$, we have
\begin{gather*}
 \l|u^{-1}\l(N\r)\r| = \l|v^{-1}\l( \l(\frac{\l|u(U)\r|}{\l|U\r|} \r)^{-\frac{1}{n}}  N \r)\r| \\ = \int_{\l(\frac{\l|u(U)\r|}{\l|U\r|} \r)^{-\frac{1}{n}}  N}~ \omega_v(y)~ dy \\ \underset{(\hreff{OTM_prop_bdd2})}{\leq}  \int_{\l(\frac{\l|u(U)\r|}{\l|U\r|} \r)^{-\frac{1}{n}}  N}~\frac{(1) \Lambda}{\lambda}~dy =\frac{(1) \Lambda}{\lambda} ~ \l|\l(\frac{\l|u(U)\r|}{\l|U\r|} \r)^{-\frac{1}{n}}  N\r| = 0 
\end{gather*}
Therefore 
\begin{gather*}
\frac{1}{\omega_u\l(u(x)\r)} = \l( \sum_{z \in P_u\l(u(x)\r)} \frac{1}{\det Du(z)} \r)^{-1} =  \l( \sum_{z \in u^{-1}\l(u(x)\r)} \frac{1}{\det Du(z)}\r)^{-1} \\ = \l( \frac{1}{\det Du(x)} \r)^{-1} = \det Du(x)
\end{gather*}
for a.e. $x \in U$. Observe that
\begin{equation} \label{OTM_cor_eqn}
\begin{gathered}
 \int_U~ \l|1- \l( \frac{\l|u(U)\r|}{\l|U\r|} \r)^{\frac{1}{n}} \r|^p\ dx \\ \leq \int_U~ \l|1- \l( \frac{\l|u(U)\r|}{\l|U\r|} \r)^{\frac{1}{n}} \r|^p \l|1+ \l( \frac{\l|u(U)\r|}{\l|U\r|} \r)^{\frac{1}{n}} + \ldots + \l( \frac{\l|u(U)\r|}{\l|U\r|} \r)^{\frac{n-1}{n}} \r|^p\ dx \\ = \int_U~ \l|1 -\frac{\l|u(U)\r|}{\l|U\r|}  \r|^p\ dx \\ \underset{\text{\cite[p.~99]{EG}}}{=} \int_U~\l|1 - \fint_U \det Du(y)\ dy \r|^p\ dx = \int_U~\l|\fint_U 1- \det Du(y) \ dy \r|^p\ dx \\ \leq \int_U~ \fint_U \l|1- \det Du(y)\r|^p\ dydx = \int_U~\l|1-\det Du(y) \r|^p\ dy  
\end{gathered}
\end{equation}
Take $s \in W^{1,\infty}\l(U,\R^n\r)$ from Proposition \hreff{OTM_prop}. Recall that for $\delta = 0$ we have $s=u$, and for $\delta >0$ we have $s \in C^{\infty}_{\operatorname{diff}}\l(\R^n,\R^n\r)$. Therefore $s$ is essentially injective. Take $x_0 \in U$. Set $S:=s+u(x_0)$ and $\widetilde{u} := u - u(x_0)$. Note that $\l|\l| \widetilde{u} \r|\r|_{L^{\infty}} \leq d$, $\l|u(U)\r|= \l|U\r|$ and $D \widetilde{u} = Du$. We have 
\begin{gather*}
\l(\frac{\l|u(U)\r|}{\l|U  \r|} \r)^{\frac{p}{n}} \int_U~ \l|u - S \r|^p\ dx = \l(\frac{\l|u(U)\r|}{\l|U  \r|} \r)^{\frac{p}{n}} \int_U~ \l|\widetilde{u} - s \r|^p\ dx \\ \leq 2^p\int_U~ \l|\widetilde{u} -    
\l(\frac{\l|u(U)\r|}{\l|U  \r|} \r)^{\frac{1}{n}} s \r|^p\ dx + 2^p\int_U~ \l|\widetilde{u} -    
\l(\frac{\l|u(U)\r|}{\l|U  \r|} \r)^{\frac{1}{n}} \widetilde{u} \r|^p\ dx 
\end{gather*}
\begin{gather*}
\underset{\text{Proposition \hreff{OTM_prop}}}{\leq} 2^p~C_0~(1)^{3p-2}\l(\frac{d}{\lambda}\r)^p \l(\frac{\Lambda}{\lambda}\r)^{2p-2} \int_U~ \l|\frac{\l|u(U)\r|}{\l|U  \r|} - \det D\widetilde{u} \r|^p\ dx + \\ + 2^p\l|\l| \widetilde{u} \r|\r|^p_{L^{\infty}} \int_U~ \l|1 - \l(\frac{\l|u(U)\r|}{\l|U  \r|} \r)^{\frac{1}{n}} \r|^p\ dx \\ \leq  2^{2p}~C_0~\l(\frac{d}{\lambda}\r)^p \l(\frac{\Lambda}{\lambda}\r)^{2p-2}\int_U~ \l|\frac{\l|u(U)\r|}{\l|U  \r|} - 1 \r|^p\ dx +\\ + 2^{2p}~C_0~\l(\frac{d}{\lambda}\r)^p \l(\frac{\Lambda}{\lambda}\r)^{2p-2} \int_U~ \l|1 - \det D\widetilde{u} \r|^p\ dx + 2^p~d^p \int_U~ \l|1 - \l(\frac{\l|u(U)\r|}{\l|U  \r|} \r)^{\frac{1}{n}} \r|^p\ dx \\
\underset{(\hreff{OTM_cor_eqn})}{\leq} 2^{2p+1}~C_0~\l(\frac{d}{\lambda}\r)^p \l(\frac{\Lambda}{\lambda}\r)^{2p-2} \int_U~ \l|1 - \det Du \r|\ dx + 2^p~d^p \int_U~ \l|1 - \det Du \r|\ dx 
\end{gather*}
Note that (\hreff{OTM_prop_bdd1}) implies $\l(\frac{\l|u(U)\r|}{\l|U  \r|} \r)^{-\frac{p}{n}} \leq \frac{1}{\lambda^{\frac{p}{n}}}$. Divide by $\l(\frac{\l|u(U)\r|}{\l|U  \r|} \r)^{\frac{p}{n}}$ to obtain
\begin{gather*}
\int_U~\l|u - S \r|^p\ dx \leq \l(\frac{\l|u(U)\r|}{\l|U  \r|} \r)^{-\frac{p}{n}} \l(2^{2p+1}~C_0\l(\frac{d}{\lambda}\r)^p \l(\frac{\Lambda}{\lambda}\r)^{2p-2} + 2^pd^p\r) \int_U~ \l|1 - \det Du \r|^p\ dx \\ \leq \l(\frac{d}{\lambda^{\frac{1}{n}}}\r)^p\l(2^{2p+1}~C_0\frac{1}{\lambda^p} \l(\frac{\Lambda}{\lambda}\r)^{2p-2} + 2^p\r) \int_U~ \l|1 - \det Du \r|^p\ dx 
\end{gather*}  
Set $C_3:=2^{2p+1}~C_0 + 2^p$. Note $C_3:=C_3(n,p)$.
\end{proof}
Note (\hreff{OTM_prop_eqn3}) bounds the cost of transporting $\omega_v~ \mathcal{H}^n \llcorner V$ to $\mathcal{H}^n \llcorner V$ by  the cumulative cost of transporting $\mu_0$ to $\mu_1$ through small changes. For small changes, we can relate Wasserstein distance to $H^{-1}$ norm. Indeed, $$W_2\l(\mu_t,\mu_{t+\epsilon} \r) \approx \l|\l|D_{\mathcal{H}^n}\l(\mu_t-\mu_{t+\epsilon}\r) \r|\r|_{H^{-1}\l(\mu_t\r)}$$ for $0< \epsilon \ll 1$ \cite[p.~234]{V}. While we can relate $\mu_t$ and $\mu_{t+\epsilon}$ for arbitrary $\epsilon$, the relation is through an inequality \cite[p.~211]{Santambrogio}. Note $$C_0 \underset{p\to 1}{\to} \infty ~\text{ and }~ C_3 \underset{p\to 1}{\to} \infty \ .$$ Therefore the approach to Proposition \hreff{OTM_prop} and Corollary \hreff{OTM_cor} does not extend to $p=1$. However, we have 
\begin{equation} \label{OTM_prop_p1}
m^{3p-2}~\l(\frac{d}{\lambda}\r)^p~ \l(\frac{\Lambda}{\lambda} \r)^{2p-2} \underset{p=1}{=} \frac{md}{\lambda} 
\end{equation}
and 
\begin{equation} \label{OTM_cor_p1}
\l(\frac{d}{\lambda^{\frac{1}{n}}}\r)^p\l(1 + \frac{1}{\lambda^p} \l(\frac{\Lambda}{\lambda}\r)^{2p-2}\r) \underset{p=1}{=} \frac{d}{\lambda^{\frac{1}{n}}} \l(1 + \frac{1}{\lambda} \r)   
\end{equation}
We use a different approach for $p=1$ obtaining constant (\hreff{OTM_prop_p1}) in Proposition \hreff{OTM_prop_1} and constant (\hreff{OTM_cor_p1}) in Corollary \hreff{OTM_cor_1}. Neither constant depends on $\Lambda$. However, we use an upper bound on the determinant from Hadamard's inequality $\det Du(x) \underset{\forall x \in U}{\leq} \l|\l|Du \r|\r|^n_{L^{\infty}\l(U,\R^n\r)}$. 

\begin{proposition} \label{OTM_prop_1}
Let $u \in W^{1,\infty}\l(U,\R^n\r)$ with $\operatorname{diam}\l(u(U)\r) \leq d$, $\mult_U(u) \leq m$, and $0< \lambda \leq \det Du(x)$ for a.e. $x \in U$. There exists $s \in W^{1,\infty}\l(U,\R^{n}\r)$ measure preserving such that $$\int_U~ \l|u - \l(\frac{\l|u(U)\r|}{\l|U\r|}\r)^{\frac{1}{n}}s \r|\ dx \leq C\int_U~ \l|\frac{\l|u(U)\r|}{\l|U\r|} - \frac{1}{\omega_u\l(u(x)\r)} \r|\ dx$$ where $C=5 \frac{m d}{\lambda}$.
\end{proposition}

\begin{proof}
Step (1) rephrases the estimate. Step (2) treats $n=1$. Step (3) through Step (7) treat $n>1$. Throughout $W_1$ denotes Wasserstein distance \cite[p.~151]{AGS}.

\textbf{(1)} Note that $u \in W^{1,\infty}\l(U,\R^n\r)$ and $0 < \lambda \leq \det Du(x)$ for a.e. $x \in U$ imply that $K\l(Du(x)\r) \leq \frac{\l|\l|Du\r|\r|_{L^{\infty}}^n}{\lambda}$ for a.e. $x \in U$. Therefore $u(U) \subset \R^n$ is open \cite[p.~43]{HK}. We have that $u(U)$ is open, bounded, connected with $\diam u(U) >0$ and $\l|u(U)\r| >0$. Note that
\begin{equation} \label{OTM_prop_1_bdd1}
\begin{gathered} 
\frac{\lambda}{m} \l|U\r| \leq \int_U~ \frac{1}{m} \det Du\ dx \leq \l|u(U)\r|
\end{gathered} 
\end{equation}
\cite[p.~99]{EG}. Take $v:=\l(\frac{\l|u(U)\r|}{\l|U \r|}\r)^{-\frac{1}{n}} u$. Note that $\operatorname{mult}_U(v)=\operatorname{mult}_U(u) \leq m$. Note 
\begin{equation} \label{OTM_prop_1_bdd4}
\l(\frac{\l|u(U)\r|}{\l|U \r|}\r)^{-1} \lambda \leq \det Dv(x) \leq \l(\frac{\l|u(U)\r|}{\l|U \r|}\r)^{-1} \l|\l| Du \r|\r|^n_{L^{\infty}}
\end{equation}
Take $V:=v(U)$. We have that $V$ is open, bounded, connected with $0< \diam V = \l(\frac{\l|u(U)\r|}{\l|U \r|}\r)^{-\frac{1}{n}} d$ and $\l|V\r| >0$. For $N \subset U$ with $\l|N\r| =0$, we have $v(N) \subset V$ measurable with $\l|v(N)\r| \leq \l|\l|Dv\r|\r|_{L^{\infty}}^n \l|N\r| = 0$ \cite[p.~92]{EG}. Taking $N:=\l\lb x \in U~:~ (\hreff{OTM_prop_1_bdd4}) \text{ does not hold at } x \r\rb$ shows that $P_v(y) = v^{-1}(y)$ for a.e. $y \in V$. This implies that $0< \mathcal{H}^0\l( P_v(y) \r) \leq m$ for a.e. $y \in V$. By (\hreff{OTM_prop_1_bdd4}), we have
\begin{equation} \label{OTM_prop_1_bound}
\begin{gathered}
\frac{\l|u(U)\r|}{\l| U\r|} \frac{1}{\l|\l|Du\r|\r|^n_{L^{\infty}}} \leq \omega_v(y) := \sum_{x \in P_v(y)} \frac{1}{\det Dv(x)}  \leq  \frac{m}{\lambda} \frac{\l|u(U)\r|}{\l| U\r|} 
\end{gathered}
\end{equation}
for a.e. $y \in V$. Note $v_{\#}\l(\mathcal{H}^n \llcorner U \r) \ll \mathcal{H}^n \llcorner V$ with density $\omega_v$ \cite[p.~81, p.~99]{EG}. This implies that
\begin{equation} \label{OTM_prop_1_bdd3}
\begin{gathered}
\l|v^{-1}\l(N\r)\r|=\int_N~ \omega_v\ dy \\ \underset{(\hreff{OTM_prop_1_bound})}{\leq} \int_N \frac{m}{\lambda} \frac{\l|u(U)\r|}{\l| U\r|} \ dy = \frac{m}{\lambda} \frac{\l|u(U)\r|}{\l| U\r|}~\l|N\r| = 0  
\end{gathered}
\end{equation}
for all $N \subset V$ with $\l|N\r|=0$. Therefore $\omega_v(v) \in L^{\infty}\l(U\r)$ and $\frac{1}{\omega_v(v)} \in L^{\infty}\l(U\r)$. Note \begin{equation} \label{OTM_prop_1_bdd2}
\begin{gathered}
\int_V~ \omega_v\ dy =  \l|U\r| =\l|u(U)\r| \l(\frac{\l|u(U)\r|}{\l|U\r|}\r)^{-1}  = \l|V\r|  
\end{gathered}
\end{equation}
Set $\delta :=  5 \diam V
\int_U~ \l|1 - \frac{1}{\omega_v\l(v(x)\r)} \r|\ dx$. Note $\l( \frac{\l|u(U)\r|}{\l|U \r|} \r) \omega_u\l(u(x)\r) = \omega_v\l(v(x)\r)$ for all $x \in U$. Note    
\begin{gather*}
5 \diam V = 5 d \l(\frac{\l|u(U)\r|}{\l|U \r|}\r)^{-\frac{1}{n}} \underset{(\hreff{OTM_prop_1_bdd1})}{\leq}  5 \frac{m d}{\lambda} \l(\frac{\l|u(U)\r|}{\l|U \r|}\r)^{1-\frac{1}{n}}
\end{gather*}
Therefore it suffices to show 
\begin{equation} \label{OTM_prop_1_eqn1}
\begin{gathered} 
\int_U~ \l|v - s \r|\ dx \leq \delta 
\end{gathered} 
\end{equation}
for some $s \in W^{1,\infty}\l(U,\R^n\r)$ measure preserving. 

Note $\delta = 0$ if and only if $\omega_v(v(x)) = 1$ for a.e. $x \in U$. If $\delta=0$, then set $s=v$. Otherwise, we can assume that $\delta >0$.

\textbf{(2)} Assume $n=1$. Note $\lambda \leq \det Du(x)$ for a.e. $x \in U$ means that $\lambda \leq u'(x)$ for a.e. $x \in U$. We have 
\begin{gather*}
u(y) = u(x) +\int_x^y u'(s)\ ds \\ \geq u(x) + \lambda(y-x)
\end{gather*}
for all $x < y$ in $U$. This shows that  $u:U \to \R$ is injective. Therefore $m=1$ and $\frac{1}{\omega_v(v(x))} = v'(x)$ for a.e. $x \in U$. We have  
\begin{gather*}
\int_U~ \l|v - (x+c) \r|\ dx \underset{(\hreff{BACK_fact})}{\leq} \l|U\r| \int_U~ \l|1-v' \r|\ dx \\ \underset{(\hreff{OTM_prop_1_bdd2})}{=} \l|V\r| \int_U~ \l|1-v' \r|\ dx  = \diam V \int_U~ \l|1-v' \r|\ dx \\ = \diam V \int_U~ \l|1-\frac{1}{\omega_v(v)} \r|\ dx
\end{gather*}
for $c=\fint_U v- x\ dx$. Therefore $\int_U~ \l|v - (x+c) \r|\ dx \leq \frac{1}{5} \delta$. This shows (\hreff{OTM_prop_1_eqn1}) for $n=1$.

\textbf{(3)} Assume $n>1$. Note $\mathcal{H}^n \llcorner U$ and $\omega_v~ \mathcal{H}^n \llcorner V$ are compactly supported. For $N \subset V$ with $\l|N\r| =0$, we have $\int_N~ \omega_v\ dy \underset{(\hreff{OTM_prop_1_bdd3})}{=} 0$. Therefore there exists $\psi: \R^n \to \R \DS \cup \infty$ convex with $D\psi(U) \subset \overline{V}$ and $s_1: U \to U$ measure preserving with $v(x) = D\psi \circ s_1(x)$ for a.e. $x \in U$ \cite[p.~119]{V}. Note that $D\psi_{\#}\l(\mathcal{H}^n \llcorner U \r) \ll \mathcal{H}^n \llcorner V$ with density $\omega_v$. We show that   
\begin{equation} \begin{gathered} \label{OTM_prop_1_eqn2}
 \int_U~\l|D\psi - s_2\r|\ dx \leq \frac{4}{5} \delta
\end{gathered} \end{equation}
for some $s_2: U \to \R^n$ measure preserving. We use (\hreff{OTM_prop_1_eqn2}) in Step (7) to deduce (\hreff{OTM_prop_1_eqn1}).
 
\textbf{(4)} Recall that $\omega_v(y) = 0$ for $y \not\in V$ by Definition \hreff{OTM_defn_measure_pushforward}. Set $V^{\epsilon} := \l\lb y \in V : \dist(y,\R^n - V) \geq \epsilon \r\rb$ and $$\omega_v^{(\epsilon)}(y):=\begin{cases} \l(\frac{1}{\l|V\r|} \int_{V^{\epsilon}} \omega_v\ dy\r)^{-1} \omega_v(y)& \text{ for } y \in V^{\epsilon} \\ 0 & \text{ for } y \in \R^n - V^{\epsilon}  \end{cases}\ .$$ Note $\frac{1}{2} \leq \frac{1}{\l|V\r|} \int_{V^{\epsilon}} \omega_v\ dy$ for $0<\epsilon \ll 1$ with $\frac{1}{\l|V\r|} \int_{V^{\epsilon}} \omega_v\ dy \underset{\epsilon \to 0}{\nearrow} 1$.  By (\hreff{OTM_prop_1_bound}) this implies that 
\begin{equation} \begin{gathered} \label{OTM_prop_1_bound2}
\frac{\l|u(U)\r|}{\l| U\r|} \frac{1}{\l|\l|Du \r|\r|_{L^{\infty}}^n} \leq  \omega_v^{(\epsilon)}(y) \leq  2\frac{m}{\lambda} \frac{\l|u(U)\r|}{\l| U\r|} 
\end{gathered} \end{equation}
for a.e. $y \in V^{\epsilon}$ with $\epsilon \ll 1$. Note  
\begin{equation} \label{OTM_prop_1_bdd7}
\begin{gathered}
\int_{V^{\epsilon}} \l|\omega_v^{(\epsilon)} - \omega_v \r|\ dx \underset{\epsilon \to 0}{\to} 0 \ . 
\end{gathered}
\end{equation}
This implies that for any $\eta \in C^0(\R^n) \cap L^{\infty}(\R^n)$, we have 
\begin{gather*}
 \l|\int_{\R^n} \omega_v^{(\epsilon)} \eta \ dx - \int_{\R^n} \omega_v \eta \ dx \r| \leq  \int_{V^{\epsilon}} \l| \omega_v^{(\epsilon)} - \omega_v \r| \l|\eta\r| \ dx + \int_{V-V^{\epsilon}} \l|  \eta ~\omega_v \r|\ dx \\ \underset{(\hreff{OTM_prop_1_bound2})}{\leq}      
 \l|\l|\eta\r|\r|_{L^{\infty}\l(\R^n \r)} \int_{V^{\epsilon}} \l| \omega_v^{(\epsilon)} - \omega_v \r| \ dx + 2 \frac{m}{\lambda} \frac{\l|u(U)\r|}{\l|U\r|}~ \l|\l|\eta\r|\r|_{L^{\infty}\l(\R^n \r)}~ \l|V-V^{\epsilon}\r| \underset{\epsilon \to 0}{\to} 0
\end{gather*}
Therefore $\omega_v^{(\epsilon)}~ \mathcal{H}^n \underset{\epsilon \to 0}{\to} \omega_v~ \mathcal{H}^n$ with respect to integration against bounded, continuous functions. 

Take $y_0 \in V$. Note $V^{\epsilon} \subset \overline{B}_{2 \diam V}\l(y_0\r)$ for any $\epsilon > 0$. Set $\mu^{(\epsilon)} := \omega^{(\epsilon)}_v ~\mathcal{H}^n \llcorner \overline{B}_{2 \diam V}\l(y_0\r)$ and $\nu := \mathbf{1}_V~ \mathcal{H}^n \llcorner \overline{B}_{2 \diam V}\l(y_0\r)$. Note that $\overline{B}_{2 \diam V}\l(y_0\r)$ is a closed convex set with \newline $\mu^{(\epsilon)}\l( \overline{B}_{2 \diam V}\l(y_0\r) \r) = \nu \l(\overline{B}_{2 \diam V}\l(y_0\r)\r)$. Therefore there exists $T^{(\epsilon)}:\overline{B}_{2 \diam V}\l(y_0\r) \to \overline{B}_{2 \diam V}\l(y_0\r)$ defined $\mu^{(\epsilon)}$ a.e. such that $T^{(\epsilon)}_{\#} \mu^{(\epsilon)} = \nu$ and
\begin{equation} \label{OTM_prop_1_bdd6}
\int_{\overline{B}_{2 \diam V}\l(y_0\r)} \l|y -T^{\epsilon}(y) \r| \omega_v^{(\epsilon)}\ dy = W_1\l(\mu^{(\epsilon)},\nu \r) 
\end{equation}
\cite[p.~38]{Ambrosio}.

\textbf{(5)} Note $\mathcal{H}^n \llcorner U$ and $\mu^{(\epsilon)}$ are compactly supported. Since $V^{\epsilon}$ is closed, the support of $\mu^{(\epsilon)}$ is contained in $V^{\epsilon}$. For $N \subset V^{\epsilon}$ with $\l|N\r| =0$, we have 
\begin{equation} \label{OTM_prop_1_bdd5}
\begin{gathered}
\mu^{(\epsilon)}\l(N\r) = \int_N~ \omega^{(\epsilon)}_v\ dy \\ \underset{(\hreff{OTM_prop_1_bound2})}{\leq} \int_N~ 2 \frac{m}{\lambda} \frac{\l|u(U)\r|}{\l| U\r|} \ dy = 2\frac{m}{\lambda} \frac{\l|u(U)\r|}{\l| U\r|}~\l|N\r| = 0  
\end{gathered}
\end{equation}
Therefore there exists $\psi^{(\epsilon)}: \R^n \to \R ~\cup~ \infty$ convex with $D\psi^{(\epsilon)}\l(U\r) \subset V^{\epsilon}$ and $D\psi^{\epsilon}_{\#}\l(\mathcal{H}^n \llcorner U \r) = \mu^{(\epsilon)}$ \cite[p.~66]{V}. Note that for $N \subset V^{\epsilon}$ with $\l|N\r|=0$, we have $$\l|\l(D\psi^{(\epsilon)} \r)^{-1}\l(N\r) \r| = \mu^{(\epsilon)}\l(N\r) \underset{(\hreff{OTM_prop_1_bdd5})}{=} 0 \ .$$ Therefore the composition of $T^{(\epsilon)} \circ D\psi^{(\epsilon)}(x)$ is defined for a.e. $x \in U$. Set $s_2^{(\epsilon)}= T^{(\epsilon)} \circ D\psi^{(\epsilon)}$. Note that $\l(s_2^{(\epsilon)}\r)_{\#}\l(\mathcal{H}^n \llcorner U\r) = T^{(\epsilon)}_{\#} \mu^{(\epsilon)} = \nu$.  Therefore $s_2^{(\epsilon)}: U \to \R^n$ is measure preserving. Note 
\begin{equation} \label{OTM_prop_1_eqn3}
\begin{gathered}
 \int_U~\l|D\psi^{(\epsilon)} - s_2^{(\epsilon)}\r|\ dx = \int_U~\l|D\psi^{(\epsilon)} - T^{(\epsilon)} \circ D\psi^{(\epsilon)} \r|\ dx \\ =\int_{V^{\epsilon}} \l|y - T^{(\epsilon)}(y) \r|\omega_v^{(\epsilon)}(y)\ dy = \int_{\overline{B}_{2 \diam V}(y_0)} \l|y - T^{(\epsilon)}(y) \r|\omega_v^{(\epsilon)}(y)\ dy \\ \underset{(\hreff{OTM_prop_1_bdd6})}{=} W_1(\mu^{(\epsilon)},\nu) \\ \underset{\text{\cite[p.~211]{V}
}}{\leq} \int_{\overline{B}_{2\diam V}(y_0)}~ \l|y-y_0 \r| ~\l|\mu^{(\epsilon)} - \nu \r|(dy) \\ \leq \diam V \int_{V^{\epsilon}} ~\l|\mu^{(\epsilon)} - \nu \r|(dy) + \diam V \int_{V- V^{\epsilon}}~\l|\nu\r|(dy) \\ = \diam V \int_{V^{\epsilon}} \l|1- \omega_v^{(\epsilon)}(y) \r|\ dy + \diam V ~\l|V-V^{\epsilon}\r| \\ = \diam V \int_{V^{\epsilon}} \l|1- \frac{1}{\omega_v^{(\epsilon)}}(y) \r|~\omega_v^{(\epsilon)}(y)~ dy + \diam V ~\l|V-V^{\epsilon}\r| \\ = \diam V \int_U~\l|1 - \frac{1}{\omega_v^{(\epsilon)}\l(D\psi^{(\epsilon)}(x)\r) } \r|\ dx + \diam V ~\l|V-V^{\epsilon}\r|
\end{gathered}
\end{equation}
The bound on $W_1(\mu^{(\epsilon)},\nu)$ comes from comparison with the coupling that fixes \newline $\operatorname{min}\l\lb \mathbf{1}_V, \omega_v^{(\epsilon)} \r\rb \mathcal{H}^n \llcorner \overline{B}_{2\diam V}(y_0)$. Any optimal coupling must fix \newline $\operatorname{min}\l\lb \mathbf{1}_V, \omega_v^{(\epsilon)} \r\rb \mathcal{H}^n \llcorner \overline{B}_{2\diam V}(y_0)$ \cite[p.~35]{V}. 

\textbf{(6)} Consider $\epsilon= \frac{1}{k}$. Note that $\omega_v^{\l(\frac{1}{k}\r)} \mathcal{H}^n \underset{k \to \infty}{\to} \omega_v \mathcal{H}^n$ implies for any $t>0$ $$\l|\l\lb x \in \R^n : \l|D\psi^{(\frac{1}{k})}(x) - D\psi(x)\r| \geq t \r\rb\r| \underset{k \to \infty}{\to} 0$$ \cite[p.~71]{V}. Since $\l|\l| D\psi^{\l(\frac{1}{k}\r)} \r|\r|_{L^{\infty}\l(U,\R^n\r)} \leq |y_0| + 2 \diam V$ for all $k>0$, there exists a subsequence denoted by $\l\lb D\psi^{\l(\frac{1}{k}\r)} \r\rb_{k>0}$ such that $\int_U~\l|D\psi^{\l(\frac{1}{k}\r)} - D\psi \r|\ dx \underset{k\to \infty}{\to} 0$. Therefore there exists $k_1>0$ such that $$\int_U~\l|D\psi^{\l(\frac{1}{k}\r)} - D\psi \r|\ dx \leq \frac{1}{5}\delta$$ for all $k_1 \leq k$. By (\hreff{OTM_prop_1_bdd7}), there exists $k_2>0$ such that $$\diam V ~ \l|V-V^{\frac{1}{k}}\r| \leq \frac{1}{5} \delta$$ and $$\diam V\int_{V^{\frac{1}{k}}} \l|\omega_v^{\l(\frac{1}{k}\r)} - \omega_v \r|\ dx \leq \frac{1}{5}\delta$$ for all $k_2\leq k$. Set $k_0:=k_1+k_2$ and $s_2:=s_2^{\l(\frac{1}{k_0}\r)}$. We have 
\begin{equation} \label{OTM_prop_1_bdd8}
\begin{gathered}
 \int_U~\l|D\psi - s_2 \r|\ dx \leq \int_U~\l| D\psi - D\psi^{\l(\frac{1}{k_0}\r)} \r|\ dx + \int_U~\l|D\psi^{\l(\frac{1}{k_0}\r)} - s_2^{\l(\frac{1}{k_0}\r)}\r|\ dx \\ \leq \frac{1}{5}\delta + \int_U~\l|D\psi^{\l(\frac{1}{k_0}\r)} - s_2^{\l(\frac{1}{k_0}\r)}\r|\ dx \\ \underset{\text{\hreff{OTM_prop_1_eqn3}}}{\leq} \frac{1}{5}\delta + \diam V \int_U~ \l|1 - \frac{1}{\omega_v^{\l(\frac{1}{k_0}\r)}\l( D\psi^{\frac{1}{k_0}}(x) \r)} \r|\ dx + \diam V ~ \l|V-V^{\frac{1}{k_0}}\r| \\ = \frac{1}{5}\delta + \diam V \int_{V^{\frac{1}{k_0}}}\l|1 -\frac{1}{ \omega_v^{\l(\frac{1}{k_0}\r)}(y)} \r|~ \omega_v^{\l(\frac{1}{k_0}\r)}(y)\ dy + \diam V ~ \l|V-V^{\frac{1}{k_0}}\r| \\ \leq \frac{1}{5}\delta + \diam V \int_{V^{\frac{1}{k_0}}}\l|1 - \omega_v^{\l(\frac{1}{k_0}\r)}(y) \r|\ dy + \frac{1}{5}\delta \\ \leq \frac{2}{5}\delta + \diam V \int_{V^{\frac{1}{k_0}}}\l|1 - \omega_v(y)\r|\ dy  + \diam V \int_{V^{\frac{1}{k_0}}}\l|\omega_v(y) - \omega_v^{\l(\frac{1}{k_0}\r)}(y) \r|\ dy \\ \leq \frac{2}{5}\delta + \diam V\int_{V}\l|1 - \omega_v(y)\r|\ dy + \frac{1}{5}\delta = \frac{3}{5}\delta + \diam V\int_V~\l|1 - \frac{1}{\omega_v\l(y\r)} \r|~\omega_v\l(y\r)\ dy \\ = \frac{3}{5}\delta + \diam V\int_U~\l|1 - \frac{1}{\omega_v\l(D\psi(x)\r)} \r|\ dx \\ \underset{s_1 \text{ measure preserving }}{=} \frac{3}{5}\delta + \diam V\int_U~\l|1 - \frac{1}{\omega_v\l(D\psi \circ s_1(x)\r)} \r|\ dx \\  = \frac{3}{5}\delta + \diam V\int_U~\l|1 - \frac{1}{\omega_v\l(v(x)\r)} \r|\ dx  = \frac{3}{5}\delta + \frac{1}{5} \delta = \frac{4}{5} \delta
 \end{gathered}
\end{equation}
This shows (\hreff{OTM_prop_1_eqn2}).

\textbf{(7)} Note that $s_2 \circ s_1: U \to \R^n$ is measure preserving because $s_2:=s_2^{\l(\frac{1}{k_0}\r)}$ with $s_2^{\l(\frac{1}{k_0}\r)}$ measure preserving. By \hreff{OTM_approx_lemma} there exists $s \in C^{\infty}_{\operatorname{diff}}\l(\R^n, \R^n\r)$ with $\det Ds \equiv 1$ such that $\int_U~ \l|s- s_2 \circ s_1 \r|\ dx \leq \frac{1}{5} \delta$. Here we use $n>1$ and $\delta >0$. Therefore 
\begin{gather*}
 \int_U~ \l|v - s \r|\ dx = \int_U~ \l| D\psi \circ s_1 - s \r|\ dx \\ \leq \int_U~ \l| D\psi \circ s_1 - s_2 \circ s_1 \r|\ dx + \int_U~ \l| s_2 \circ s_1 - s \r|\ dx \\ \underset{s_1 \text{ measure preserving }}{=} \int_U~ \l| D\psi - s_2 \r|\ dx + \int_U~ \l| s_2 \circ s_1 - s \r|\ dx \\ \underset{(\hreff{OTM_prop_1_bdd8})}{\leq} \frac{4}{5}\delta + \frac{1}{5}\delta = \delta  
\end{gather*}
This shows (\hreff{OTM_prop_1_eqn1}) for $n>1$. 
\end{proof}

\begin{corollary} \label{OTM_cor_1}
Take $u \in W^{1,\infty}\l(U,\R^n\r)$ essentially injective with $\operatorname{diam}\l(u(U)\r) \leq d$ and $0 < \lambda \leq \det Du(x)$ for a.e. $x \in U$. There exists $S \in W^{1,\infty}\l(U,\R^{n}\r)$ essentially injective, measure preserving such that $$\int_U~ \l|u - S \r|\ dx \leq C \int_U~ \l|1 - \det Du \r|\ dx$$ where $C=10~\frac{d}{\lambda^{\frac{1}{n}}}~\l(1 + \frac{1}{\lambda}\r)$
\end{corollary}
\begin{proof}
Following Step (1) of Proposition \hreff{OTM_prop_1}, we have 
\begin{gather*}
 u^{-1}(y) = v^{-1}\l( \l(\frac{\l|u(U)\r|}{\l|U\r|} \r)^{-\frac{1}{n}}~ y \r) \underset{(\hreff{OTM_prop_1_bound})}{=} P_v\l( \l(\frac{\l|u(U)\r|}{\l|U\r|} \r)^{-\frac{1}{n}}~ y \r) = P_u(y) 
\end{gather*}
for a.e. $y \in u(U)$. Note $\mathcal{H}^0\l( u^{-1}(y) \r) =1$ for a.e. $y \in u(U)$ by essential injectivity. For $N \subset u(U)$ with $\l|N\r|=0$, we have
\begin{gather*}
 \l|u^{-1}\l(N\r)\r| = \l|v^{-1}\l( \l(\frac{\l|u(U)\r|}{\l|U\r|} \r)^{-\frac{1}{n}}  N \r)\r| \\ = \int_{\l(\frac{\l|u(U)\r|}{\l|U\r|} \r)^{-\frac{1}{n}}  N}~ \omega_v(y)~ dy \\ \underset{(\hreff{OTM_prop_1_bound})}{\leq}  \int_{\l(\frac{\l|u(U)\r|}{\l|U\r|} \r)^{-\frac{1}{n}}  N}~\frac{1}{\lambda}\frac{\l|u\l(U\r)\r|}{\l|U\r|}~dy  =\frac{1}{\lambda}\frac{\l|u\l(U\r)\r|}{\l|U\r|} ~ \l|\l(\frac{\l|u(U)\r|}{\l|U\r|} \r)^{-\frac{1}{n}}  N\r| = 0 
\end{gather*}
Therefore 
\begin{gather*}
\frac{1}{\omega_u\l(u(x)\r)} = \l( \sum_{z \in P_u\l(u(x)\r)} \frac{1}{\det Du(z)} \r)^{-1} =  \l( \sum_{z \in u^{-1}\l(u(x)\r)} \frac{1}{\det Du(z)}\r)^{-1} \\ = \l( \frac{1}{\det Du(x)} \r)^{-1} = \det Du(x)
\end{gather*}
for a.e. $x \in U$. Observe that
\begin{equation} \label{OTM_cor_1_eqn}
\begin{gathered}
 \int_U~ \l|1- \l( \frac{\l|u(U)\r|}{\l|U\r|} \r)^{\frac{1}{n}} \r|\ dx \\ \leq \int_U~ \l|1- \l( \frac{\l|u(U)\r|}{\l|U\r|} \r)^{\frac{1}{n}} \r| \l|1+ \l( \frac{\l|u(U)\r|}{\l|U\r|} \r)^{\frac{1}{n}} + \ldots + \l( \frac{\l|u(U)\r|}{\l|U\r|} \r)^{\frac{n-1}{n}} \r|\ dx  \\ = \int_U~ \l|1 -\frac{\l|u(U)\r|}{\l|U\r|}  \r|\ dx \\ \underset{\text{\cite[p.~99]{EG}}}{=} \int_U~\l|1 - \fint_U \det Du(y)\ dy \r|\ dx = \int_U~\l|\fint_U 1- \det Du(y) \ dy \r|\ dx \\ \leq \int_U~ \fint_U \l|1- \det Du(y)\r|\ dydx = \int_U~\l|1-\det Du(y) \r|\ dy  
 \end{gathered}
\end{equation}
Take $s \in W^{1,\infty}\l(U,\R^n\r)$ from Proposition \hreff{OTM_prop_1}. Recall that for $\delta = 0$ we have $s=u$, and for $\delta >0$ we have $s \in C^{\infty}_{\operatorname{diff}}\l(\R^n,\R^n\r)$. Therefore $s$ is essentially injective. Take $x_0 \in U$. Set $S:=s+u(x_0)$ and $\widetilde{u} := u - u(x_0)$. Note that $\l|\l| \widetilde{u} \r|\r|_{L^{\infty}} \leq d$, $\l|u(U)\r| = \l|U\r|$, and $D \widetilde{u} = Du$. We have
\begin{gather*}
\l(\frac{\l|u(U)\r|}{\l|U  \r|} \r)^{\frac{1}{n}} \int_U~ \l|u - S \r|\ dx = \l(\frac{\l|u(U)\r|}{\l|U  \r|} \r)^{\frac{1}{n}} \int_U~ \l|\widetilde{u} - s \r|\ dx \\ \leq \int_U~ \l|\widetilde{u} -    
\l(\frac{\l|u(U)\r|}{\l|U  \r|} \r)^{\frac{1}{n}} s \r|\ dx + \int_U~ \l|\widetilde{u} -    
\l(\frac{\l|u(U)\r|}{\l|U  \r|} \r)^{\frac{1}{n}} \widetilde{u} \r|\ dx 
\end{gather*}
\begin{gather*}
\underset{\text{Proposition \hreff{OTM_prop_1}}}{\leq} 5\frac{d}{\lambda} \int_U~ \l|\frac{\l|u(U)\r|}{\l|U  \r|} - \det D\widetilde{u} \r|\ dx + \l|\l| \widetilde{u} \r|\r|_{L^{\infty}(U)} \int_U~ \l|1 - \l(\frac{\l|u(U)\r|}{\l|U  \r|} \r)^{\frac{1}{n}} \r|\ dx \\ \leq 5\frac{d}{\lambda} \int_U~ \l|\frac{\l|u(U)\r|}{\l|U  \r|} - 1 \r|\ dx + 5\frac{d}{\lambda} \int_U~ \l|1 - \det D\widetilde{u} \r|\ dx + d \int_U~ \l|1 - \l(\frac{\l|u(U)\r|}{\l|U  \r|} \r)^{\frac{1}{n}} \r|\ dx \\ \underset{(\hreff{OTM_cor_1_eqn})}{\leq} 5\frac{d}{\lambda} \int_U~ \l| 1 - \det Du \r|\ dx + 5\frac{d}{\lambda} \int_U~ \l|1 - \det Du \r|\ dx + d \int_U~ \l|1 - \det Du \r|\ dx \\ = \l(10 \frac{d}{\lambda} + d\r) \int_U~ \l|1 - \det Du \r|\ dx
\end{gather*}
Note that $\l(\frac{\l|u(U)\r|}{\l|U  \r|} \r)^{-\frac{1}{n}} \underset{(\hreff{OTM_prop_1_bdd1})}{\leq} \frac{1}{\lambda^{\frac{1}{n}}}$. Divide by $\l(\frac{\l|u(U)\r|}{\l|U  \r|} \r)^{\frac{1}{n}}$ to obtain
\begin{gather*}
\int_U~\l|u - S \r|\ dx \leq \l(\frac{\l|u(U)\r|}{\l|U  \r|} \r)^{-\frac{1}{n}} \l(10 \frac{d}{\lambda} + d\r) \int_U~ \l|1 - \det Du \r|\ dx \\ \leq \frac{d}{\lambda^{\frac{1}{n}}}\l(1 + \frac{10}{\lambda}\r) \int_U~ \l|1 - \det Du \r|\ dx \\ \leq 10 \frac{d}{\lambda^{\frac{1}{n}}}\l(1 + \frac{1}{\lambda}\r) \int_U~ \l|1 - \det Du \r|\ dx
\end{gather*}
\end{proof}

We can make several observations regarding the dependence on constants.

Example \hreff{OTM_example} (3) suggests restricting to maps in $W^{1,p}$ for $p>n$. Following \cite{K}, we take $p=\infty$. Note that the approaches to Proposition \hreff{OTM_prop} and Proposition \hreff{OTM_prop_1} rely on the change of variables formula \cite[p.~99]{EG}. This precludes the use of Holder space norms. 

The factor $\frac{d}{\lambda^{\frac{1}{n}}}$ from the constants in Corollary \hreff{OTM_cor} and Corollary \hreff{OTM_cor_1} should be compared to $K^{\frac{1}{n}}$ in Definition \hreff{OTM_defn_mult_K}. Estimates related to (\hreff{BACK_FJM}) can incorporate $K$ \cite[p.~60]{FZ}. However, we do not incorporate $K$ because the dependence on $\l|\l| Du \r|\r|_{L^{\infty}}$. 
For Brenier decomposition $u = D\psi \circ s$, we cannot relate $\l|\l| Du \r|\r|_{L^{\infty}}$ and $\l|\l| D\psi \r|\r|_{L^{\infty}}$ through a rearrangement inequality \cite[p.~109]{V}. Instead, we incorporate the bound on diameter because $\diam D\psi\l(U\r) = \diam u(U) \leq d$.      

We should relate (\hreff{OTM_prop_bdd3_5}) to Calder\'{o}n-Zygmund estimates. Indeed, we can compare the factor $\frac{\Lambda}{\lambda}$ from the constant in Proposition \hreff{OTM_prop} to a Muckenhoupt weight \cite[p.~35]{S}. 

Solutions to the Monge-Amp\`{e}re equation constructed from solutions to a continuity equation (\hreff{OTM_prop_cont_eqn_ref}) are standard \cite[p.~192]{Dacorogna2}. However, the estimates involve constants depending on the region, in particular, the regularity of the boundary. Here the region is $u(U)$, meaning any dependence on the region is a dependence on $u$. This precludes the use of estimates from the Monge-Amp\`{e}re literature \cite[p.~4]{RY}.

\vspace{-0.1cm}

\section{Matrix nearness problems} \label{MNP}

Note that $SL(n) := \l\lb A \in \operatorname{Mat}^{n\times n} : \det A =1 \r\rb$ is the collection of linear measure preserving maps. For $A \in \operatorname{Mat}^{n \times n}$, we can understand the deviation of $A$ from measure preserving as $\l|1-\det A\r|$ or $\dist\l(A,SL(n)\r)$. The main result of Section \hreff{MNP} is Corollary \hreff{MNP_cor}. It relates the different notions of deviation from measure preserving.   

\begin{definition} \label{MNP_defn_norm_sl_SL}
 For $A \in \operatorname{Mat}^{n \times n}$, set $||A|| := \underset{|v|=1}{\operatorname{sup}} \l|A~v\r|$. Observe that for $A \in \operatorname{Mat}^{n \times n}$, we have $||A|| \leq |A| \leq \sqrt{n}~ ||A||$.
\end{definition}

The nonlinear constraint determining $SL(n)$ leads to an unbounded collection of matrices that do not form a subspace of $\operatorname{Mat}^{n\times n}$. Therefore the matrix nearness problem for $SL(n)$ differs from the matrix nearness problem for other collections of matrices \cite{Higham}.  

\begin{example} \label{MNP_example}
Take $A := \begin{bmatrix} m & 0 \\ 0 & m^{-2} \end{bmatrix}$ for $m>3$. Set $B := \DS\frac{1}{\det^{1/n} A} A = \begin{bmatrix} m^{3/2} & 0 \\ 0 & m^{-3/2} \end{bmatrix}$ and $\widetilde{B} := \begin{bmatrix} m & 0 \\ 0 & m^{-1} \end{bmatrix}$. Note 
\begin{equation*}
|A-B| = |A| \cdot \l|1 - \frac{1}{\det^{1/n} A} \r|, \; \; |A-\widetilde{B}| = \frac{1}{||A^{-1}||} \l| 1 - \frac{1}{\det A} \r| 
\end{equation*}
We have $|A-\widetilde{B}| \leq \DS\frac{2}{m} < m(\sqrt{m}-1) \leq |A - B|$. Therefore $$\dist(A,SL(n)) \neq |A| \cdot \l|1 - \frac{1}{\det^{1/n} A} \r|$$ Note that for $1 \ll m$, we have $|1-\det A| \approx 1$ and $\dist(A,SL(n)) \leq \l|A-\widetilde{B}\r| \approx 0$. 
\end{example}

Fact \hreff{MNP_symmetric_lemma} is standard \cite[p.~144]{AGS}. Lemma \hreff{MNP_lemma} reduces the matrix nearness problem to diagonal matrices. 

\begin{fact} \label{MNP_symmetric_lemma}
 Let $A,B \in \operatorname{Mat}^{n\times n}$ be symmetric. If $A$ is positive definite, then $AB$ has real eigenvalues.
\end{fact}

\begin{proof}
Suppose $A B v = \lambda v$ with $\lambda \in \C$ and $v \in \C^n$ nonzero. Note $$\overline{v}^T B v = \lambda \overline{v}^T A^{-1} v \ .$$ Since $A$ and $B$ are symmetric, there exist $M \in O(n)$ and $N \in O(n)$ such that $A = M^T \diag(a_1,\ldots,a_n) M$ and $B = N^T \diag(b_1,\ldots,b_n)N$ where $\{b_j\}_{j=1}^n \subset \R$ and $\{a_j\}_{j=1}^n \subset \R_{>0}$. This implies that $$\overline{v}^T B v = \overline{Nv}^T \diag(b_1,\ldots,b_n) N v  \in \R$$ and $$\overline{v}^T A^{-1} v = \overline{Mv}^T \diag\l(\frac{1}{a_1},\ldots,\frac{1}{a_n}\r) M v \in \R_{>0} \ .$$   
Therefore $\lambda \in \R$.
\end{proof}

\begin{lemma} \label{MNP_lemma} 
Take $A \in \operatorname{Mat}^{n \times n}$ with $\det A >0$. Let $0<a_1 \leq \cdots \leq a_n$ be the singular values of $A$. We have 
\begin{equation} \label{MNP_lemma_eqn0}
\dist^2\l(A,SL(n)\r) = \min\l\lb\sum_{j=1}^n (c_j - a_j)^2 ~:~ c_j > 0 \text{ with } \prod_{j=1}^n c_j =1 \r\rb \ .
\end{equation}
Moreover, there exists a minimizer $\diag(d_1,\ldots,d_n)$ with $0<d_1 \leq \cdots \leq d_n$.
\end{lemma}

\begin{proof}
\textbf{(1)} Note $A = (A A^T)^{\frac{1}{2}} \cdot (A A^T)^{-\frac{1}{2}} A$ where $(A A^T)^{-\frac{1}{2}} A \in SO(n)$.  This implies that for any $B \in SL(n)$, we have $$|A-B| = \l|(A A^T)^{\frac{1}{2}} - B \l( (A A^T)^{-\frac{1}{2}} A \r)^{-1}\r|$$ where $\det B \l( (A A^T)^{-\frac{1}{2}} A \r)^{-1} = 1$. Therefore $$\dist\l(A,SL(n)\r) = \dist((A A^T)^{\frac{1}{2}}, SL(n)) \ .$$ Since $(A A^T)^{\frac{1}{2}}$ is positive definite symmetric, there exists $M \in O(n)$ such that \newline $M^T (A A^T)^{\frac{1}{2}} M = \operatorname{diag}(a_1, \ldots, a_n)$. Recall that $0<a_1\leq \ldots \leq a_n$ are the singular values of $A$. This implies that for any $B \in SL(n)$, we have $$\l|(A A^T)^{\frac{1}{2}}-B\r| = \l|M^T(A A^T)^{\frac{1}{2}}M-M^TBM\r| = \l|\diag(a_1,\ldots,a_n)-M^TBM\r|$$ where $\det M^T B M = 1$. Therefore 
\begin{equation} \label{MNP_lemma_eqn6}
\begin{gathered}
\dist\l(A,SL(n)\r) = \dist\l(\diag(a_1,\ldots,a_n), SL(n)\r) \ . 
\end{gathered}
\end{equation}
\textbf{(2)} Note that the derivative of $\operatorname{Mat}^{n \times n} \ni X \to \det X - 1 \in \R$ at any $B \in SL(n)$ is $\operatorname{cof} B = \det B ~ (B^{-1})^T = (B^{-1})^T$. Note that the derivative of $\operatorname{Mat}^{n \times n} \ni X \to \l|\diag(a_1,\ldots,a_n)-X\r|^2 \in \R$ at any $B \in SL(n)$ is $2 \l(B-\diag(a_1,\ldots,a_n) \r)$. Take 
\begin{equation} \label{MNP_lemma_eqn2}
\begin{gathered}
B \in \operatorname{argmin} \l\lb \l|\diag(a_1,\ldots,a_n) - X\r|^2~:~ X \in SL(n) \r\rb \ .
\end{gathered}
\end{equation}
Observe that $B$ is a critical point of $\operatorname{Mat}^{n \times n} \ni X \to \l|\diag(a_1,\ldots,a_n)-X\r|^2 \in \R$ restricted to $SL(n)$. Therefore there exists $\lambda \in \R$ such that  
\begin{gather*}
B - \diag(a_1,\ldots,a_n)= \lambda~ (B^{-1})^T \ .  
\end{gather*}
Set $S:=(B B^T)^{\frac{1}{2}}$ and $O:=(B B^T)^{-\frac{1}{2}} B$. Note $B=SO$ with $O \in SO(n)$ and $S$ positive definite symmetric with $\det S=1$. Substituting, we have 
\begin{gather*}
 SO - \diag(a_1,\ldots,a_n) = \lambda ~S^{-1}O
\end{gather*}
Rearranging, we have
\begin{equation} \label{MNP_lemma_eqn1}
\begin{gathered} 
O^T = \diag(a_1,\ldots,a_n)^{-1} \l(S-\lambda~ S^{-1}\r) \\ = \diag\l(\frac{1}{a_1},\ldots,\frac{1}{a_n}\r) \l(S-\lambda~ S^{-1}\r) \ .
\end{gathered}
\end{equation}
Note that $S- \lambda~ S^{-1}$ is symmetric. Note that $\diag\l(\frac{1}{a_1},\ldots,\frac{1}{a_n}\r)$ is positive definite symmetric. Therefore, Fact \hreff{MNP_symmetric_lemma} implies that the eigenvalues of $O^T$ are real. 

\textbf{(3)} Since $O \in SO(n)$, we have $O^T O = I = O O^T$. This implies the existence of $U \in U(n)$ such that $\overline{U}^T O^T U = \diag(\lambda_1,\ldots,\lambda_n)$. Here $\l\lb \lambda_j \r\rb_{j=1}^n \subset \R$ are the eigenvalues of $O^T$. Note that 
\begin{gather*}
\l(\overline{U}^T O^T U\r)^{-1} \underset{O \in SO(n)}{=} \overline{U}^T O U \underset{O \in \operatorname{Mat}^{n\times n}}{=}\overline{\overline{U}^T O^T U}^T \\ =\overline{\diag(\lambda_1,\ldots,\lambda_n)}^T = \diag\l(\lambda_1,\ldots,\lambda_n\r) \ . 
\end{gather*}
Therefore $$I = \l(\overline{U}^T O^T U\r)^{-1}~\l(\overline{U}^T O^T U\r) = \diag\l(\lambda_1^2,\ldots,\lambda_n^2\r) \ .$$ This implies that $\lambda_j = \pm 1$ for $j=1,\ldots,n$. Note that 
\begin{gather*}
O^T O^T =\l( U \diag(\lambda_1,\ldots,\lambda_n) \overline{U}^T \r) \l( U \diag(\lambda_1,\ldots,\lambda_n) \overline{U}^T \r) \\ = U  \diag\l(\lambda_1^2,\ldots,\lambda_n^2\r) \overline{U}^T  = U~ \overline{U}^T=I \ .  
\end{gather*}
Therefore $O^T = \l(O^T\r)^{-1}$. Since $O \in SO(n)$, we deduce $O^T = O$. Since $O$ is symmetric, we can take $U \in O(n) \subset U(n)$. Moreover, 
\begin{gather*}
\diag\l(a_1,\ldots,a_n\r) O = \diag\l(a_1,\ldots,a_n\r) O^T \\ \underset{(\hreff{MNP_lemma_eqn1})}{=} S - \lambda S^{-1} \underset{S^T = S}{=} \l(S - \lambda S^{-1} \r)^T \\ \underset{(\hreff{MNP_lemma_eqn1})}{=} O \diag(a_1,\ldots,a_n)^T =O \diag(a_1,\ldots,a_n) \ .
\end{gather*}
Therefore $O$ and $\diag(a_1,\ldots,a_n)$ are simultaneously diagonalizable with
\begin{equation} \label{MNP_lemma_eqn3}
\begin{gathered}
U^T O U = \overline{U}^T O^T \overline{U}=\diag\l(\lambda_1,\ldots,\lambda_n\r) =  \diag(\pm 1,\ldots,\pm 1)  \\ \text{ and }~ U^T \diag(a_1,\ldots,a_n) U = \diag\l(a_{\sigma(1)},\ldots,a_{\sigma(n)}\r) 
\end{gathered}
\end{equation}
where $\sigma: \l\lb 1,\ldots,n\r\rb \to \l\lb 1,\ldots,n\r\rb$ is a permutation. Set $\widetilde{S}:= U^T S U $. Since $S$ is positive definite symmetric with $\det S = 1$, we have $\widetilde{S}$ is positive definite symmetric with $\det \widetilde{S} = 1$. Note that 
\begin{equation} \label{MNP_lemma_eqn4}
\begin{gathered}
 \dist^2 \l(\diag(a_1,\ldots,a_n),SL(n)\r) \underset{(\hreff{MNP_lemma_eqn2})}{=}  \l|B-\diag(a_1,\ldots,a_n)\r|^2 \\ = \l|S O - \diag(a_1,\ldots,a_n)\r|^2 \\ = \l| \l(U^T S U\r) \l(U^T O U \r) - U^T \diag(a_1,\ldots,a_n) U \r|^2 \\ \underset{(\hreff{MNP_lemma_eqn3})}{=}
 \l|\widetilde{S} \diag(\lambda_1,\ldots,\lambda_n) - \diag\l(a_{\sigma(1)},\ldots,a_{\sigma(n)}\r) \r|^2 \\= \l|\widetilde{S} - \diag\l(a_{\sigma(1)},\ldots,a_{\sigma(n)}\r) \diag(\lambda_1,\ldots,\lambda_n) \r|^2 \ .
\end{gathered}
\end{equation}
Note that  
\begin{equation} \label{MNP_lemma_eqn5}
\begin{gathered}
\widetilde{S} - \lambda \widetilde{S}^{-1} =U^T \l( S - \lambda S^{-1} \r) U \\ \underset{(\hreff{MNP_lemma_eqn1})}{=} U^T \l( \diag(a_1,\ldots,a_n) O^T \r) U  = \l(U^T \diag(a_1,\ldots,a_n) U \r)\l(U^T O^T U \r) \\ \underset{(\hreff{MNP_lemma_eqn3})}{=} \diag\l(a_{\sigma(1)},\ldots,a_{\sigma(n)}\r) ~\diag(\lambda_1,\ldots,\lambda_n) \ .
\end{gathered}
\end{equation}
Since $\widetilde{S}$ is positive definite symmetric, there exists $N \in O(n)$ such that $N^T \widetilde{S} N = \diag(b_1,\ldots,b_n)$ with $0 < b_1 \leq \ldots \leq b_n$. Since $\det \widetilde{S} =1$, we have $\prod_{j=1}^n~b_j = 1$. Multiplying (\hreff{MNP_lemma_eqn5}) by $N^T$ and $N$, we have 
\begin{equation} \label{MNP_lemma_eqn8}
\diag(b_1,\ldots,b_n) - \lambda \diag\l(\frac{1}{b_1},\ldots,\frac{1}{b_n}\r) =   N^T \diag\l(\lambda_1~ a_{\sigma(1)},\ldots,\lambda_n~ a_{\sigma(n)}\r) N \ .
\end{equation}
Note $\diag(b_1,\ldots,b_n) - \lambda \diag\l(\frac{1}{b_1},\ldots,\frac{1}{b_n}\r)$ diagonal implies that 
\begin{equation} \label{MNP_lemma_eqn9}
N^T \diag\l(\lambda_1~ a_{\sigma(1)},\ldots,\lambda_n~ a_{\sigma(n)}\r) N = \diag\l(\lambda_{\widetilde{\sigma}(1)}~ a_{\widetilde{\sigma}\circ \sigma (1)},\ldots,\lambda_{\widetilde{\sigma}(n)}~ a_{\widetilde{\sigma} \circ \sigma(n)}\r) 
\end{equation}
where $\widetilde{\sigma}: \l\lb 1,\ldots,n\r\rb \to \l\lb 1,\ldots,n\r\rb$ is a permutation. Therefore   
\begin{equation} \label{MNP_lemma_eqn11}
\begin{gathered}
 \dist^2\l(\diag(a_1,\ldots,a_n),SL(n)\r) \underset{(\hreff{MNP_lemma_eqn4})}{=} \l| \widetilde{S} - \diag\l(\lambda_1~ a_{\sigma(1)},\ldots,\lambda_n~ a_{\sigma(n)}\r) \r|^2 \\ = \l|N^T \widetilde{S} N - N^T \diag\l(\lambda_1~ a_{\sigma(1)},\ldots,\lambda_n~ a_{\sigma(n)}\r) N \r|^2 \\ \underset{(\hreff{MNP_lemma_eqn9})}{=} \l|\diag(b_1,\ldots,b_n) -  \diag\l(\lambda_{\widetilde{\sigma}(1)}~ a_{\widetilde{\sigma} \circ \sigma(1)},\ldots,\lambda_{\widetilde{\sigma}(n)}~ a_{\widetilde{\sigma}\circ \sigma (n)}\r)\r|^2 \\ \underset{(\hreff{MNP_lemma_eqn3})}{=} \l|\diag(\lambda_{\widetilde{\sigma}(1)}~b_1,\ldots,\lambda_{\widetilde{\sigma}(n)}~b_n) -  \diag\l(a_{\widetilde{\sigma} \circ \sigma(1)},\ldots, a_{\widetilde{\sigma}\circ \sigma (n)}\r)\r|^2 \\ = \l| \diag\l(\lambda_{\sigma^{-1}(1)}~ b_{\sigma^{-1} \circ\widetilde{\sigma}^{-1}(1)},\ldots,\lambda_{\sigma^{-1}(n)}~ b_{\sigma^{-1} \circ\widetilde{\sigma}^{-1}(n)}\r) - \diag(a_1,\ldots,a_n)\r|^2   
\end{gathered}
\end{equation}
Since $0<a_j$ and $0<b_j$ for $1 \leq j \leq n$, we have $$\l| b_{\sigma^{-1} \circ\widetilde{\sigma}^{-1}(j)} - a_j\r| \leq \l| \lambda_{\sigma^{-1}(j)}~ b_{\sigma^{-1}\circ \widetilde{\sigma}^{-1}(j)} - a_j\r| \leq \l|  b_{\sigma^{-1} \circ\widetilde{\sigma}^{-1}(j)} + a_j\r|$$ for $1 \leq j \leq n$ because $\lambda_{\sigma^{-1}(j)} = \pm 1$. This implies that 
\begin{gather*}
\l| \diag\l(b_{\sigma^{-1} \circ\widetilde{\sigma}^{-1}(1)},\ldots,b_{\sigma^{-1} \circ\widetilde{\sigma}^{-1}(n)}\r) - \diag(a_1,\ldots,a_n)\r|^2 \\ \leq \l| \diag\l(\lambda_{\sigma^{-1}(1)}~ b_{\sigma^{-1} \circ\widetilde{\sigma}^{-1}(1)},\ldots,\lambda_{\sigma^{-1}(n)}~ b_{\sigma^{-1} \circ\widetilde{\sigma}^{-1}(n)}\r) - \diag(a_1,\ldots,a_n)\r|^2 \ .
\end{gather*}
Since $\prod_{j=1}^n~b_j = 1$, we have $\diag\l(b_{\sigma^{-1} \circ\widetilde{\sigma}^{-1}(1)},\ldots,b_{\sigma^{-1} \circ\widetilde{\sigma}^{-1}(n)}\r) \in SL(n)$. Therefore we can assume $\lambda_j =1$ for $1 \leq j \leq n$. We have   
\begin{equation} \label{MNP_lemma_eqn7}
\begin{gathered}
\dist^2\l(A,SL(n)\r) \underset{(\hreff{MNP_lemma_eqn6})}{=} \dist^2\l(\diag(a_1,\ldots,a_n),SL(n)\r) \\ \underset{(\hreff{MNP_lemma_eqn11})}{=} \l| \diag\l(\lambda_{\sigma^{-1}(1)}~ b_{\sigma^{-1} \circ\widetilde{\sigma}^{-1}(1)},\ldots,\lambda_{\sigma^{-1}(n)}~ b_{\sigma^{-1} \circ\widetilde{\sigma}^{-1}(n)}\r) - \diag(a_1,\ldots,a_n)\r|^2 \\ = \l| \diag\l(b_{\sigma^{-1} \circ\widetilde{\sigma}^{-1}(1)},\ldots,b_{\sigma^{-1} \circ\widetilde{\sigma}^{-1}(n)}\r) - \diag(a_1,\ldots,a_n)\r|^2  
\end{gathered}
\end{equation}
This shows (\hreff{MNP_lemma_eqn0}). Suppose that $b_{\sigma^{-1} \circ\widetilde{\sigma}^{-1}(i)} > b_{\sigma^{-1} \circ\widetilde{\sigma}^{-1}(j)}$ for $i < j$. Since $a_i \leq a_j$, we have
\begin{gather*}
 \l(b_{\sigma^{-1} \circ\widetilde{\sigma}^{-1}(i)} - a_i\r)^2 + \l(b_{\sigma^{-1} \circ\widetilde{\sigma}^{-1}(j)} - a_j\r)^2 = \l(b_{\sigma^{-1} \circ\widetilde{\sigma}^{-1}(i)} - a_j\r)^2 + \l(b_{\sigma^{-1} \circ\widetilde{\sigma}^{-1}(j)} - a_i\r)^2 + \\ + \l(b_{\sigma^{-1} \circ\widetilde{\sigma}^{-1}(i)} - a_j\r)\l(b_{\sigma^{-1} \circ\widetilde{\sigma}^{-1}(j)} - a_i\r) - \l(b_{\sigma^{-1} \circ\widetilde{\sigma}^{-1}(i)} - a_i\r)\l(b_{\sigma^{-1} \circ\widetilde{\sigma}^{-1}(j)} - a_j\r) \\ = \l(b_{\sigma^{-1} \circ\widetilde{\sigma}^{-1}(i)} - a_j\r)^2 + \l(b_{\sigma^{-1} \circ\widetilde{\sigma}^{-1}(j)} - a_i\r)^2 + \l(b_{\sigma^{-1} \circ\widetilde{\sigma}^{-1}(i)} - b_{\sigma^{-1} \circ\widetilde{\sigma}^{-1}(j)}\r)\l(a_j - a_i\r) \\ \geq \l(b_{\sigma^{-1} \circ\widetilde{\sigma}^{-1}(i)} - a_j\r)^2 + \l(b_{\sigma^{-1} \circ\widetilde{\sigma}^{-1}(j)} - a_i\r)^2
\end{gather*}
Therefore we can assume that $b_{\sigma^{-1} \circ\widetilde{\sigma}^{-1}(i)} \leq b_{\sigma^{-1} \circ\widetilde{\sigma}^{-1}(j)}$ for $i < j$. Set $d_j:=b_{\sigma^{-1} \circ\widetilde{\sigma}^{-1}(j)}$.

\textbf{(4)} Incorporating (\hreff{MNP_lemma_eqn9}) into (\hreff{MNP_lemma_eqn8}), we have $d_j - \lambda_{\sigma^{-1}(j)} a_j = \frac{\lambda}{d_j}$ for $1\leq j \leq n$. Since $\lambda_{\sigma^{-1}(j)} =1$ for $1\leq j \leq n$, we have 
\begin{equation} \label{MNP_lemma_eqn10}
d_j - a_j = \frac{\lambda}{d_j} \ .
\end{equation}
This implies that
\begin{equation} \label{MNP_lemma_eqn12}
\begin{gathered} 
d_j = a_j \text{ for $1\leq j\leq n$ iff } \lambda = 0 \text{ iff } \prod_{i=1}^n a_i = 1 \\ d_j > a_j \text{ for $1\leq j\leq n$ iff } \lambda > 0 \text{ iff } \prod_{i=1}^n a_i < 1 \\ d_j < a_j \text{ for $1\leq j\leq n$ iff } \lambda < 0 \text{ iff } \prod_{i=1}^n a_i > 1 
\end{gathered}
\end{equation}
We use Step (4) in Proposition \hreff{MNP_prop}.
\end{proof}
Take $A \in \operatorname{Mat}^{n\times n}$ with $0 < \det A$. Let $0<a_1\leq \ldots \leq a_n$ be the singular values of $A$. We can try to determine a minimizer $\diag\l(c_1,\ldots,c_n\r)$ in (\hreff{MNP_lemma_eqn0}) through matching the larger entries of $\diag\l(a_1,\ldots,a_n\r)$. Take $$c_j:=\begin{cases} a_j & \text{ for } j>1 \\ \frac{1}{a_2 \cdots a_n} & \text{ for } j=1  \end{cases} \ .$$ Note $c_1=\frac{1}{c_2 \cdots c_n} = \frac{a_1}{a_1 \cdots a_n}$. We have $$\l(\sum_{j=1}^n (c_j - a_j)^2 \r)^{1/2} = a_1 \l|1 - \frac{1}{a_1 \cdots a_n}\r| = \frac{1}{||A^{-1}||} \l|1 - \frac{1}{\det A}\r| \ .$$ This implies that 
\begin{equation} \label{MNP_upper_bdd}
 \dist\l(A,SL(n)\r) \leq \frac{1}{||A^{-1}||} \l|1 - \frac{1}{\det A}\r| \leq \frac{\sqrt{n}}{|A^{-1}|} \l|1 - \frac{1}{\det A}\r| \ . 
\end{equation}
If the determinant of $A$ is bounded away from 0, then we can show a lower bound.  
\begin{proposition} \label{MNP_prop}
Take $0<\theta \leq \frac{1}{2}$. If $A \in \operatorname{Mat}^{n \times n}$ with $\det A \geq \theta >0$, then 
\begin{equation} \label{MNP_prop_eqn0}
\frac{1}{\l|A^{-1}\r|}\l|1-\frac{1}{\det A}\r| \leq  \frac{C_4}{\theta} \dist\l(A,SL(n)\r)
\end{equation}
for a constant $C_4:=C_4(n)$.
\end{proposition}
\begin{proof}
Set $\R^n_{>0} := \{(x_j)_{j=1}^n  \in \R^n~:~x_j > 0 \text{ for } 1 \leq j \leq n\}$ and \newline $\R^n_1:=\l\lb (x_j)_{j=1}^n \in \R_{>0}^n~:~ \prod_{j=1}^n~x_j = 1 \r\rb$. Step (1) rephrases the estimate. Step (2) and Step (3) correspond to two regions in $\R^n_{>0}$.

\textbf{(1)} Let $0<a_1\leq \cdots \leq a_n$ be the singular values of $A$ meaning the eigenvalues of $\l(AA^T\r)^{\frac{1}{2}}$. Set $\wda:=\diag(a_1,\ldots,a_n)$. 
Note that 
\begin{equation} \label{MNP_prop_eqn4}
\det \wda = \det A \ .
\end{equation}
Since $\l(AA^T\r)^{\frac{1}{2}}$ is positive definite symmetric, there exists $M \in O(n)$ such that $\l(AA^T\r)^{\frac{1}{2}} = M^T \wda M$. Note that 
\begin{equation} \label{MNP_prop_eqn3}
\l|A^{-1}\r| = \l|\l(\l(AA^T\r)^{\frac{1}{2}}~\l(AA^T\r)^{-\frac{1}{2}} A  \r)^{-1}\r| \\ \underset{\l(AA^T\r)^{-\frac{1}{2}} A \in O(n)}{=} \l|\l(AA^T\r)^{-\frac{1}{2}}\r| \\ \underset{M \in O(n)}{=} \l|\wda^{-1}\r|
\end{equation}
By Lemma \hreff{MNP_lemma}, we have 
\begin{equation} \label{MNP_prop_eqn2}
\dist\l(A,SL(n)\r) = \dist\l(\wda,SL(n)\r) = \l|\wda - D \r|
\end{equation}
for $D:=\diag(d_1,\ldots,d_n)$ with $0 < d_1 \leq \ldots \leq d_n$ and $\prod_{j=1}^n~d_j = 1$. By (\hreff{MNP_prop_eqn4}), (\hreff{MNP_prop_eqn3}), and (\hreff{MNP_prop_eqn2}) it suffices to show that  
\begin{equation} \label{MNP_prop_eqn1}
\frac{1}{\l|\wda^{-1}\r|}\l|1-\frac{1}{\det \wda}\r| \leq \frac{C_4}{\theta} \dist\l(\wda,SL(n)\r) 
\end{equation}
for a constant $C_4:=C_4(n)$.  

\textbf{(2)} Take $(a_j)_{j=1}^n \in \R^n_{>0}$ with $\frac{1}{2} \leq \det \wda$. By Lemma \hreff{MNP_lemma},  
\begin{equation} \label{MNP_prop_eqn6}
\begin{gathered}
\frac{1}{\l|\l|\wda^{-1}\r|\r|}\l|1-\frac{1}{\det \wda}\r| \leq C \dist(\wda,SL(n)) 
\end{gathered}
\end{equation}
for a constant $C:=C(n)$ is equivalent to 
\begin{equation} \label{MNP_prop_eqn8}
\begin{gathered}
a_1 \l|1-\prod_{j=1}^n~ \frac{1}{a_j}\r| \leq C \dist\l((a_j)_{j=1}^n ,\R^n_1\r) 
\end{gathered}
\end{equation}
for a constant $C:=C(n)$. Suppose to the contrary that for all $j>1$ there exist $\l(a_i^{(j)}\r)_{i=1}^n \subset \R^n_{>0}$ with $\frac{1}{2} \leq \prod_{i=1}^n~a^{(j)}_i$ such that 
\begin{equation} \label{MNP_prop_eqn9}
\begin{gathered}
a^{(j)}_1 \l|1-\prod_{i=1}^n~ \frac{1}{a^{(j)}_i}\r| > j \dist\l((a^{(j)}_i)_{i=1}^n ,\R^n_1\r) \\ \text{ and } a_1^{(j)} \leq \ldots \leq a_n^{(j)} \ .
\end{gathered}
\end{equation}
Set $A_j:=\diag\l(a_1^{(j)},\ldots,a_n^{(j)}\r)$. Note 
\begin{equation} \label{MNP_prop_eqn11}
0 \underset{(\hreff{MNP_prop_eqn9})}{<}\l|1-\frac{1}{\det A_j} \r| \underset{\frac{1}{2} \leq \det A_j}{\leq} 1
\end{equation}
This implies that $A_j \not\in SL(n)$. By Lemma \hreff{MNP_lemma}, we have 
\begin{equation} \label{MNP_prop_eqn16}
\begin{gathered}
\dist\l((a^{(j)}_i)_{i=1}^n ,\R^n_1\r)=\dist\l(A_j,SL(n)\r) = \l|A_j- D_j \r|
\end{gathered}
\end{equation}
for $D_j:=\diag(d^{(j)}_1,\ldots,d^{(j)}_n)$ with $0 < d^{(j)}_1 \leq \ldots \leq d^{(j)}_n$ and $\prod_{i=1}^n~d^{(j)}_i = 1$. Therefore  
\begin{gather*}
a^{(j)}_1 \l|1-\prod_{i=1}^n~ \frac{1}{a^{(j)}_i}\r| \underset{(\hreff{MNP_prop_eqn9})}{>} j \dist\l((a^{(j)}_i)_{i=1}^n ,\R^n_1\r) \geq j\l|a^{(j)}_1 - d^{(j)}_1\r|
\end{gather*}
Since $0<a^{(j)}_1$, we obtain 
\begin{gather*}
 1 \underset{(\hreff{MNP_prop_eqn11})}{\geq} \l|1-\prod_{i=1}^n~ \frac{1}{a^{(j)}_i}\r| \geq j \l| 1 - \frac{d_1^{(j)}}{a_1^{(j)}} \r|
\end{gather*}
Therefore 
\begin{equation} \label{MNP_prop_eqn14}
\frac{1}{2} \leq \frac{a_1^{(j)}}{d_1^{(j)}} \leq 2 
\end{equation}
for all $j>1$. By (\hreff{MNP_lemma_eqn10}), we have  $d_i^{(j)} - a_i^{(j)} = \frac{\lambda^{(j)}}{d_i^{(j)}}$ for $i=1,\ldots,n$ where $\lambda^{(j)} \in \R$. This implies that $$\max\l\lb \l|d_i^{(j)}-a_i^{(j)}\r| : i=1,\ldots,n \r\rb= \max\l\lb \l| \frac{\lambda^{(j)}}{d_i^{(j)}} \r| : i=1,\ldots,n \r\rb = \l|a_1^{(j)}-d_1^{(j)}\r| \ .$$ Therefore 
\begin{equation} \label{MNP_prop_eqn15}
0 \underset{A_j \not\in SL(n)}{<} \l|D_j^{-1}\l(A_j - D_j\r)\r| \leq \frac{\sqrt{n}}{d_1^{(j)}} \l|a_1^{(j)} - d_1^{(j)}\r| \underset{(\hreff{MNP_prop_eqn14})}{\leq} \sqrt{n}
\end{equation}
We have 
\begin{equation} \label{MNP_prop_eqn10}
\begin{gathered}
\frac{1}{\l|\l|D_j^{-1}\r|\r|} \l|D_j\l(A_j - D_j\r) \r| \underset{(\hreff{MNP_prop_eqn15})}{\leq} d_1^{(j)} \cdot \frac{\sqrt{n}}{d_1^{(j)}} \l|a_1^{(j)} - d_1^{(j)}\r| \\ \underset{(\hreff{MNP_prop_eqn16})}{\leq} \sqrt{n} \dist\l((a^{(j)}_i)_{i=1}^n ,\R^n_1\r) \underset{(\hreff{MNP_prop_eqn9})}{<}\frac{\sqrt{n}}{j}~ 
a^{(j)}_1 \l|1-\prod_{i=1}^n~ \frac{1}{a^{(j)}_i}\r| \\ =\frac{\sqrt{n}}{j} \frac{1}{\l|\l|A_j^{-1}\r|\r|} \l|1-\frac{1}{\det A_j} \r| \underset{\frac{1}{2} \leq \det A_j}{\leq} \frac{2\sqrt{n}}{j} \frac{1}{\l|\l|A_j^{-1}\r|\r|} \l|1-\det A_j \r| \\ = \frac{2\sqrt{n}}{j} \frac{1}{\l|\l|A_j^{-1}\r|\r|} \l|1-\det\l(D_j + A_j - D_j \r) \r| 
\end{gathered}
\end{equation}
\begin{gather*}
\underset{\det D_j = 1}{=}   \frac{2\sqrt{n}}{j} \frac{1}{\l|\l|A_j^{-1}\r|\r|} \l|1-\det\l( 1 + D_j^{-1}\l(A_j - D_j\r) \r) \r| \\ =  \frac{2\sqrt{n}}{j} \frac{1}{\l|\l|A_j^{-1}\r|\r|} \l|1-\l( 1 + \tr\l( D_j^{-1}\l(A_j - D_j\r)\r) + \ldots + \det\l( D_j^{-1}\l(A_j - D_j\r) \r) \r)\r| \\ \leq  \frac{2\sqrt{n}}{j} \frac{1}{\l|\l|A_j^{-1}\r|\r|} \sum_{i=1}^n \binom{n}{i}\l| D_j^{-1}\l(A_j - D_j\r)\r|^i \leq \frac{2n^{n+\frac{1}{2}}}{j} \frac{1}{\l|\l|A_j^{-1}\r|\r|} \sum_{i=1}^n \l| D_j^{-1}\l(A_j - D_j\r)\r|^i
\end{gather*}
Note $0 \underset{(\hreff{MNP_prop_eqn15})}{<} \l|D^{-1}_j(A_j -D_j)\r|$. Divide both sides of (\hreff{MNP_prop_eqn10}) by $|D^{-1}_j(A_j -D_j)|$ to obtain $$\frac{1}{\l|\l|D_j^{-1}\r|\r|} \leq \frac{2n^{n+\frac{1}{2}}}{j} \frac{1}{\l|\l|A_j^{-1}\r|\r|} \sum_{i=0}^{n-1} \l| D_j^{-1}\l(A_j - D_j\r)\r|^i \ . $$ Therefore we have
\begin{gather*}
 j \leq  2n^{n+\frac{1}{2}}~ \frac{\l|\l|D^{-1}_j\r|\r|}{\l|\l|A^{-1}_j\r|\r|}  ~\sum_{i = 0}^{n-1} \l|D^{-1}_j (A_j - D_j)\r|^i \\ = 2n^{n+\frac{1}{2}} ~\DS\frac{a_1^{(j)}}{d_1^{(j)}} ~\sum_{i = 0}^{n-1} \l|D^{-1}_j (A_j - D_j)\r|^i \\ \underset{(\hreff{MNP_prop_eqn15})}{\leq} 2 n^{n+\frac{1}{2} + \frac{n}{2}+1}  ~\DS\frac{a_1^{(j)}}{d_1^{(j)}} \underset{(\hreff{MNP_prop_eqn14})}{\leq} 4 n^{n+\frac{n}{2}+\frac{3}{2}}   
\end{gather*}
This is a contradiction for $j>4 n^{n+\frac{n}{2} + \frac{3}{2}}$. Since (\hreff{MNP_prop_eqn9}) is not valid, we deduce (\hreff{MNP_prop_eqn8}). Note $C:=C(n)$ because (\hreff{MNP_prop_eqn9}) pertains to distance between sets in $\R^n$. This shows (\hreff{MNP_prop_eqn6}).

\textbf{(3)} Take $(a_j)_{j=1}^n \in \R^n_{>0}$ with $\theta \leq \det \wda \leq \frac{1}{2}$. By (\hreff{MNP_lemma_eqn10}) we have $d_j - \frac{\lambda}{d_j} = a_j$. By (\hreff{MNP_lemma_eqn12}) we have $a_j < d_j$ for $j=1,\ldots,n$ because $1 > \det \wda$. We obtain 
\begin{gather*}
\frac{d_j}{a_j} - 1 = \frac{d_j-a_j}{a_j} = \frac{\lambda}{d_j} \frac{1}{a_j} \\ \leq \frac{\lambda}{d_1} \frac{1}{a_1} = \frac{d_1-a_1}{a_1} = \frac{d_1}{a_1} - 1 
\end{gather*}
Therefore $1 \leq \frac{d_j}{a_j} \leq \frac{d_1}{a_1}$ for $1\leq j \leq n$. Note
\begin{gather*}
 \frac{1}{\theta} \leq \frac{1}{\det \wda} = \frac{\det D}{\det \wda} = \prod_{j=1}^n \frac{d_j}{a_j} \leq \l( \frac{d_1}{a_1} \r)^n \ .
\end{gather*}
This implies $\l(\frac{1}{\theta^{1/n}} -1\r)a_1 \leq d_1 - a_1$. We have 
\begin{gather*}
 a_1 = \l(\frac{1}{\theta^{1/n}} -1 \r)^{-1} \l(\frac{1}{\theta^{1/n}} -1 \r) a_1 \leq \l(\frac{1}{\theta^{1/n}} -1 \r)^{-1} (d_1 - a_1) \\ \leq \l(\frac{1}{\theta^{1/n}} -1 \r)^{-1} \l|\wda - D\r| \underset{(\hreff{MNP_prop_eqn2})}{=} \l(\frac{1}{\theta^{1/n}} -1 \r)^{-1} \dist(\wda,SL(n))  
\end{gather*}
Therefore 
\begin{equation} \label{MNP_prop_eqn7}
\begin{gathered}
\frac{1}{\l|\l|\wda^{-1} \r|\r|} \l|1 - \frac{1}{\det \wda} \r| \leq \l(\frac{1}{\theta^{1/n}} -1 \r)^{-1} \dist\l(\wda,SL(n)\r) \cdot \l|\frac{1}{\theta} - 1\r| \\ \underset{\theta \leq \frac{1}{2}}{\leq} \l(2^{\frac{1}{n}}-1 \r)^{-1} \frac{1}{\theta} \dist\l(\wda,SL(n)\r) 
\end{gathered}
\end{equation}
\textbf{(4)} Set $C_4:=\frac{1}{2}\l(1+C\r) + \l(2^{\frac{1}{n}}-1 \r)^{-1}$. Note $C_4:=C_4(n)$ because $C:=C(n)$. By (\hreff{MNP_prop_eqn6}) and (\hreff{MNP_prop_eqn7}) we have
\begin{gather*}
\frac{1}{\l|\wda^{-1} \r|} \l|1 - \frac{1}{\det \wda} \r| \leq \frac{1}{\l|\l|\wda^{-1} \r|\r|} \l|1 - \frac{1}{\det \wda} \r| \\ \leq \l(1+C+ \l(2^{\frac{1}{n}}-1 \r)^{-1}~ \frac{1}{\theta}\r)\dist\l(\wda,SL(n)\r) \underset{\theta \leq \frac{1}{2}}{\leq} \frac{C_4}{\theta} \dist\l(\wda,SL(n)\r) 
\end{gather*}
This shows (\hreff{MNP_prop_eqn1}).
\end{proof} 

\begin{corollary} \label{MNP_cor}
Take $0<\theta \leq \frac{1}{2}$. If $A \in \operatorname{Mat}^{n \times n}$ with $\det A \geq \theta >0$, then $$\frac{\theta}{C}~ \frac{1}{\l|A^{-1}\r|}\l|1-\frac{1}{\det A}\r| \leq \dist\l(A,SL(n)\r) \leq C~\frac{1}{\l|A^{-1}\r|}\l|1-\frac{1}{\det A}\r|$$ for a constant $C:=C(n)$.
\end{corollary}

\begin{proof}
Set $C:=\sqrt{n} + C_4$. Note $C:=C(n)$ because $C_4:=C_4(n)$. We have 
\begin{gather*}
\frac{\theta}{C} \frac{1}{\l|A^{-1}\r|}\l|1-\frac{1}{\det A}\r| \leq \frac{\theta}{C_4} \frac{1}{\l|A^{-1}\r|}\l|1-\frac{1}{\det A}\r| \\ \underset{(\hreff{MNP_prop_eqn0})}{\leq} \dist\l(A,SL(n)\r) \\ \underset{(\hreff{MNP_upper_bdd})}{\leq} \frac{\sqrt{n}}{|A^{-1}|} \l|1 - \frac{1}{\det A}\r| \leq C~\frac{1}{|A^{-1}|} \l|1 - \frac{1}{\det A}\r|\ .
\end{gather*}
\end{proof}

\vspace{-0.7cm}

\section{Special linear group and symplectic group} \label{EFA}
   
\subsection{Special linear group} \label{SLG}

The main results of Section \hreff{SLG} are Corollary \hreff{SLG_cor} and Proposition \hreff{SLG_linear}. We generalize Proposition \hreff{OTM_prop} and Proposition \hreff{OTM_prop_1} in Proposition \hreff{SLG_prop} to remove dependence on the multiplicity function and the diameter of the image. Incorporating Proposition  \hreff{MNP_prop}, we can show Corollary \hreff{SLG_cor}. Recall that Corollary \hreff{SLG_cor} allows us to extend (\hreff{BACK_fact}) beyond $n=1$ (cf. (\hreff{OTM_defn_mult_K_eqn})). Proposition \hreff{SLG_linear} treats (\hreff{BACK_bdd_linear}).

\begin{proposition} \label{SLG_prop}
Let $1 \leq p < \infty$. Take $u\in W^{1,\infty}\l(U,\R^n\r)$ with $0 < \lambda \leq \det Du(x) \leq \Lambda$ for a.e. $x \in U$. There exists $s \in L^{\infty}(U,\R^{n})$ differentiable at a.e. $x \in U$ with $\det Ds(x) = 1$ such that $$\int_U~ \l|u - s \r|^p\ dx \leq C \int_U~ \l|1-\det Du\r|^p\ dx$$ where $C=C_5 \l(\frac{1}{\lambda}\r)^{\frac{p}{n}} \l(1+ \frac{1}{\lambda^p} \l(\frac{\Lambda}{\lambda} \r)^{2p-2}\r)$ for a constant $C_5:=C_5(n,p)$.
\end{proposition}

\begin{proof}
Step (1) rephrases the estimate. Step (2) provides a covering of $U$. Step (3) applies Corollary \hreff{OTM_cor} and Corollary \hreff{OTM_cor_1} to the sets in the cover.

\textbf{(1)} Let $C^{(0)}:=C^{(0)}(n)$ denote the constant from \cite[p.~251]{EG}. Let 
\begin{equation} \label{SLG_prop_constant_1}
C^{(1)}:=10~\l(\frac{1}{\lambda}\r)^{\frac{1}{n}} \l(1+ \frac{1}{\lambda} \r)
\end{equation}
denote the constant from Corollary \hreff{OTM_cor_1} with $d=1$, and 
\begin{equation} \label{SLG_prop_constant_p}
C^{(p)}:=C_3\l(\frac{1}{\lambda}\r)^{\frac{p}{n}} \l(1+ \frac{1}{\lambda^p} \l(\frac{\Lambda}{\lambda} \r)^{2p-2}\r)
\end{equation}
denote the constant from Corollary \hreff{OTM_cor} with $d=1$. For $1\leq p <\infty$, set $\delta := C^{(p)} \int_U~\l|1-\det Du\r|^p\ dx$. Set $C_5:=\l(1+6 \cdot 2^{2p} \r)\l(10+C_3\r)$. Note that $C_5:=C_5(n,p)$ because $C_3:=C_3(n,p)$. It suffices to show that 
\begin{equation} \label{SLG_prop_eqn0}
 \int_U~\l|u-s\r|^p~dx \leq \l(1+6 \cdot 2^{2p} \r) \delta
\end{equation}
for some $s \in L^{\infty}\l(U,\R^n\r)$ differentiable for a.e. $x \in U$ with $\det Ds(x) =1$. Note $\delta =0$ if and only $\det Du(x) = 1$ for a.e. $x \in U$. If $\delta = 0$, then set $s=u$. Otherwise, we can assume that $\delta >0$. Choose $\delta_1>0$ such that 
\begin{equation} \label{SLG_prop_eqn1}
\begin{gathered}
  \delta_1~C^{(p)}~\l(1+\l(C^{(0)}\r)^n \r)^p~ \operatorname{Lip}^{pn}(u) \leq \delta \ ,\\ \delta_1 ~\l(1+ \l(C^{(0)}\r)^p\r)~ \operatorname{Lip}^p(u) ~\diam^p(U) \leq \delta \ ,\\ \text{ and } \delta_1 ~\l(\l|\l|u \r|\r|_{L^{\infty}(U)} + \underset{x \in U}{\sup}~\l|x\r| \r)^p \leq \delta
\end{gathered}
\end{equation}
Choose $\delta_2>0$ such that 
\begin{equation} \label{SLG_prop_eqn2}
\begin{gathered}
\delta_2 ~\l(C^{(0)}\r)^p ~\operatorname{Lip}^p(u)~ \diam^p(U) \leq \delta \ ,\\ \text{ and } \delta_2~ \l(\l|\l|u \r|\r|_{L^{\infty}(U)} + \underset{x \in U}{\sup}~\l|x\r| \r)^p \leq \delta 
\end{gathered}
\end{equation}
Extend $u:U \to \R^n$ to $\widetilde{u}:\R^n \to \R^n$ with $\operatorname{Lip}\l(\widetilde{u}\r) = \operatorname{Lip}(u)$ \cite[p.~80]{EG}. There exists $v \in C^1\l(\R^n,\R^n\r)$ with 
\begin{equation} \label{SLG_prop_eqn3}
\l|\l\lb x \in \R^n~:~\widetilde{u}(x) \neq v(x) \text{ or } D\widetilde{u}(x) \neq Dv(x) \r\rb \r| \leq \delta_1
\end{equation}
and $\underset{x \in \R^n}{\sup}~\l|Dv(x)\r| \leq C^{(0)} \operatorname{Lip}(u)$ \cite[p.~251]{EG}. 

\textbf{(2)} Set $V:=\l\lb x \in U~:~\det Dv > \frac{\lambda}{2} \r\rb$. Note that $V \subset U$ open. Note that $\lambda \leq \det Du(x)$ for a.e. $x \in U$ implies 
\begin{equation} \label{SLG_prop_eqn9}
\l|U - V\r| \underset{(\hreff{SLG_prop_eqn3})}{\leq} \delta_1 
\end{equation}
Fix $x_0 \in U$ such that $u(x_0)=v(x_0)$. Choose $M>0$ such that 
\begin{equation} \label{SLG_prop_eqn4}
\l(1-\frac{1}{10^n}\r)^M \l|V\r| \leq \delta_2 \ . 
\end{equation}
For any $x \in V$, the inverse function theorem implies that $v: \overline{B}_r(x) \to \R^n$ injective for $0 < r\ll 1$ because  $\det Dv(x) > 0$ \cite[p.~447]{GT}. Set $V^{(1)} := V$ and $$\mathcal{F}^{(1)}:=\l\lb \overline{B}_r(x) \subset V^{(1)} : r \leq \frac{1}{2C^{(0)} \operatorname{Lip}(u)} \text{ and } v:\overline{B}_r(x) \to \R^n \text{ injective} \r\rb$$ The Vitali covering lemma \cite[p.~27]{EG} implies the existence of a countable subset $\mathcal{G}^{(1)} \subset \mathcal{F}^{(1)}$ of disjoint balls such that $V^{(1)} \subset \underset{\overline{B}_r(x) \in \mathcal{G}^{(1)}}{\DS \cup} \overline{B}_{5r}(x)$. Note that 
\begin{gather*}
 \l|V^{(1)}\r| \leq \underset{\overline{B}_r(x) \in \mathcal{G}^{(1)}}{\sum}\l|\overline{B}_{5r}(x)\r| \\ = 5^n \underset{\overline{B}_r(x) \in \mathcal{G}^{(1)}}{\sum}\l|\overline{B}_{r}(x)\r| \underset{\text{disjoint balls}}{=} 5^n \l| \underset{\overline{B}_r(x) \in \mathcal{G}^{(1)}}{\DS \cup} \overline{B}_{r}(x) \r|  
\end{gather*}
This implies that $\frac{1}{5^n} \l|V^{(1)} \r| \leq \l|\underset{\overline{B}_r(x) \in \mathcal{G}^{(1)}}{\DS \cup} \overline{B}_r(x) \r|$. Therefore $\l|V^{(1)} - \underset{\overline{B}_r(x) \in \mathcal{G}^{(1)}}{\DS \cup} \overline{B}_r(x) \r| \leq \l(1- \frac{1}{5^n}\r)\l|V^{(1)}\r|$. Since $\mathcal{G}^{(1)}$ is countable, there exist $\l\lb \overline{B}_{r_{1,j}}\l(x_{1,j}\r) \r\rb_{j=1}^{m_1} \subset \mathcal{G}^{(1)}$ such that $$\l|V^{(1)} - \DS \cup_{j=1}^{m_1} \overline{B}_{r_{1,j}}\l(x_{1,j}\r)\r| \leq \l(1- \frac{1}{10^n}\r) \l|V^{(1)} \r|$$ Inductively, take $V^{(k)} := V^{(k-1)} - \DS \cup_{i=1}^{k-1} \DS \cup_{j=1}^{m_i} \overline{B}_{r_{i,j}}\l(x_{i,j}\r)$ for $k>1$. Set $$\mathcal{F}^{(k)}:=\l\lb \overline{B}_r(x) \subset V^{(k)} : r \leq \frac{1}{2C^{(0)} \operatorname{Lip}(u)} \text{ and } v:\overline{B}_r(x) \to \R^n \text{ injective} \r\rb \ .$$ The Vitali covering lemma \cite[p.~27]{EG} implies the existence of a countable subset $\mathcal{G}^{(k)} \subset \mathcal{F}^{(k)}$ of disjoint balls such that $\l|V^{(k)} - \underset{\overline{B}_r(x) \in \mathcal{G}^{(k)}}{\DS \cup} \overline{B}_r(x) \r| \leq \l(1-\frac{1}{5^n}\r)\l|V^{(k)}\r|$. Since $\mathcal{G}^{(k)}$ is countable, there exist $\l\lb \overline{B}_{r_{k,j}}\l(x_{k,j}\r) \r\rb_{j=1}^{m_k} \subset \mathcal{G}^{(k)}$ such that $$\l|V^{(k)} - \DS \cup_{j=1}^{m_k} \overline{B}_{r_{k,j}}\l(x_{k,j}\r)\r| \leq \l(1- \frac{1}{10^n}\r) \l|V^{(k)} \r|$$ This implies that 
\begin{gather*}
 \l|V-\DS \cup_{i=1}^k \DS \cup_{j=1}^{m_i} \overline{B}_{r_{i,j}}\l(x_{i,j}\r) \r| = \l|V^{(k)} -  \DS \cup_{j=1}^{m_k} \overline{B}_{r_{k,j}}\l(x_{k,j}\r)\r| \\ \leq \l(1-\frac{1}{10^n}\r)\l|V^{(k)}\r| = \l(1-\frac{1}{10^n}\r)\l|V^{(k-1)} -\DS \cup_{j=1}^{m_{k-1}} \overline{B}_{r_{k-1,j}}\l(x_{k-1,j}\r) \r| \\ \leq \l(1-\frac{1}{10^n}\r)^2\l|V^{(k-1)}\r| \leq \ldots \ldots \leq \l(1-\frac{1}{10^n}\r)^k\l|V^{(1)}\r| = \l(1-\frac{1}{10^n}\r)^k\l|V\r|
\end{gather*}
Therefore 
\begin{equation} \label{SLG_prop_eqn5}
\l|V-\DS \cup_{i=1}^M \DS \cup_{j=1}^{m_i} \overline{B}_{r_{i,j}}\l(x_{i,j}\r) \r| \underset{(\hreff{SLG_prop_eqn4})}{\leq} \delta_2 
\end{equation}
Note that $r_{i,j} \leq \DS\frac{1}{2C^{(0)} \operatorname{Lip}(u)}$ implies 
\begin{gather*}
\l|v(x) - v(y) \r| \leq \l|\l|Dv\r|\r|_{L^{\infty}\l(\R^n,\R^n\r)} \l|x-y\r| \underset{(\hreff{SLG_prop_eqn3})}{\leq} 2r_{i,j}~C^{(0)}~ \operatorname{Lip}(u) \leq 1
\end{gather*}
for all $x,y \in B_{r_{i,j}}\l(x_{i,j}\r)$. Therefore 
\begin{equation} \label{SLG_prop_eqn6}
\diam v\l(B_{r_{i,j}}\l(x_{i,j}\r)\r) \leq 1 \ . 
\end{equation}
\textbf{(3)} By Corollary \hreff{OTM_cor} and Corollary \hreff{OTM_cor_1} there exist $s_{i,j} \in W^{1,\infty}\l(B_{r_{i,j}}\l(x_{i,j}\r),\R^n\r)$ essentially injective measure preserving such that 
\begin{equation} \label{SLG_prop_eqn7}
\int_{B_{r_{i,j}}\l(x_{i,j}\r)}~\l|v-s_{i,j} \r|^p\ dx \underset{(\hreff{SLG_prop_eqn6})}{\leq} C^{(p)} \int_{B_{r_{i,j}}\l(x_{i,j}\r)}~ \l|1-\det Dv \r|^p\ dx \ .
\end{equation}
Note that $\det Ds_{i,j}(x) =1$ for a.e. $x \in B_{r_{i,j}}\l(x_{i,j}\r)$ (cf. Definition \hreff{OTM_defn_mult_K}). Set 
\begin{gather*}
s(x):=\begin{cases} s_{i,j}(x) & \text{ for } x \in B_{r_{i,j}}\l(x_{i,j}\r) \\ x & \text{ for } U - \DS \cup_{i=1}^M \DS \cup_{j=1}^{m_i}  B_{r_{i,j}}\l(x_{i,j}\r) \end{cases}
\end{gather*}
Note $s \in L^{\infty}(U,\R^n)$. Since $\l|\DS \cup_{i=1}^M \DS \cup_{j=1}^{m_i}  \partial B_{r_{i,j}}\l(x_{i,j}\r) \r| = 0$, this implies that $\det Ds(x) =1$ for a.e. $x \in U$ and 
\begin{equation} \label{SLG_prop_eqn8}
\l|V-\DS \cup_{i=1}^M \DS \cup_{j=1}^{m_i} B_{r_{i,j}}\l(x_{i,j}\r) \r| \underset{(\hreff{SLG_prop_eqn5})}{\leq} \delta_2 \ .
\end{equation}
We have 
\begin{gather*}
 \int_U~\l|u-s\r|^p\ dx \leq 2^p \int_U~\l|u-v\r|^p\ dx + 2^p \int_U~\l|v-s\r|^p\ dx \\ = 2^p \int_{\l\lb x\in U~:~u(x) \neq v(x) \r\rb}\l|u-v\r|^p\ dx + 2^p \int_U~\l|v-s\r|^p\ dx \\ \underset{u(x_0) = v(x_0)}{\leq} 2^{2p} \int_{\l\lb x\in U~:~u(x) \neq v(x) \r\rb}\l|u-u(x_0)\r|^p\ dx + 2^{2p} \int_{\l\lb x\in U~:~u(x) \neq v(x) \r\rb}\l|v-v(x_0)\r|^p\ dx + \\ + 2^p \int_U~\l|v-s\r|^p\ dx \\ \underset{(\hreff{SLG_prop_eqn3})}{\leq} 2^{2p}\Lip^p(u) \diam^p(U) \l|\l\lb x\in U~ :~u(x) \neq v(x) \r\rb\r| + \\ + 2^{2p}    \l(C^{(0)}\r)^p \Lip^p(u) \diam^p(U) \l|\l\lb x\in U : u(x) \neq v(x) \r\rb\r| + 2^p \int_U~\l|v-s\r|^p\ dx \\ \underset{(\hreff{SLG_prop_eqn3})}{\leq}  2^{2p}\l(1+\l(C^{(0)}\r)^p\r) \Lip^p(u) \diam^p(U) ~\delta_1 + 2^p \sum_{i=1}^M \sum_{j=1}^{m_i}\ \int_{B_{r_{i,j}}\l(x_{i,j}\r)} \l|v - s_{i,j} \r|^p\ dx + \\ + 2^p \int_{U-\DS \cup_{i=1}^M \DS \cup_{j=1}^{m_i} B_{r_{i,j}}\l(x_{i,j}\r)}\l|v-x\r|^p\ dx \\ \underset{(\hreff{SLG_prop_eqn1})}{\leq}  \delta + 2^p \sum_{i=1}^M \sum_{j=1}^{m_i}\ \int_{B_{r_{i,j}}\l(x_{i,j}\r)} \l|v - s_{i,j} \r|^p\ dx + \\ + 2^p \int_{U-\DS \cup_{i=1}^M \DS \cup_{j=1}^{m_i} B_{r_{i,j}}\l(x_{i,j}\r)}\l|v-x\r|^p\ dx \\
 \underset{(\hreff{SLG_prop_eqn7})}{\leq} \delta + 2^p C^{(p)} \sum_{i=1}^M \sum_{j=1}^{m_i} \int_{B_{r_{i,j}}\l(x_{i,j}\r)} \l|1-\det Dv\r|^p\ dx + \\ + 2^p \int_{U-\DS \cup_{i=1}^M \DS \cup_{j=1}^{m_i} B_{r_{i,j}}\l(x_{i,j}\r)} \l|v-x\r|^p\ dx \\ \underset{\text{disjoint balls}}{\leq} \delta + 2^p C^{(p)}  \int_{U} \l|1-\det Dv\r|^p\ dx + 2^{2p} \int_{U-\DS \cup_{i=1}^M \DS \cup_{j=1}^{m_i} B_{r_{i,j}}\l(x_{i,j}\r)} \l|v-v(x_0)\r|^p\ dx + \\ + 2^{2p} \int_{U-\DS \cup_{i=1}^M \DS \cup_{j=1}^{m_i} B_{r_{i,j}}\l(x_{i,j}\r)} \l|x-u(x_0)\r|^p\ dx 
 \end{gather*}
 \begin{gather*}
  \underset{(\hreff{SLG_prop_eqn3})}{\leq} \delta + 2^{2p} C^{(p)}  \int_{U} \l|1-\det Du\r|^p\ dx + 2^{2p} C^{(p)}  \int_{U} \l|\det Du-\det Dv\r|^p\ dx + \\ + 2^{2p} \l(C^{(0)}\r)^p\Lip^p(u)\diam^p(U) \l|U-\DS \cup_{i=1}^M \DS \cup_{j=1}^{m_i} B_{r_{i,j}}\l(x_{i,j}\r)\r| + \\ + 2^{2p}\l( \l|\l|u\r|\r|_{L^{\infty}(U)} + \underset{x \in U}{\sup}~\l|x\r| \r)^p\l|U-\DS \cup_{i=1}^M \DS \cup_{j=1}^{m_i} B_{r_{i,j}}\l(x_{i,j}\r)\r| \\ \underset{(\hreff{SLG_prop_eqn8}) \text{ and } (\hreff{SLG_prop_eqn9})}{\leq} \l(1+2^{2p}\r)\delta + 2^{2p} C^{(p)} \int_{\l\lb x \in U~:~Du(x) \neq Dv(x) \r\rb} \l| \det Du - \det Dv \r|^p\ dx + \\ + 2^{2p}\l(C^{(0)}\r)^p\Lip^p(u)\diam^p(U) \l(\delta_1 + \delta_2\r) + 2^{2p}\l( \l|\l|u\r|\r|_{L^{\infty}(U)} + \underset{x \in U}{\sup}~\l|x\r| \r)^p\l(\delta_1+\delta_2\r) \\ \underset{(\hreff{SLG_prop_eqn1}) \text{ and } (\hreff{SLG_prop_eqn2})}{\leq} (1+2^{2p})\delta + 2^{2p} C^{(p)} \l(\Lip^n(u) + \Lip^n(v) \r)^p \l| \l\lb x \in U : Du(x) \neq Dv(x) \r\rb \r| + \\ + 2^{2p}\l(\delta + \delta\r) + 2^{2p}\l(\delta + \delta\r) \\ \underset{(\hreff{SLG_prop_eqn3})}{\leq} (1+5 \cdot 2^{2p})\delta +  2^{2p} C^{(p)} \l(\Lip^n(u) + \l(C^{(0)}\r)^n \Lip^n(u) \r)^p \l| \l\lb x \in U : Du(x) \neq Dv(x) \r\rb \r| \\ \underset{(\hreff{SLG_prop_eqn3})}{\leq} (1+5 \cdot 2^{2p})\delta + 2^{2p}C^{(p)} \Lip^{pn}(u)\l(1+\l(C^{(0)}\r)^n\r)^p\delta_1 \\ \underset{(\hreff{SLG_prop_eqn1})}{\leq} \l(1+6 \cdot 2^{2p}\r)\delta
\end{gather*}
This shows (\hreff{SLG_prop_eqn0}). 
\end{proof}

\begin{corollary} \label{SLG_cor}
Let $1 \leq p < \infty$. Take $u \in W^{1,\infty}\l(U,\R^n\r)$ with $0 < \lambda \leq \det Du(x) \leq \Lambda$ for a.e. $x \in U$. There exists $s \in L^{\infty}(U,\R^{n})$ differentiable for a.e. $x \in U$ with $\det Ds(x) = 1$ such that $$\int_U~ \l|u - s \r|^p\ dx \leq C \int_U~ K^p\l(Du\r) \dist^p\l(Du,SL(n)\r)\ dx$$ for a constant $C:=C(n,p,\lambda,\Lambda)$.
\end{corollary}
\begin{proof}
\textbf{(1)} Take $A \in \operatorname{Mat}^{n \times n}$ with $\lambda \leq \det A \leq \Lambda$. Let $0 \leq a_1 \leq \ldots \leq a_n$ denote the singular values of $A$. Note that 
\begin{equation} \begin{gathered} \label{SLG_cor_bdd1}
 \frac{a_n}{a_1} = \frac{a_n^n}{a_1 a_n^{n-1}} \leq \frac{a_n^n}{a_1 \cdots a_n} = \frac{a_n^n}{\det A} \leq K(A)
\end{gathered} \end{equation}
Note $\lambda \leq a_1 \cdots a_n \leq a_n^n$ implies $\lambda^{\frac{1}{n}} \leq a_n$. By 
Proposition \hreff{MNP_prop}, we have 
\begin{equation} \begin{gathered} \label{SLG_cor_bdd4}
 \frac{1}{\l|A^{-1}\r|} \l|1-\frac{1}{\det A} \r| \leq \frac{C_4}{ \operatorname{min}\l\lb \lambda, \frac{1}{2} \r\rb} \dist\l(A,SL(n)\r)
\end{gathered} \end{equation}
Note 
\begin{equation} \label{SLG_cor_bdd3}
 a_1 = \frac{1}{\l|\l| \l(A A^T \r)^{-\frac{1}{2}} \r|\r|} \leq \frac{\sqrt{n}}{\l| \l(A A^T \r)^{-\frac{1}{2}} \r|} \underset{\l(A A^T \r)^{-\frac{1}{2}} A \in O(n)}{=} \frac{\sqrt{n}}{\l|A^{-1}\r|}
\end{equation}
We have 
\begin{gather*}
K(A) \dist\l(A,SL(n)\r) \underset{\text{(\hreff{SLG_cor_bdd1})}}{\geq} \frac{a_n}{a_1} \dist\l(A,SL(n)\r) \\ \underset{(\hreff{SLG_cor_bdd4})}{\geq} \frac{a_n}{a_1}~ \frac{\operatorname{min}\l\lb \lambda, \frac{1}{2} \r\rb}{C_4} \frac{1}{\l|A^{-1}\r|} \l|1-\frac{1}{\det A}\r| \underset{(\hreff{SLG_cor_bdd3})}{\geq} \frac{a_n}{a_1}~ \frac{\operatorname{min}\l\lb \lambda, \frac{1}{2} \r\rb}{C_4} \frac{a_1}{\sqrt{n}}  \l|1-\frac{1}{\det A}\r| \\ = \frac{a_n}{\sqrt{n}}~ \frac{\operatorname{min}\l\lb \lambda, \frac{1}{2} \r\rb}{C_4} \frac{1}{\det A}  \l|1-\det A\r| \underset{\lambda^{\frac{1}{n}}\leq a_n}{\geq} \frac{\lambda^{\frac{1}{n}}}{\Lambda} ~ \frac{\operatorname{min}\l\lb \lambda, \frac{1}{2} \r\rb}{C_4}  \l|1-\det A\r|
\end{gather*}
Therefore 
\begin{equation} \begin{gathered} \label{SLG_cor_bdd2}
\l|1-\det A\r| \leq C_6 ~K(A)\dist\l(A,SL(n)\r) 
\end{gathered} \end{equation}
where $C_6:=C_6(n,\lambda,\Lambda)$. 

\textbf{(2)} By Proposition \hreff{SLG_prop} there exists $s \in L^{\infty}\l(U,\R^n\r)$ differentiable for a.e. $x \in U$ with $\det Ds(x) = 1$ such that $$\int_U~\l|u-s\r|^p\ dx \leq C \int_U~ \l|1-\det Du\r|^p\ dx$$ where $C:=C(n,p,\lambda,\Lambda)$. Therefore 
\begin{gather*}
 \int_U~\l|u-s\r|^p\ dx \leq C \int_U~\l|1-\det Du\r|^p\ dx \\ \underset{(\hreff{SLG_cor_bdd2})}{\leq} C\int_U~ K^p(Du) \dist^p\l(Du,SL(n)\r)\ dx 
\end{gather*}
where $C:=C(n,p,\lambda,\Lambda)$.
\end{proof}
Recall that for $A \in \operatorname{Mat}^{n \times n}$ we have $\det\l(I + \epsilon A\r) = 1 + \epsilon \tr A + o(\epsilon)$ because $D \det\l(I\r) : A = \operatorname{cof} I : A = \tr A$. Therefore $sl(n) := \l\lb A \in \operatorname{Mat}^{n\times n} : \operatorname{tr} A =0\r\rb$ is a linear approximation to $SL(n)$ near the identity matrix. Near the identity map we can approximate the collection of maps with Jacobian equal to one by flows of vector fields with divergence zero. Therefore an estimate comparable to Corollary \hreff{OTM_cor} and Corollary \hreff{OTM_cor_1} should bound the deviation of a map from divergence zero maps by the deviation of its derivative from trace zero matrices.
  
\begin{proposition} \label{SLG_linear}
Take $u \in W^{1,p}(U,\R^n)$. 

\textbf{(1)} Let $1<p<\infty$. There exists $v \in W^{1,p}(U,\R^n)$ with $\div v =0$ such that$$\int_U~ \l|Du - Dv\r|^p\ dx \leq C \int_U~ \dist^p\l(Du,sl(n)\r)\ dx$$ where $C:=C(n,p)$.

\textbf{(2)} Let $p=1$. There exists $v \in W^{1,1}(U,\R^n)$ with $\div v =0$ in $U$ such that $$\int_U~ \l|u-v\r|\ dx \leq C\int_U~ \dist\l(Du,sl(n)\r)\ dx$$ where $C:=C(n, \diam U)$. 
\end{proposition}

\begin{proof}
Note that $sl(n)$ is subspace of $\operatorname{Mat}^{n \times n}$. For $A \in \operatorname{Mat}^{n \times n}$ have orthogonal decomposition $$A=\l(A-\frac{\operatorname{tr}A}{n} I \r) + \frac{\operatorname{tr}A}{n} I\ .$$ Therefore $\dist\l(A,sl(n)\r) = \frac{1}{\sqrt{n}}\l|\operatorname{tr}(A)\r|$.

\textbf{(1)} There exists $\psi \in W^{2,p}(U)$ with $\Delta \psi(x) = \operatorname{div}u(x)$ for a.e. $x \in U$ \cite[p.~230]{GT}. Moreover $$\int_U~ \l|D^2 \psi \r|^p\ dx \leq C \int_U~ \l|\div u \r|^p\ dx$$ where $C:=C(n,p)$. Set $v:= u- D\psi$. Note $v \in W^{1,p}(U,\R^n)$ with $\operatorname{div} v = 0$. Note
\begin{gather*}  
\int_U~ \l|Du - Dv\r|^p\ dx = \int_U~ \l|D^2 \psi\r|^p\ dx \leq C \int_U~ \l| \div u\r|^p\ dx \\ = C \int_U~ n^{\frac{p}{2}} \l|\frac{1}{\sqrt{n}}\tr Du\r|^p\ dx = C \int_U~ \dist^p\l(Du,sl(n)\r)\ dx 
\end{gather*} 
where $C:=C(n,p)$.

\textbf{(2)} If $n=1$, then take $v:=\fint_U u \ dx$. Otherwise, we can assume that $n>1$. If $\div u =0$ in $U$, then set $v:=u$. Otherwise, we can assume that $0 < \delta:=\int_U~ \l|\div u\r|\ dx$. Choose $w \in C^{\infty}(\overline{U},\R^n)$ such that $$\int_U~\l|u-w\r| + \l|Du - Dw\r| \ dx \leq \delta $$ \cite[p.~127]{EG}. Set $$\Gamma(z):=\begin{cases} \frac{1}{n(2-n)\l|B_1(0)\r|} \l|z\r|^{2-n} & \text{ for } n>2 \\ \frac{1}{2\pi}\log \l|z\r| & \text{ for } n=2 \end{cases}$$ and $\psi(x):=\int_U~ \Gamma(x-y) \div w(y)\ dy$. Note $\psi \in C^2(U)$ with $\Delta \psi = \div w$ and $$\psi_{x_i}(x) = \int_U~\Gamma_{x_i}(x-y)\div w(y)\ dy$$ \cite[p.~55]{GT}. Since $\l|\Gamma_{x_i}(x-y)\r| \leq \frac{1}{n\l|B_1(0)\r|}\l|x-y\r|^{1-n}$, Young's inequality \cite[p.~271]{S} implies that 
\begin{gather*}
\l|\l|\psi_{x_i}\r|\r|_{L^1(U)} \leq \l|\l|\Gamma_{x_i} \r|\r|_{L^1(B_{2 \diam U}(0))} \cdot \l|\l|\div w \r|\r|_{L^1(U)} \leq C \l|\l|\div w \r|\r|_{L^1(U)}
\end{gather*}
where $C:=C(n,\diam U)$. Set $v:=w-D\psi$. Note $\psi \in W^{1,1}(U,\R^n)$ and
\begin{gather*}
 \int_U~\l|u-v\r|\ dx \leq \int_U~\l|u-w\r|\ dx + \int_U~ \l| w- v\r|\ dx \\ \leq \delta + \int_U~ \l| D\psi \r|\ dx \leq \delta + C\int_U~ \l| \div w\r|\ dx \\ \leq \delta + C\int_U~\l| \div w - \div u \r|\ dx + C\int_U~ \l| \div u \r|\ dx \\ \leq \l(1+C\r)\delta + C\int_U~ \l| \div u \r|\ dx = \l(1+2C\r) \delta  
\end{gather*}
\end{proof}
Note that Proposition \hreff{SLG_linear} implies (\hreff{BACK_bdd_linear}). The Calder\'{o}n-Zygmund inequality does not hold for $p=1$. Therefore we do not expect a bound on derivatives to hold for $p=1$. Indeed, Korn's inequality does not hold for $p=1$ \cite{CFM}. 

Note that we do not obtain Proposition \hreff{SLG_linear} from Corollary \hreff{OTM_cor} or Corollary \hreff{OTM_cor_1}. Indeed, the Neumann boundary condition arises in Corollary \hreff{OTM_cor} and Corollary \hreff{OTM_cor_1}, but does not arise in Proposition \hreff{SLG_linear}. Therefore a dependence on the region through the constant in the Poincar\'{e} inequality does not arise in Proposition \hreff{SLG_linear}. The construction in \cite[p.~230]{GT} involving the fundamental solution allows for a constant $C:=C(n,p)$ or $C:=C(n,\diam U)$.  

\vspace{-0.3cm}

\subsection{Symplectic group} \label{SG}

The main results of Section \hreff{SG} are Corollary \hreff{SG_cor} and Proposition \hreff{SG_linearization}. We extend Lemma \hreff{OTM_approx_lemma} in Lemma \hreff{SG_approx_lemma}. After showing an analogue of Corollary \hreff{OTM_cor} and Corollary \hreff{OTM_cor_1} in Proposition \hreff{SG_prop}, we deduce Corollary \hreff{SG_cor}. Note that Corollary \hreff{SG_cor} is the analogue of Corollary \hreff{SLG_cor}. Proposition \hreff{SG_linearization} treats (\hreff{BACK_bdd_linear}) for $sp(2n)$.

\begin{definition} \label{SG_defn_sp_SP}
\textbf{(1)} Define the matrix $J := \begin{bmatrix} 0_{n \times n} & -1_{n \times n} \\ 1_{n \times n} & 0_{n \times n} \end{bmatrix}$. Take $Sp(2n) :=$ \newline  $\l\lb A \in \operatorname{Mat}^{2n\times 2n} : A^T J A = J \r\rb$. Take $sp(2n) :=\l\lb A \in \operatorname{Mat}^{2n\times 2n} : JA = (JA)^T \r\rb$. 

\textbf{(2)} Let $B \subset \R^{2n}$ be measurable and $S \in L^{\infty}\l(B,\R^{2n}\r)$. If $S$ is differentiable for a.e. $x \in B$ with $DS(x) \in Sp(2n)$, then call $S$ a symplectic map.
\end{definition}

Note $A \in Sp(2n)$ if and only if $J A v \cdot A w = Jv \cdot w$ for all $v,w \in \R^{2n}$. For $$v=\underbrace{\l(0,\ldots,0,\nu,0,\ldots,0\r)}_{k\text{th entry}} \text{ and } w=\overbrace{\l(0,\ldots,0,\omega,0,\ldots,0\r)}^{(k+n)\text{th entry}}$$ we have that $\l| J v \cdot w\r|$ is the area of the rectangle determined by $v,w$. Therefore we can understand $Sp(2n)$ as the collection of matrices preserving area in coordinate planes $\R_{x_k} \times \R_{x_{k+n}}$. 
\begin{example} \label{SG_example} 
Take $\gamma \in C^1(\R^{2n},\R)$. Consider the flow $$\begin{cases} \frac{d}{dt} \phi(x,t) = D\gamma\l( \phi(x,t) \r)\cdot J & \\ \phi(x,0) = x & \end{cases}$$ The map $\R^{2n} \ni x \to \phi(x,1) \in \R^{2n}$ is a symplectic map generated by the flow of Hamiltonian vector field $D\gamma \cdot J$. For example, $\gamma(x) := \frac{1}{2} \sum_{j=1}^n x_k^2 + x_{k+n}^2$ means $\phi(\cdot,1)$ rotates each $\R_{x_k} \times \R_{x_{k+n}}$ clockwise. Note that the collection of symplectic maps generated by the flow of Hamiltonian vector fields is closed under composition. Indeed, for $\gamma_1, \gamma_2 \in C^1(\R^{2n},\R)$ with flows $\phi_1(\cdot,1),\phi_2(\cdot,1)$, the composition $\phi_2(\cdot,1) \circ \phi_1(\cdot,1)$ corresponds to the Hamiltonian vector field $D\l(\gamma_2(x) + \gamma_1\l(\phi_2(x,-1) \r)\r)\cdot J$. 
\end{example}
Note $Sp(2n) \subset SL(2n)$ implies symplectic maps have Jacobian equal to one. Maps with Jacobian equal to one may not be symplectic for $n>1$. However, we can use observations of Katok \cite[p.~545]{Katok} to obtain approximations in $L^p$ for $1\leq p<\infty$.  
\begin{lemma} \label{SG_approx_lemma}
Take $U \subset \R^{2n}$ and $s: U \to \R^{2n}$ measure preserving. Let $1\leq p<\infty$. For any $\epsilon >0$, there exists $S \in C^{\infty}_{\operatorname{diff}}(\R^{2n},\R^{2n})$ symplectic generated by a Hamiltonian vector field such that $$\int_U~ \l|s - S \r|^p\ dx \leq \epsilon \ .$$   
\end{lemma}
\begin{proof}
\textbf{(1)} Consider $\widetilde{Q} \subset \R^n$ a cube. Suppose $\widetilde{w}_i: \widetilde{Q} \to \widetilde{Q}$ measure preserving and $\widetilde{S}_i \in C^1\l(\widetilde{Q},\widetilde{Q}\r)$ for $i=1,2$. Note  
\begin{gather*}
\int_{\widetilde{Q}}~ \l|\widetilde{S}_2 \circ \widetilde{S}_1 - \widetilde{w}_2 \circ \widetilde{w}_1 \r|^p\ dx \leq 2^p\int_{\widetilde{Q}}~ \l|\widetilde{S}_2 \circ \widetilde{S}_1 - \widetilde{S}_2 \circ \widetilde{w}_1 \r|^p + \l|\widetilde{S}_2 \circ \widetilde{w}_1 - \widetilde{w}_2 \circ \widetilde{w}_1 \r|^p\ dx \\ \underset{\widetilde{w}_1 \text{ measure preserving }}{=} 2^p\int_{\widetilde{Q}}~ \l|\widetilde{S}_2 \circ \widetilde{S}_1 - \widetilde{S}_2 \circ \widetilde{w}_1 \r|^p\ dx + 2^p\int_{\widetilde{Q}}~ \l|\widetilde{S}_2 - \widetilde{w}_2 \r|^p\ dx \\ \leq 2^p \l|\l|D\widetilde{S}_2\r|\r|^p_{L^{\infty}\l(\widetilde{Q},\R^n\r)}  \int_{\widetilde{Q}}~ \l|\widetilde{S}_1 - \widetilde{w}_1 \r|^p\ dx + 2^p\int_{\widetilde{Q}}~ \l|\widetilde{S}_2 - \widetilde{w}_2 \r|^p\ dx 
\end{gather*}
Therefore we have 
\begin{equation} \label{SG_approx_lemma_eqn1}
\begin{gathered}
\l|\l| \widetilde{S}_2 \circ \widetilde{S}_1 - \widetilde{w}_2 \circ \widetilde{w}_1 \r|\r|_{L^p\l(\widetilde{Q},\R^n\r)} \\ \leq 2\l|\l|\widetilde{S}_2 - \widetilde{w}_2 \r|\r|_{L^p\l(\widetilde{Q},\R^n\r)} + 2 \l|\l| D \widetilde{S}_2 \r|\r|_{L^{\infty}\l(\widetilde{Q},\R^n\r)} \l|\l|\widetilde{S}_1 - \widetilde{w}_1 \r|\r|_{L^p\l(\widetilde{Q},\R^n\r)}
\end{gathered}
\end{equation}
\textbf{(2)} Choose $\widetilde{\epsilon}>0$ such that $2^p\l(\widetilde{\epsilon}+\widetilde{\epsilon}^p\r) \leq \epsilon$. We can apply Lemma \hreff{OTM_approx_lemma} because $1<2n$.  There exists a cube $Q \subset \R^n$ containing $U$ and $s_2:Q \to Q$ measure preserving such that $s_2(x) \underset{\forall x \in U}{=} c + s(x)$ where $c \in \R^n$. There exists $N>0$ and a permutation $\sigma: \l\lb 1,\ldots,N \r\rb \to \l\lb 1,\ldots,N \r\rb$ such that $\int_Q~ \l|s_2 - w\r|^p\ dx \leq \widetilde{\epsilon}$ where $$w(x):=x-x_i + x_{\sigma(i)} \text{ for } x_i \in Q_i$$ \cite[p.~154]{BG}. Here $\l\lb Q_i\r\rb_{i=1}^N$ is a division of $Q$ into parallel cubes of equal size with centers $\l\lb x_i \r\rb_{i=1}^N$. Decomposing $\sigma$ into transpositions shows that $w=w_M \circ \cdots \circ w_1$ for $$w_i(x) := \begin{cases} x & \text{ for } x \not\in Q_j \DS \cup Q_k \\ x - x_j + x_k & \text{ for } x \in Q_j \\ x- x_k + x_j & \text{ for } x \in Q_k \end{cases}$$ where $j:=j(i)$, $k:=k(i)$. Therefore 
\begin{equation} \label{SG_approx_lemma_eqn2}
 \int_Q~ \l|s_2 - w_M \circ \cdots \circ w_1\r|^p\ dx \leq \widetilde{\epsilon}
\end{equation}
Note that a transposition of cubes $Q_j$ and $Q_k$ in $Q$ can be decomposed into a composition of transpositions of adjacent cubes in $Q$. Therefore we can assume that $x_j$ and $x_k$ differ in a single entry denoted by $\ell:=\ell(i)$. For any $\delta_i > 0$, there exists $\gamma_i: \R^2 \to \R$ with $\operatorname{spt}~\gamma_i \subset Q_j \DS \cup Q_k$ such that $\int_Q~ \l|\phi_i(\cdot,1) - w_i \r|^p\ dx \leq \delta_i^p$ \cite[p.~151]{BG}. Here we use the notation from Example \hreff{SG_example}. Note that we can identify $\R^2$ with either $\R_{x_{\ell}} \times \R_{x_{\ell+n}}$ or $\R_{x_{\ell-n}} \times \R_{x_{\ell}}$. Therefore we obtain $S_i \in C^{\infty}_{\operatorname{diff}}\l(\R^{2n},\R^{2n}\r)$ symplectic generated by a Hamiltonian vector field such that $\operatorname{spt}~S_i \subset Q_j \DS \cup Q_k$ and
\begin{equation} \label{SG_approx_lemma_eqn3}
\int_Q~ \l|S_i - w_i \r|^p\ dx \leq \delta_i^p \ .
\end{equation}
Set $\delta_M:=\frac{\widetilde{\epsilon}}{2 M}$. For $i=M-1,\ldots,1$, choose $0<\delta_i$ such that 
\begin{equation} \label{SG_approx_lemma_eqn5}
\delta_i~\prod^{M}_{q=i+1}\l|\l|DS_{q}\r|\r|_{L^{\infty}\l(Q,\R^n\r)}  \leq     \DS\frac{\widetilde{\epsilon}}{M 2^{M-i+1}}   \ .
\end{equation}
We have  
\begin{equation} \label{SG_approx_lemma_eqn4}
\begin{gathered}
\l|\l|w_M \circ \cdots \circ w_1 - S_M \circ \cdots \circ S_1 \r|\r|_{L^p\l(Q,\R^n\r)} \underset{(\hreff{SG_approx_lemma_eqn1})}{\leq} 2\l|\l|S_M - w_M\r|\r|_{L^p\l(Q,\R^n\r)} + \\ + \sum_{i=1}^{M-1}~2^{i+1}\l|\l|S_{M-i} - w_{M-i}\r|\r|_{L^p\l(Q,\R^n\r)} \prod^{M}_{q=M-i+1}\l|\l|DS_{q}\r|\r|_{L^{\infty}\l(Q,\R^n\r)} \\ \underset{(\hreff{SG_approx_lemma_eqn3})}{\leq} 2\delta_M + \sum_{i=1}^{M-1}~2^{i+1} \delta_{M-i} \prod^{M}_{q=M-i+1}\l|\l|DS_{q}\r|\r|_{L^{\infty}\l(Q,\R^n\r)} \\ \underset{(\hreff{SG_approx_lemma_eqn5})}{\leq} \frac{\widetilde{\epsilon}}{M} + \sum_{i=1}^{M-1} 2^{i+1} \frac{\widetilde{\epsilon}}{M 2^{i+1}} = \widetilde{\epsilon} 
\end{gathered}
\end{equation}
Set $S:=S_M \circ \cdots S_1 - c$. Note $S \in C^{\infty}_{\operatorname{diff}}\l(\R^{2n},\R^{2n}\r)$ symplectic generated by a Hamiltonian vector field. Here we use the observation from Example \hreff{SG_example}. Note 
\begin{gather*}
 \int_U~ \l|s - S\r|^p \ dx = \int_U~ \l|s_2 - S_M \circ \cdots \circ S_1\r|^p \ dx \underset{U \subset Q}{\leq} \int_Q~ \l|s_2 - S_M \circ \cdots \circ S_1\r|^p \ dx \\ \leq 2^p \int_Q~ \l|s_2 - w_M \circ \cdots \circ w_1\r|^p \ dx + 2^p \int_Q~ \l|w_M \circ \cdots \circ w_1 - S_M \circ \cdots \circ S_1\r|^p \ dx \\ \underset{(\hreff{SG_approx_lemma_eqn2}) \text{ and } (\hreff{SG_approx_lemma_eqn4})}{\leq} 2^p\l(\widetilde{\epsilon} + \widetilde{\epsilon}^p\r) \underset{\text{definition}}{\leq} \epsilon
\end{gather*}
\vspace{-0.5cm}
\end{proof}

Note that we cannot extend Lemma \hreff{SG_approx_lemma} to $L^{\infty}$. The
$L^{\infty}$ limit of a sequence of symplectic maps is a symplectic map. Therefore approximations are restricted to $L^p$ for $1\leq p <\infty$.

\begin{proposition} \label{SG_prop}
Let $1 \leq p < \infty$. Take $u \in W^{1,\infty}\l(U,\R^{2n}\r)$ essentially injective with $\diam u(U) \leq d$ and $0 < \lambda \leq \det Du(x) \leq \Lambda$ for a.e. $x \in U$. There exists $S \in W^{1,\infty}\l(U,\R^{2n}\r)$ essentially injective with $DS^T(x) ~J~ DS(x)=J$ for a.e. $x \in U$ such that $$\int_U~ \l|u - S \r|^p\ dx \leq C \int_U~ \l|Du^T J Du - J\r|^p\ dx$$ where $C=C_7\l(1 + \l(1+\Lambda\r)^{2n-1} \l(\frac{d}{\lambda^{\frac{1}{n}}}\r)^p\l(1 + \frac{1}{\lambda^p}\l(\frac{\Lambda}{\lambda} \r)^{2p-2} \r)   \r)$ for a constant $C_7:=C_7\l(n,p\r)$.
\end{proposition}
\begin{proof} 
\textbf{(1)} Take $A \in \operatorname{Mat}^{2n \times 2n}$ with $\lambda \leq \det A \leq \Lambda$. We show that
\begin{equation} \begin{gathered} \label{SG_prop_bdd6}
 \l|1 - \det A\r| \leq \l(2n\r)^{8n}\l(1+\Lambda\r)^{2n-1} \l|A^T J A - J \r|
\end{gathered} \end{equation}
Note that 
\begin{equation} \label{SG_prop_bdd1}
\begin{gathered}
 \l|\operatorname{cof} A\r| \leq \sum_{1\leq i_1<\ldots<i_{2n-1}\leq 2n} \sum_{\sigma:\l\lb i_1,\ldots,i_{2n-1}\r\rb \to \l\lb j_1,\ldots,j_{2n-1}\r\rb} \l|a^{i_1}_{\sigma(i_1)} \cdots a^{i_{2n-1}}_{\sigma(i_{2n-1})} \r| \\ \leq \l(2n\r)^2\sum_{\sigma:\l\lb i_1,\ldots,i_{2n-1}\r\rb \to \l\lb j_1,\ldots,j_{2n-1}\r\rb} \frac{1}{2n-1} \sum_{k=1}^{2n-1} \l(a^{i_k}_{\sigma(i_k)}\r)^{2n-1} \\ \leq \l(2n\r)^2 \frac{\l(2n-1\r)!}{2n-1} \l(2n-1\r) \l|A\r|^{2n-1} \underset{1\leq n}{\leq} \l(2n\r)^{2n+2} \l|A\r|^{2n-1} 
\end{gathered}
\end{equation}
Note that $\l|1-\det A\r| \leq 1 + \Lambda$. If $1 + \Lambda \leq \l|A^T J A - J \r|$, then $\l|1 - \det A\r| \leq  \l|A^T J A - J \r|$. Otherwise, we can assume 
\begin{equation} \label{SG_prop_bdd2}
1 + \Lambda \geq \l|A^T J A - J \r| \ . 
\end{equation}
Note $1 \leq 1 + \lambda \leq \l|1+\det A\r|$ implies 
\begin{equation} \label{SG_prop_bdd3}
\begin{gathered} 
 \l|1-\det A\r| \leq \l|1-\det A\r| \cdot \l|1+\det A\r|  = \l|1-\l(\det A\r)^2\r| \\ \underset{\det J =1}{=} \l|1-\det \l(A^T J A\r)\r| \underset{\det J =1}{=} \l|\det J -\det\l( A^T J A\r)\r|
\end{gathered}
\end{equation}
Note that $\operatorname{Mat}^{2n \times 2n} \ni X \to \det X \in \R$ is differentiable with $D \det (X) = \operatorname{cof}~X$. We have  
\begin{gather*} 
 \l|1-\det A\r| \underset{(\hreff{SG_prop_bdd3})}{\leq} \l|\det J -\det A^T J A\r| \leq \l|A^T J A - J \r|\underset{0 \leq t \leq 1}{\sup} \l|D\det\l( (1-t)J + t A^T JA \r)\r| \\ = \l|A^T J A - J \r|\underset{0 \leq t \leq 1}{\sup} \l|\operatorname{cof}\l( (1-t)J + t A^T JA \r)\r| \\ \underset{(\hreff{SG_prop_bdd1})}{\leq} \l(2n\r)^{2n+2}\l|A^T J A - J \r|\underset{0 \leq t \leq 1}{\sup} \l| (1-t)J + t A^T JA\r|^{2n-1} 
 \end{gather*}
 \begin{gather*}
 \leq \l(2n\r)^{2n+2} \l|A^T J A - J \r| \l( \l|J\r| + \l|A^T JA - J\r| \r)^{2n-1} \\ \underset{(\hreff{SG_prop_bdd2})}{\leq} \l(2n\r)^{2n+2} \l( \sqrt{2n} + 1 + \Lambda \r)^{2n-1} \l|A^T J A - J \r| \\ \underset{0<\Lambda}{\leq} \l(2n\r)^{2n+2} \l(1+ \sqrt{2n}\r)^{2n-1}\l(1+\Lambda \r)^{2n-1} \l|A^T J A - J \r| \\ \underset{1\leq n}{\leq} \l(2n\r)^{8n} \l(1+\Lambda \r)^{2n-1} \l|A^T J A - J \r|  
\end{gather*}
This shows (\hreff{SG_prop_bdd6}).

\textbf{(2)} If $Du^T(x) J Du(x) = J$ for a.e. $x \in U$, then take $S:=u$. Otherwise, we can assume $\int_U~ \l|Du^T J Du - J\r|^p\ dx >0$. By Corollary \hreff{OTM_cor} and Corollary \hreff{OTM_cor_1}, there exists $s \in W^{1,\infty}\l(U, \R^{2n}\r)$ essentially injective measure preserving such that 
\begin{equation} \label{SG_prop_bdd5}
\int_U~ \l|u-s\r|^p\ dx \underset{(\hreff{SLG_prop_constant_1}) \text{ and } (\hreff{SLG_prop_constant_p})}{\leq} C^{(p)} \int_U~\l|1-\det Du\r|^p\ dx 
\end{equation}
By Lemma \hreff{SG_approx_lemma} there exists $S \in C^{\infty}_{\operatorname{diff}}\l(\R^{2n},\R^{2n}\r)$ with $Ds^T J Ds \equiv J$ such that 
\begin{equation} \label{SG_prop_bdd7} 
\int_U~\l|s-S\r|^p\ dx \leq \int_U~\l|Du^T J Du - J\r|^p\ dx \ . 
\end{equation}
Set $C_7:=2^p\l(2n\r)^{8n}\l(10+C_3\r)$ where $C_3$ from Corollary \hreff{OTM_cor}. We have  
\begin{gather*}
\int_U~\l|u-S\r|^p\ dx \leq 2^p \int_U~\l|u-s\r|^p\ dx + 2^p \int_U~\l|s-S\r|^p\ dx \\ \underset{(\hreff{SG_prop_bdd5}) \text{ and } (\hreff{SG_prop_bdd7})}{\leq} 2^p C^{(p)} \int_U~\l|1 - \det Du\r|^p\ dx + 2^p\int_U~\l|Du^T J Du - J\r|^p\ dx \\ \underset{(\hreff{SG_prop_bdd6})}{\leq} 2^p \l(2n\r)^{8n} \l(1+\Lambda \r)^{2n-1} C^{(p)}\int_U~\l|Du^T J Du - J\r|^p\ dx + 2^p\int_U~\l|Du^T J Du - J\r|^p\ dx \\ \leq C_7\l(1+\l(1+\Lambda \r)^{2n-1}C^{(p)}\r) \int_U~\l|Du^T J Du - J\r|^p\ dx
\end{gather*}
Note that $C_7:=C_7\l(n,p\r)$ because $C_3:=C_3(n,p)$.
\end{proof}
\begin{corollary} \label{SG_cor}
Let $1 \leq p < \infty$. Take $u \in W^{1,\infty}\l(U,\R^{2n}\r)$ with $0 < \lambda \leq \det Du(x) \leq \Lambda$ for a.e. $x \in U$. There exists $S \in L^{\infty}(U,\R^{n})$ differentiable for a.e. $x \in U$ with $DS^T(x)~J~DS(x) = J$ such that $$\int_U~ \l|u - S \r|^p\ dx \leq C \int_U~ K^p\l(Du\r) \dist^p\l(Du,Sp(2n)\r)\ dx$$ for a constant $C:=C(n,p,\lambda,\Lambda)$.
\end{corollary}
\begin{proof}
Note that $Sp(2n) \subset SL(2n)$ implies $\dist\l(A,SL(2n)\r) \leq \dist\l(A,Sp(2n)\r)$ for all $A \in \operatorname{Mat}^{2n \times 2n}$. Therefore  
\begin{equation} \begin{gathered} \label{SG_cor_bdd1}
 \l|1-\det A\r| \underset{(\hreff{SLG_cor_bdd2})}{\leq} C_6~ K(A)~\dist\l(A,SL(n)\r) \leq C_6~ K(A)~\dist\l(A,Sp(2n)\r)
\end{gathered} \end{equation}
If $\dist\l(Du(x),Sp(2n)\r)=0$ for a.e. $x \in U$, then $Du^T(x)~ J~ Du(x) = J$ for a.e. $x \in U$. Set $S:=u$. Otherwise, we can assume that $0<\int_U~ K^p\l(Du\r) \dist^p\l(Du,Sp(2n)\r) \ dx$ because $1 \underset{(\hreff{OTM_defn_mult_K_eqn})}{\leq} K\l(Du\r)$. If $\det Du(x) =1$ for a.e. $x \in U$, then set $s:=u$. Otherwise, we can follow Step (2) and Step (3) of Proposition \hreff{SLG_prop}. We have $V \subset U$ open, and disjoint balls $B_{r_{i,j}}\l(x_{i,j}\r) \subset V$ for $1\leq i \leq M,~ 1\leq j \leq m_i$. There exist $s_{i,j} \in W^{1,\infty}\l(B_{r_{i,j}}\l(x_{i,j}\r),\R^{2n}\r)$ essentially injective measure preserving. We set 
\begin{gather*}
s(x):=\begin{cases} s_{i,j}(x) & \text{ for } x \in B_{r_{i,j}}\l(x_{i,j}\r) \\ x & \text{ for } U - \DS \cup_{i=1}^M \DS \cup_{j=1}^{m_i}  B_{r_{i,j}}\l(x_{i,j}\r) \end{cases}
\end{gather*}
We have
\begin{equation} \label{SG_cor_bdd2}
\begin{gathered}
 \sum_{i=1}^M\sum_{j=1}^{m_i}~\int_{B_{r_{i,j}}\l(x_{i,j}\r)}~\l|u-s_{i,j}\r|^p~dx \underset{\text{disjoint balls}}{\leq} \int_U~\l|u-s\r|^p~dx \\ \underset{(\hreff{SLG_prop_eqn0})}{\leq} \widetilde{C}\int_U~\l|1-\det Du\r|^p~dx
\end{gathered}
\end{equation}
where $\widetilde{C}:=\widetilde{C}\l(n,p,\lambda,\Lambda\r)$. By Lemma \hreff{SG_approx_lemma}, there exists
$S_{i,j} \in  C^{\infty}_{\operatorname{diff}}\l(\R^{2n},\R^{2n}\r)$ with $DS^T~J~DS \equiv J$ such that 
\begin{gather*}
\int_U~\l|s_{i,j} - S_{i,j}\r|^p\ dx \leq 2^{-j -\sum_{l < i} m_l}~\int_U~ K^p\l(Du\r) \dist^p\l(Du,Sp(2n)\r) \ dx   
\end{gather*}
This implies that 
\begin{equation}\label{SG_cor_bdd3}
\begin{gathered} 
 \sum_{i=1}^M \sum_{j=1}^{m_i} \int_{B_{r_{i,j}}\l(x_{i,j}\r)} \l| s_{i,j} - S_{i,j} \r|^p\ dx \\ \leq \int_U~ K^p\l(Du\r) \dist^p\l(Du,Sp(2n)\r) \ dx \sum_{i=1}^M \sum_{j=1}^{m_i} 2^{-j -\sum_{l < i} m_l} \\ \leq \int_U~ K^p\l(Du\r) \dist^p\l(Du,Sp(2n)\r) \ dx
\end{gathered}
\end{equation}
Set
\begin{gather*}
S(x):=\begin{cases} S_{i,j}(x) & \text{ for } x \in B_{r_{i,j}}\l(x_{i,j}\r) \\ x & \text{ for } U - \DS \cup_{i=1}^M \DS \cup_{j=1}^{m_i}  B_{r_{i,j}}\l(x_{i,j}\r) \end{cases}
\end{gather*}
Note $S \in L^{\infty}(U,\R^n)$. Since $\l|\DS \cup_{i=1}^M \DS \cup_{j=1}^{m_i}  \partial B_{r_{i,j}}\l(x_{i,j}\r) \r| = 0$, this implies that $DS^T(x)~J~DS(x) =1$ for a.e. $x \in U$. We have
\begin{gather*}
 \int_U~\l|u-S\r|^p\ dx = \int_{U - \DS \cup_{i=1}^M \DS \cup_{j=1}^{m_i}  B_{r_{i,j}}\l(x_{i,j}\r)}~\l|u-x\r|^p~dx + \sum_{i=1}^M \sum_{j=1}^{m_i}~ \int_{B_{r_{i,j}}\l(x_{i,j}\r)}~\l|u-S_{i,j}\r|^p~dx \\ \leq \l|U - \DS \cup_{i=1}^M \DS \cup_{j=1}^{m_i}  B_{r_{i,j}}\l(x_{i,j}\r)\r|\l(\l|\l|u\r|\r|_{L^{\infty}\l(U,\R^n\r)} + \underset{x \in U}{\sup}~ \l|x\r| \r)^p + \\ + 2^p \sum_{i=1}^M \sum_{j=1}^{m_i}~ \int_{B_{r_{i,j}}\l(x_{i,j}\r)}~\l|u-s_{i,j}\r|^p~dx + 2^p \sum_{i=1}^M \sum_{j=1}^{m_i}~ \int_{B_{r_{i,j}}\l(x_{i,j}\r)}~\l|S_{i,j}-s_{i,j}\r|^p~dx \\ \underset{(\hreff{SG_cor_bdd2}) \text{ and } (\hreff{SG_cor_bdd3})}{\leq} \l|U - \DS \cup_{i=1}^M \DS \cup_{j=1}^{m_i}  B_{r_{i,j}}\l(x_{i,j}\r)\r|\l(\l|\l|u\r|\r|_{L^{\infty}\l(U,\R^n\r)} + \underset{x \in U}{\sup}~ \l|x\r| \r)^p + \\ +2^p~\widetilde{C}\int_U~\l|1-\det Du\r|^p~dx + 2^p\int_U~ K^p\l(Du\r) \dist^p\l(Du,Sp(2n)\r) \ dx \\ \leq \l|U -V\r|\l(\l|\l|u\r|\r|_{L^{\infty}\l(U,\R^n\r)} + \underset{x \in U}{\sup}~ \l|x\r| \r)^p + \\ + \l|V - \DS \cup_{i=1}^M \DS \cup_{j=1}^{m_i}  B_{r_{i,j}}\l(x_{i,j}\r)\r|\l(\l|\l|u\r|\r|_{L^{\infty}\l(U,\R^n\r)} + \underset{x \in U}{\sup}~ \l|x\r| \r)^p + \\ + 2^p~\widetilde{C}\int_U~\l|1-\det Du\r|^p~dx + 2^p\int_U~ K^p\l(Du\r) \dist^p\l(Du,Sp(2n)\r) \ dx \\ \underset{(\hreff{SLG_prop_eqn4}) \text{ and } (\hreff{SLG_prop_eqn5})}{\leq} \l(\delta_1 + \delta_2\r)\l(\l|\l|u\r|\r|_{L^{\infty}\l(U,\R^n\r)} + \underset{x \in U}{\sup}~ \l|x\r| \r)^p + \\ + 2^p~\widetilde{C}\int_U~\l|1-\det Du\r|^p~dx + 2^p\int_U~ K^p\l(Du\r) \dist^p\l(Du,Sp(2n)\r) \ dx 
  \end{gather*}
  \begin{gather*}
  \underset{(\hreff{SLG_prop_eqn1}) \text{ and } (\hreff{SLG_prop_eqn2})}{\leq} \l(2+2^p\r)\widetilde{C}\int_U~\l|1-\det Du\r|^p~dx + 2^p\int_U~ K^p\l(Du\r) \dist^p\l(Du,Sp(2n)\r) \ dx \\ \underset{(\hreff{SG_cor_bdd1})}{\leq} \l(2^p + \l(2+2^p\r)\widetilde{C}~C^p_6 \r)  \int_U~ K^p\l(Du\r) \dist^p\l(Du,Sp(2n)\r) \ dx 
\end{gather*}
Set $C:=2^p + \l(2+2^p\r)\widetilde{C}~C^p_6 $. Note that $C:=C\l(n,p,\lambda,\Lambda\r)$ because $\widetilde{C}:=\widetilde{C}\l(n,p,\lambda,\Lambda\r)$ and $C_6:=C_6\l(n,\lambda,\Lambda\r)$.
\end{proof}
Recall that for $A \in \operatorname{Mat}^{n \times n}$ we have $\l(I + \epsilon A\r)^T~J~\l(I + \epsilon A\r) = J + \epsilon\l(JA - \l(JA\r)^T\r) + o(\epsilon)$. Therefore $sp(2n)$ is a linear approximation to $Sp(2n)$ near the identity matrix. Near the identity map we can approximate the collection of symplectic maps by flows of Hamiltonian vector fields. Therefore an estimate comparable to Proposition \hreff{SG_prop} should bound the deviation of a map from flows of Hamiltonian vector fields by the deviation of its derivative from $sp(2n)$.

\begin{fact} \label{SG_linear_fact}
Let $1 \leq p < \infty$. Take $V \subset \R^{n}$ be open, bounded, and convex. Take $z \in W^{1,p}(V,\R^{n})$. If $Dz(x)=Dz^T(x)$ for a.e. $x \in V$, then there exists $\alpha \in W^{2,p}(V)$ such that $z = D\alpha$.
\end{fact}
\begin{proof}
Fix $x_0 \in V$. For $0 \leq T < 1$ take  $V_T:=\l\lb (1-t) x_0 + t x : x \in \partial V, \; 0 \leq t < T \r\rb$. Note that $\overline{V_T} \subset \subset V$ implies that $\dist(V_T,\partial V) > 0$. For $0 < \varepsilon < \dist(V_T,\partial V)$ take $\beta^{(\varepsilon)} \in C^{\infty}_{\text{cpt}}\l(B_{\varepsilon}(0) \r)$ standard mollifier. Set $z_{\varepsilon}^i(x) := \int_V~ \beta^{(\varepsilon)}(x-y) z^i(y) \ dy$ for $x \in V_T$ and $z_{\varepsilon}:=\l(z_\varepsilon^i\r)_{i=1}^n$. We have  
\begin{gather*}
\l( z_{\varepsilon}^i \r)_{x_j}(x) = \int_V~ \beta^{(\varepsilon)}(x-y) z^i_{x_j}(y) \ dy \\ \underset{Dz=Dz^T}{=} \int_V~ \beta^{(\varepsilon)}(x-y) z^j_{x_i}(y) \ dy = \l( z_{\varepsilon}^j \r)_{x_i}(x)
\end{gather*}
for all $x \in V_T$. Since $V_T$ is simply connected, the Poincar\'{e} lemma implies the existence of $\alpha^{(\varepsilon)}_T \in C^{\infty}(V_T)$ such that $D \alpha^{(\varepsilon)}_T = z_{\varepsilon}$. Assume $\int_{V_T} \alpha^{(\varepsilon)}_T \ dx = 0$. Note $\l|\l|z_{\varepsilon}-z\r|\r|_{W^{1,p}(V_T)} \underset{\epsilon \to 0}{\to} 0$ implies the existence of $\alpha_T \in W^{2,p}(V_T)$ such that $\l|\l|\alpha^{(\varepsilon)}_T-\alpha_T\r|\r|_{W^{2,p}(V_T)} \underset{\epsilon \to 0}{\to} 0$. For any $0\leq T \leq S <1$, we have have $D\alpha_T =Dz= D \alpha_S$ in $V_T$. This implies that $\alpha_S = \alpha_T + C_{T,S}$ where $C_{T,S} \in \R$. Set $$\alpha(x) := \begin{cases} \alpha_{\frac{1}{2}}(x) & \text{ for } x \in V_{\frac{1}{2}} \\ \alpha_{1-\frac{1}{2^n}}(x) - C_{1-\frac{1}{2^{n-1}},1-\frac{1}{2^n}} & \text{ for } x \in V_{1-\frac{1}{2^n}} - V_{1-\frac{1}{2^{n-1}}} \end{cases}$$ for $n>1$. 
\end{proof}
\begin{lemma} \label{SG_linear_lemma}
Take $v \in W^{1,p}(U,\R^{n})$ with $1 < p < \infty$. If $U \subset \R^{n}$ is convex, then 
\begin{equation} \label{SG_linear_lemma_eqn0}
\int_U~ \l|Dv\r|^p\ dx \leq C_8 \int_U~ \l|\div v\r|^p + \l| Dv - Dv^T \r|^p\ dx +  C_8 \int_{\partial U} \l| v \cdot \nu \r|^p \ dS  
\end{equation}
for a constant $C_8:=C_8(p,U)$.
\end{lemma}
\begin{proof}
Set $\widetilde{p} := \max\l\lb p,\frac{p}{p-1} \r\rb$. Set $\varepsilon := \int_U~ \l|\div v\r|^p + \l| Dv - Dv^T \r|^p\ dx + \int_{\partial U} \l| v \cdot \nu \r|^p \ dS$. Step (1) treats $\epsilon=0$. Step (2) and Step (3) treat $\epsilon>0$. 

\textbf{(1)} Take $\varepsilon = 0$. Fact \hreff{SG_linear_fact} implies that $v = D \alpha$ where $\alpha \in W^{2,p}(U)$. Note $$\begin{cases} \Delta \alpha = 0 & \text{ in } U \\ D\alpha \cdot \nu = 0 & \text{ on } \partial U \end{cases} \ .$$ The Neumann problem has a weak solution in $U$ unique up to constant \cite[p.~2147]{GS}. Therefore $v \equiv 0$.

\textbf{(2)} Take $\varepsilon > 0$. Assume $v \in C^{\infty}(\overline{U},\R^n)$ throughout Step (2). Suppose to the contrary that (\hreff{SG_linear_lemma_eqn0}) does not hold for any $C_8>0$. For all $k>0$, there exist $v^{(k)} \in C^{\infty}(\overline{U},\R^n)$ such that 
\begin{equation} \label{SG_linear_lemma_eqn2}
\int_U~ \l|Dv^{(k)}\r|^{p}\ dx > k \int_U~ \l|\div v^{(k)}\r|^{p} + \l| Dv^{(k)} - \l(Dv^{(k)}\r)^T \r|^{p}\ dx + \int_{\partial U} \l| v^{(k)} \cdot \nu \r|^{p} \ dS \ .
\end{equation}
Note $0<\l|\l| v^{(k)} \r|\r|_{W^{1,\widetilde{p}}(U,\R^n)}$. Having scaled by $\l|\l| v^{(k)} \r|\r|^{-p}_{W^{1,\widetilde{p}}(U,\R^n)}$, we can assume that 
\begin{equation} \label{SG_linear_lemma_eqn1}
 \l|\l| v^{(k)} \r|\r|_{W^{1,\widetilde{p}}(U,\R^n)} = 1 \ . 
 \end{equation}
There exists $w \in W^{1,\widetilde{p}}(U,\R^n)$ and a subsequence denoted $\l\lb v^{(k)} \r\rb_{k=1}^{\infty}$ such that 
\begin{equation} \label{SG_linear_lemma_eqn3}
v^{(k)}  \underset{k \to \infty}{\to} w \text{ and } Dv^{(k)} \underset{k \to \infty}{\rightharpoonup} Dw 
\end{equation}
in $L^{\widetilde{p}}(U,\R^n)$. For all $\gamma \in C^{\infty}(\overline{U})$
\begin{gather*}
 \l|\l| \gamma  \r|\r|_{L^{\frac{p}{p-1}}(U)} \cdot \l|\l| \div v^{(k)}  \r|\r|_{L^{p}(U)} + C \l|\l| \gamma \r|\r|_{W^{1,\frac{p}{p-1}}(U)} \cdot \l|\l| v^{(k)} \cdot \nu \r|\r|_{L^{p}(\partial U)} \\ \geq -\int_U~ \gamma \div v^{(k)}\ dx + \int_{\partial U} \gamma v^{(k)} \cdot \nu \ dS \\ = \int_U~ D\gamma \cdot v^{(k)} \ dx 
\end{gather*} 
where $C:=C(U)$ \cite[p.~133]{EG}. Note that $\underset{k \to \infty}{\operatorname{lim}}~\int_U~ D\gamma \cdot v^{(k)} \ dx \underset{(\hreff{SG_linear_lemma_eqn3})}{=} \int_U~ D\gamma \cdot w \ dx$.
Note that (\hreff{SG_linear_lemma_eqn2}) implies $\l|\l| v^{(k)} \cdot \nu \r|\r|_{L^{p}(\partial U)} \to 0$ and $\l|\l| \div v^{(k)} \r|\r|_{L^{p}(U)} \to 0$. Therefore 
\begin{equation} \label{SG_linear_lemma_eqn} 
 \int_U~ D\gamma \cdot w \ dx = 0
\end{equation}
for all $\gamma \in C^{\infty}(\overline{U})$. Note that for any $N \leq k_1 < \ldots < k_M$ and $0< c_{k_j} \leq 1$ with $\sum_{j=1}^M c_{k_j} = 1$, (\hreff{SG_linear_lemma_eqn2}) implies 
\begin{gather*}
 \l|\l|\sum_{j=1}^M c_{k_j}\l( Dv^{(k_j)} - \l(Dv^{(k_j)}\r)^T \r)\r|\r|_{L^p(U,\R^{n\times n})} \\ \leq \underset{1 \leq j \leq M}{\sup} \l|\l| Dv^{(k_j)} - \l(Dv^{(k_j)}\r)^T \r|\r|_{L^p(U,\R^{n\times n})} \sum_{j=1}^M c_{k_j} \\ \leq \underset{k\geq N}{\sup} \l|\l| Dv^{(k)} - \l(Dv^{(k)}\r)^T \r|\r|_{L^p(U,\R^{n\times n})} \underset{N \to \infty}{\to} 0 
\end{gather*}
By (\hreff{SG_linear_lemma_eqn3}), Mazur's theorem gives $Dw=Dw^T$. By Fact \hreff{SG_linear_fact}, we have $w = D \alpha$ where $\alpha \in W^{2,p}(U)$. Therefore (\hreff{SG_linear_lemma_eqn}) implies $$\begin{cases} \Delta \alpha = 0 & \text{ in } U \\ D\alpha \cdot \nu = 0 & \text{ on } \partial U \end{cases} \ .$$ The Neumann problem has a weak solution in $U$ unique up to constant \cite[p.~2147]{GS}. Therefore $w \equiv 0$. Note $$\underset{k\to \infty}{\operatorname{lim}}~\int_U~ \l|Dv^{(k)} - \l(Dv^{(k)}\r)^T \r|^p\ dx \underset{(\hreff{SG_linear_lemma_eqn2})}{=} 0$$ implies that $\l\lb\operatorname{curl} v^{(k)}_{x_j}\r\rb_{k=1}^{\infty}$ lies in a compact subset of $W^{-1,p}(U,\R^{n\times n})$ for $j=1,\ldots,n$. Note $$\underset{k\to \infty}{\operatorname{lim}}~ \int_U~ \l|\div v^{(k)}\r|^p \ dx  \underset{(\hreff{SG_linear_lemma_eqn2})}{=} 0$$ implies the existence of a subsequence denoted $\l\lb v^{(k)} \r\rb_{k=1}^{\infty}$ such that $\div v^{(k)}(x) \to 0$ for a.e. $x \in U$. Note $\l|\l| v^{(k)} \r|\r|_{W^{1,\widetilde{p}}(U,\R^n)} \underset{(\hreff{SG_linear_lemma_eqn1})}{=} 1$ shows that $\underset{k>0}{\sup}~ \l|\l| \div v^{(k)} \r|\r|_{L^{\widetilde{p}}(U)} < \infty$. Therefore dominated convergence implies that $$\int_U~ \l|\div v^{(k)}\r|^{\widetilde{p}} \ dx  \underset{k \to \infty}{\to} 0 \ .$$ This implies that $\l\lb \div v^{(k)}_{x_j} \r\rb_{k=1}^{\infty}$ lies in a compact subset of $W^{-1,\frac{p}{p-1}}(U)$ for $j=1,\ldots,n$. Therefore the div-curl lemma \cite[p.~53]{Evans} implies that for all $\gamma \in C^{\infty}_{\text{cpt}}(U)$ 
$$\int_U~ v^{(k)}_{x_j} \cdot v^{(k)}_{x_j} \gamma \ dx \to \int_U~ w_{x_j} \cdot w_{x_j} \gamma \ dx \underset{w\equiv 0}{=} 0 \ .$$ where $j=1,\ldots,n$. This implies that $\int_U~ \l| Dv^{(k)} \r|^2 \ dx \underset{k \to \infty}{\to} 0$. Therefore there exists a subsequence denoted $\l\lb v^{(k)} \r\rb_{k=1}^{\infty}$ such that $Dv^{(k)}(x) \to 0$ for a.e. $x \in U$. Note $\l|\l| v^{(k)} \r|\r|_{W^{1,\widetilde{p}}(U,\R^n)} \underset{(\hreff{SG_linear_lemma_eqn1})}{=} 1$ shows that $\underset{k>0}{\sup}~\l|\l| Dv^{(k)} \r|\r|_{L^{\widetilde{p}}(U,\R^{n \times n})} <\infty$. Therefore dominated converegence implies that $\int_U~ \l|Dv^{(k)} \r|^{\widetilde{p}} \ dx \underset{k \to \infty}{\to} 0$. Since $\underset{k \to \infty}{\operatorname{lim}}~ v^{(k)} \underset{(\hreff{SG_linear_lemma_eqn3})}{=} w \equiv 0$, we have $\int_U~ \l|v^{(k)} \r|^{\widetilde{p}} \ dx \underset{k \to \infty}{\to} 0$. This contradicts (\hreff{SG_linear_lemma_eqn1}).

\textbf{(3)} For any $\delta > 0$, there exists $v_{\delta} \in C^{\infty}(\overline{U},\R^n)$ such that $\l|\l|v - v_{\delta} \r|\r|_{W^{1,p}(U,\R^n)} \leq \delta$ \cite[p.~127]{EG}. This implies that 
\begin{equation} \label{SG_linear_lemma_eqn4}
\begin{gathered}
 \int_{\partial U} \l|v \cdot \nu - v_{\delta} \cdot \nu \r|^p\ dS \leq \l|\l|\nu\r|\r|^p_{L^{\infty}\l(\partial U,\R^n\r)}  \int_{\partial U} \l|v - v_{\delta} \r|^p\ dS \\ \leq C \int_U~ \l|Dv - Dv_{\delta} \r|^p + \l|v - v_{\delta} \r|^p \ dx \leq 
\end{gathered}
\end{equation}
where $C:=C(p,U)$ \cite[p.~133]{EG}. Therefore 
\begin{equation} \label{SG_linear_lemma_eqn5}
\begin{gathered}
\int_U~ \l|Dv\r|^p \ dx \leq 2^p \int_U~ \l|Dv-Dv_{\delta}\r|^p + \l|D v_{\delta} \r|^p \ dx \\ \underset{\text{Step (2)}}{\leq} 2^p \int_U~ \l|Dv-Dv_{\delta}\r|^p\ dx +\\+ C \int_U~ \l|\div v_{\delta}\r|^p + \l| Dv_{\delta} - Dv_{\delta}^T \r|^p\ dx + \int_{\partial U} \l| v_{\delta} \cdot \nu \r|^p \ dS \\ \underset{(\hreff{SG_linear_lemma_eqn4})}{\leq} C \int_U~ \l|Dv-Dv_{\delta}\r|^p + \l|v-v_{\delta}\r|^p \ dx +\\+ C \int_U~ \l|\div v\r|^p + \l| Dv - Dv^T \r|^p\ dx + \int_{\partial U} \l| v \cdot \nu \r|^p \ dS \\ \leq C~\delta^p + C \int_U~ \l|\div v\r|^p + \l| Dv - Dv^T \r|^p\ dx + \int_{\partial U} \l| v \cdot \nu \r|^p \ dS
\end{gathered}
\end{equation}
where $C:=C(p,U)$. Take $\delta:=\varepsilon^{\frac{1}{p}}$. By (\hreff{SG_linear_lemma_eqn5}), we have $\int_U~ \l|Dv\r|^p \ dx \leq 2 C~\epsilon$. Set $C_8:=2C$. Note $C_8:=C_8(p,U)$ because $C:=C(p,U)$.
\end{proof}
Lemma \hreff{SG_linear_lemma} uses a compensated compactness argument. For $v=\l(v^i\r)_{i=1}^n \in W^{1,p}_0(U,\R^n)$, we can give a different argument. Extend $v$ to $\R^n$. Denote the Fourier transform of $v^i_{x_j}$ by $\widehat{v^i_{x_j}}$. Note that  $$\widehat{v^i_{x_j}} = \l( \frac{1}{\xi_i \xi_j} \sum_{\ell=1}^n \xi^2_{\ell} \r)^{-1} \l( \l( \sum_{\ell=1}^n v^{\ell}_{x_{\ell}} \r)^{\wedge}  + \sum_{k \neq i} \frac{\xi_k}{\xi_i} \l(v^i_{x_k} - v^k_{x_i}\r)^{\wedge}\r)  \ .$$ Note that for all $1\leq i,~j,~k \leq n$
\begin{equation} \label{SG_multiplier_1}
 \l( \frac{1}{\xi_i \xi_j} \sum_{\ell=1}^n \xi^2_{\ell} \r)^{-1} \text{ and } \l( \frac{1}{\xi_i \xi_j} \sum_{\ell=1}^n \xi^2_{\ell} \r)^{-1} ~ \frac{\xi_k}{\xi_i}
\end{equation}
are homogeneous of degree 0, and $C^{\infty}$ away from the origin. Therefore the operators determined by multipliers (\hreff{SG_multiplier_1}) are bounded in $L^p$ \cite[p.~109]{S}. This shows (\hreff{SG_linear_lemma_eqn0}).
\begin{proposition} \label{SG_linearization} 
Let $1<p<\infty$. Take $u \in W^{1,p}(U,\R^{2n})$. If $U \subset \R^{2n}$ is convex with $C^{1,1}$ boundary, then there exists $\psi \in W^{2,p}(U)$ such that $$\int_U~ \l|Du - J D^2 \psi\r|^p\ dx \leq C \int_U~ \dist^p\l(Du,sp(2n)\r)\ dx$$                                         where $C:=C(p,U)$. Note that we can remove any assumption on $\partial U$ for $p=2$.
\end{proposition}
\begin{proof}
Note that $sp(2n)$ is a subspace of $\operatorname{Mat}^{2n \times 2n}$. Suppose $A \in \operatorname{Mat}^{2n \times 2n}$. Recall that the nearest symmetric matrix to $JA$ is $\frac{1}{2}\l( JA + (JA)^T \r)$. This implies that 
\begin{equation} \label{SG_linearization_eqn1}
\begin{gathered}
\dist(A,sp(2n)) = \operatorname{min}\l\lb \l|A-B \r|~:~B \in sp(2n) \r\rb \\ \underset{J \in O(n)}{=} \operatorname{min}\l\lb |JA - JB| ~:~B \in sp(2n) \r\rb = \frac{1}{2}\l|JA - (JA)^T\r| \ . 
\end{gathered}
\end{equation}
There exists $\eta \in W^{2,p}(U)$ with $-\Delta \eta = \operatorname{div}\l(u\cdot J\r)$ in $U$ \cite[p.~230]{GT}. Set $v:= D\eta + u\cdot J$. Note $v \in W^{1,p}(U,\R^{2n})$ with $\operatorname{div} v = 0$. There exists $\phi \in W^{2,p}(U)$ with 
\begin{equation} \label{SG_linearization_eqn2}
\begin{cases} \Delta \phi = 0 & \text{ in } U \\ D\phi \cdot \nu = v \cdot \nu & \text{ on } \partial U \end{cases}
\end{equation}
\cite[p.~126]{Grisvard}. Note that a solution exists for $p=2$ without any assumption on $\partial U$ \cite[p.~149]{Grisvard}. Set $\psi = \eta - \phi$. We obtain
\begin{gather*}
 \int_U~ \l|Du - J D^2 \psi \r|^p\ dx \underset{J\in O(n)}{=} \int_U~ \l| J Du - J^2 D^2 \l( \eta - \phi \r) \r|^p\ dx \\ \underset{J^2=-I}{=} \int_U~\l|D^2 \l( \eta - \phi \r) + D\l( u\cdot J\r) \r|^p\ dx = \int_U~ \l|Dv - D^2 \phi\r|^p\ dx  \\ \underset{\text{Lemma }\hreff{SG_linear_lemma}}{\leq} C_8 \int_U~\l|\div \l(v-D\phi\r) \r|^p + \l|D\l(v-D\phi\r) - D\l(v-D\phi\r)^T\r|^p\ dx + \\ + C_8\int_{\partial U}~\l|\l(v-D\phi\r)\cdot \nu\r|^p~dS \\ \underset{(\hreff{SG_linearization_eqn2})}{=} C_8 \int_U~\l|D\l(v-D\phi\r) - D\l(v-D\phi\r)^T\r|^p\ dx \\ =C_8 \int_U~\l|Dv - Dv^T\r|^p\ dx =C_8 \int_U~\l|D\l(v-D\eta\r) - D\l(v-D\eta\r)^T\r|^p\ dx \\= C_8 \int_U~ \l|D\l(u\cdot J\r) - \l(D\l( u\cdot J\r)\r)^T\r|^p\ dx \\ = C_8 \int_U~ \l|J Du - \l(J Du\r)^T\r|^p\ dx \underset{(\hreff{SG_linearization_eqn1})}{=} 2^p~C_8 \int_U~ \dist^p\l(Du,sp(2n)\r)\ dx
\end{gather*}
Set $C:=2^p~C_8$. Note $C:=C(p,U)$ because $C_8:=C_8(p,U)$.
\end{proof}
Fact \hreff{SG_linear_fact} requires $U$ to be simply connected. However, convexity is appropriate because convexity appears in the Neumann problem with $p\neq 2$. Indeed, there exist Lipschitz domains where the Neumann problem cannot be solved for all $1<p<\infty$ without the convexity assumption \cite[p.~2148]{GS}. Moreover, Desvillettes-Villani treat Lemma \hreff{SG_linear_lemma} for $p=2$ in a convex region with regular boundary \cite[p.~617]{DV}. 
  
\vspace{-0.2cm}  
  
\section{Incompressible limits}  \label{NIM}

We apply the estimates of Section \hreff{OTM} to understand the trend of compressible deformations to incompressible deformatons for $0 \ll \kappa$. The main result of Section \hreff{NIM_Static} is Proposition \hreff{NIM_prop_static}. The main result of Section \hreff{NIM_Dynamic} is Proposition \hreff{NIM_prop}. Recall that the compressibility of a material is measured by the reciprocal of the bulk modulus $$\frac{1}{\kappa} = - \DS\frac{\l(\frac{\Delta V}{V} \r)}{\Delta P}$$ where $P$ denotes pressure, $V$ denotes volume and $\Delta V$ denotes volume change caused by pressure. For rubber-like materials $\kappa \approx 10^9 N/m^2$ meaning a $1\%$ volume decrease arises from $10^7 N/m^2$ of pressure. Typically, the volume will change by less than $0.01\%$ with fracture occuring for a change over $1\%$. Therefore rubber-like materials are nearly incompressible. For a nearly incompressible material, the energy function $W$ is modified to treat volume change as a material constraint. We express $W$ as $$W(A) = W^{\iso}(A) + \kappa W^{dil}(A) \ .$$  Arising from the arrangement of molecules, $W^{\iso}$ measures a structural response to deformation. Arising from molecule-molecule interactions, $W^{\dil}$ measures a fluid response to deformation. Through experiments on the material $U$ involving simple traction, simple shear, etc., we can determine $W^{\iso}$. If we assume $W^{\dil}\l(A\r) = w\l(\det A\r)$, then we can obtain a splitting $$W\l(A\r) = W^{\iso}\l(\DS\frac{1}{\det^{\frac{1}{3}}A}A\r) + \int_1^{\det A}~ w(x)~dx$$ \cite[p.~72]{Dacorogna}. The assumption is appropriate for nearly incompressible materials.
\subsection{Static limits} \label{NIM_Static}
Assume that $\partial U$ is $C^3$. Take $v \in C^3\l(\overline{U},\R^n\r)$ invertible with $v^{-1} \in C^3\l(v\l(\overline{U}\r),\R^n\r)$ and $\l|v(U)\r|=\l|U\r|$. For $p>n$, set $$\mathcal{A}:=\l\lb u \in W^{1,p}\l(U,\R^n\r)~:~u=v \text{ on } \partial U \r\rb$$ and $\mathcal{A}^{\iso}:=\l\lb u \in \mathcal{A}~:~ \det Du = 1 \text{ in } U \r\rb$. Take $W^{\iso} \in C^{\infty}\l(\operatorname{Mat}^{n\times n},\R_{\geq 0}\r)$ polyconvex with $c_1\l|A\r|^p - c_2 \leq W^{\iso}\l(A\r) \leq c_1\l|A\r|^p + c_2$ for all $A \in \operatorname{Mat}^{n\times n}$. Here $c_1,c_2>0$. 
\begin{example} \label{NIM_Dynamic_example} 
 For $n=3$, we have $W_{\text{neo-Hookean}}^{\iso}\l(A\r)=\l|\frac{1}{\det^{\frac{1}{n}} A}~A\r|^2 - n$ is polyconvex  but $W_{\text{Mooney-Rivlin}}^{\iso}\l(A\r)=\l|\frac{1}{\det^{\frac{1}{n}} A}~A\r|^2 - n + \l|\operatorname{cof}\l(\frac{1}{\det^{\frac{1}{n}} A}~A\r) \r|^2 - n$ is not polyconvex \cite[p.~7]{Dacorogna}. 
\end{example}
Let $0<\lambda \leq 1 \leq  \Lambda$. Take $w: \R \to \R\cup \infty$ convex with $w(1)=0$, and 
\begin{equation} \label{NIM_static_assumption}
w(x) \geq \begin{cases} c_3\l(1-x\r)^{\frac{p}{n}} & \text{ for } \lambda \leq x \leq \Lambda \\ \infty & \text{ otherwise }  \end{cases} \ .
\end{equation}
Here $c_3>0$. Assume $W^{\dil}(A) = w\l(\det A\r)$. Define $I_{\kappa}: \mathcal{A} \ni u \to \int_U~W\l(Du\r)~dx \in \R\cup \infty$ and $$I_{\infty}: \mathcal{A} \ni u \to \begin{cases} \int_U~W^{\iso}\l(Du\r)~dx & \text{ for } u \in \mathcal{A}^{\iso} \\ \infty & \text{ otherwise }  \end{cases} \in \R\cup \infty \ .$$
\begin{lemma} \label{NIM_lemma_gamma}
There exists $u_{\kappa} \in \operatorname{argmin}\l\lb I_{\kappa}(u)~:~ u \in \mathcal{A} \r\rb$. There exists \newline $u_{\infty} \in \operatorname{argmin}\l\lb I_{\infty}(u)~:~ u \in \mathcal{A} \r\rb$. Moreover $\underset{\kappa \to \infty}{\Gamma-\lim}~I_{\kappa} = I_{\infty}$ in the weak topology on $\mathcal{A}$.
\end{lemma}
\begin{proof}
\textbf{(1)} There exists $\hat{v} \in C^2\l(v\l(\overline{U}\r),v\l(\overline{U}\r)\r)$ invertible with $\hat{v}^{-1} \in C^2\l(v\l(\overline{U}\r),v\l(\overline{U}\r)\r)$ such that 
\begin{gather*}
 \begin{cases} \det D\hat{v}(x) = \det Dv^{-1}(x) & \text{ for } x \in v\l(U\r) \\ \hat{v}(x) =x & \text{ for } x \in \partial v\l(U\r)   \end{cases}
\end{gather*}
\cite[p.~191]{Dacorogna2}. Here we use $\l|U\r|=\l|v\l(U\r)\r|$. Set $\widetilde{v}:=\hat{v}\circ v$. Note $\widetilde{v} \in C^2\l(\overline{U},v\l(\overline{U}\r)\r)$ invertible with $\widetilde{v}^{-1} \in C^2\l(v\l(\overline{U}\r),\overline{U}\r)$. Note $\det D\widetilde{v} \equiv 1$ and $v|_{\partial U} = \widetilde{v}|_{\partial U}$. Therefore $\mathcal{A} \neq \emptyset$ with $0 \leq \underset{z \in \mathcal{A}}{\inf}~I_{\kappa}\l(z\r) < \infty$, and $\mathcal{A}^{\iso} \neq \emptyset$ with $0 \leq \underset{z \in \mathcal{A}}{\inf}~I_{\infty}\l(z\r) < \infty$. 

Note $W^{\iso}$ polyconvex and $w$ convex implies that $I_{\kappa}$ is polyconvex. Note $W^{\iso}$ coercive and $0\leq w$ implies that $I_{\kappa}$ is coercive. Therefore $I_{\kappa}$ is polyconvex and coercive. There exists $u_{\kappa} \in \operatorname{argmin}\l\lb I_{\kappa}(u)~:~ u \in \mathcal{A} \r\rb$ \cite[p.~32]{Evans}.
 
\textbf{(2)} Take $z_{\infty} \in \mathcal{A}$. Note that 
\begin{equation*} 
 I_{\infty}\l(z_{\infty}\r) = \begin{cases} \infty & \text{ for } z_{\infty} \not\in \mathcal{A}^{\iso} \\ \int_U~W^{\iso}\l(Dz_{\infty}\r)~dx & \text{ for } z_{\infty} \in \mathcal{A}^{\iso}  
  \end{cases}
\end{equation*}
Therefore
\begin{equation} \label{NIM_static_lemma_bdd1}
 \underset{\kappa\to \infty}{\operatorname{limsup}}~ I_{\kappa}\l(z_{\infty}\r) \leq I_{\infty}\l(z_{\infty}\r)
\end{equation}
Suppose that $z_{\kappa_j} \underset{\kappa_j \to \infty}{\rightharpoonup} z_{\infty}$ in $W^{1,p}\l(U,\R^n\r)$ where $\l\lb z_{\kappa_j} \r\rb_{j=1}^{\infty} \subset \mathcal{A}$. There are two cases. If $z_{\infty} \in \mathcal{A}^{\iso}$, then 
\begin{gather*}
 I_{\infty}\l(z_{\infty}\r) \underset{z_{\infty} \in \mathcal{A}^{\iso}}{=} \int_U~W^{\iso}\l(Dz_{\infty}\r)~dx \\ \underset{W^{\iso} \text{ polyconvex}}{\leq} \underset{\kappa_j \to \infty}{\operatorname{liminf}}~\int_U~W^{\iso}\l(Dz_{\kappa_j}\r)~dx \underset{0\leq w}{\leq} \underset{\kappa_j \to \infty}{\operatorname{liminf}}~I_{\kappa_j}\l(z_{\kappa_j}\r) 
\end{gather*}
Otherwise $z_{\infty} \in \mathcal{A} - \mathcal{A}^{\iso}$. Note that $z_{\kappa_j} \underset{\kappa_j \to \infty}{\rightharpoonup} z_{\infty}$ in $W^{1,p}\l(U,\R^n\r)$ implies $\det Dz_{\kappa_j} \underset{\kappa_j \to \infty}{\rightharpoonup} \det Dz_{\infty}$ in $L^{\frac{p}{n}}\l(U,\R^n\r)$ \cite[p.~31]{Evans}. This implies that 
\begin{gather*}
 0\underset{z_{\infty} \in \mathcal{A} - \mathcal{A}^{\iso}}{<} \int_U~w\l(\det Dz_{\infty}\r)~dx \underset{w \text{ convex}}{\leq} \underset{\kappa_j \to \infty}{\operatorname{liminf}}~\int_U~w\l(\det Dz_{\kappa_j}\r)~dx 
\end{gather*}
We obtain $\underset{\kappa_j \to \infty}{\operatorname{liminf}}~I_{\kappa_j}\l(z_{\kappa_j}\r) = \infty = I_{\infty}\l(z_{\infty}\r)$. Therefore we have 
\begin{equation} \label{NIM_static_lemma_bdd2}
 I_{\infty}\l(z_{\infty}\r) \leq \underset{\kappa_j \to \infty}{\operatorname{liminf}}~I_{\kappa_j}\l(z_{\kappa_j}\r)
\end{equation}
in both cases. By (\hreff{NIM_static_lemma_bdd1}) and (\hreff{NIM_static_lemma_bdd2}), we have that $I_{\kappa}$ $\Gamma$-converges to $I_{\infty}$ on $\mathcal{A}$ in the weak topology.

\textbf{(3)} Note that
\begin{equation} \label{NIM_static_lemma_bdd3}
\begin{gathered}
 \int_U~ c_1\l|Du_{\kappa}\r|^p - c_2~dx \leq \int_U~W^{\iso}\l(Du_{\kappa}\r)~dx \\ \underset{0\leq w}{\leq} I_{\kappa}\l(u_{\kappa}\r) \underset{u_{\kappa} \in \operatorname{argmin}\l\lb I_{\kappa}(u)~:~ u \in \mathcal{A} \r\rb}{\leq} I_{\kappa}\l(\widetilde{v}\r) \\ \underset{\widetilde{v} \in \mathcal{A}^{\iso}}{=} \int_U~W^{\iso}\l(D\widetilde{v}\r)~dx
\end{gathered}
\end{equation}
This implies that $\underset{0<\kappa}{\sup}~\l|\l|u_{\kappa}\r|\r|_{W^{1,p}\l(U,\R^n\r)} <\infty$. Therefore there exists a subsequence $\l\lb u_{\kappa_j}\r\rb_{j=1}^{\infty}$ and $u_{\infty} \in \mathcal{A}$ such that $u_{\kappa_j} \underset{\kappa_j \to \infty}{\rightharpoonup} u_{\infty}$ in $W^{1,p}\l(U,\R^n\r)$. We have 
\begin{equation*} 
I_{\infty}\l(u_{\infty}\r) \underset{(\hreff{NIM_static_lemma_bdd2})}{\leq} \underset{\kappa_j \to \infty}{\operatorname{liminf}}~I_{\kappa_j}\l(u_{\kappa_j}\r) < \infty \ . 
\end{equation*}
Therefore $u_{\infty} \in \mathcal{A}^{\iso}$. Take $z \in \mathcal{A}$. We have 
\begin{gather*} 
I_{\infty}\l(u_{\infty}\r) \underset{(\hreff{NIM_static_lemma_bdd2})}{\leq} \underset{\kappa_j \to \infty}{\operatorname{liminf}}~I_{\kappa_j}\l(u_{\kappa_j}\r) \\ \underset{u_{\kappa} \in \operatorname{argmin}\l\lb I_{\kappa}(u)~:~ u \in \mathcal{A} \r\rb}{\leq} \underset{j \to \infty}{\operatorname{liminf}}~I_{\kappa_j}\l(z\r) \underset{(\hreff{NIM_static_lemma_bdd1})}{\leq} I_{\infty}\l(z\r)  
\end{gather*}
Therefore $u_{\infty} \in \operatorname{argmin}\l\lb I_{\infty}(u)~:~ u \in \mathcal{A} \r\rb$
\end{proof}
Note that we do not have a growth condition on $w$. The determinant constraint requires that $w$ be unbounded. Therefore we cannot apply partial regularity for polconvex functionals to show that the minimizers are Lipschitz \cite{Kristensen}.   
\begin{proposition} \label{NIM_prop_static}
 If $u_{\kappa} \in W^{1,\infty}\l(U,\R^n\r)$, then there exists $s_{\kappa} \in W^{1,\infty}\l(U,\R^n\r)$ essentially injective measure preserving such that $$\int_U~\l|u_{\kappa} - s_{\kappa}\r|^{\frac{p}{n}}~dx \leq \DS\frac{C}{\kappa}$$ where $C:=C\l(p,U,v,W\r)$.
\end{proposition}
\begin{proof}
Note that $I_{\kappa}\l(u_{\kappa}\r) < \infty$ implies that $0<\lambda \leq \det Du_{\kappa} \leq \Lambda$ for a.e. $x \in U$. Recall that $v \in C^3\l(\overline{U},\R^n\r)$ and $u_{\kappa} \in W^{1,\infty}\l(U,\R^n\r)$ with $u_{\kappa}=v$ on $\partial U$. Therefore $u_{\kappa}$ is essentially injective with $u_{\kappa}\l(\overline{U}\r) = v\l(\overline{U}\r)$ \cite[p.~3]{Ball}. This implies that $\diam u(U) \leq 2~\l|\l|v\r|\r|_{L^{\infty}\l(U,\R^n\r)}$. By Corollary \hreff{OTM_cor}, there exists $s_{\kappa} \in W^{1,\infty}\l(U,\R^n\r)$ essentially injective measure preserving such that 
\begin{equation} \label{NIM_prop_static_bdd1}
\int_U~\l|u_{\kappa} - s_{\kappa}\r|^{\frac{p}{n}}~dx \leq \hat{C} \int_U~\l|1-\det Du_{\kappa}\r|^{\frac{p}{n}}~dx   
\end{equation}
where $\hat{C}:=\hat{C}\l(p,2~\l|\l|v\r|\r|_{L^{\infty}\l(U,\R^n\r)},\lambda,\Lambda\r)=\hat{C}\l(p,U,v,W\r)$. We have
\begin{gather*}
\kappa~\int_U~\l|u_{\kappa} - s_{\kappa}\r|^{\frac{p}{n}}~dx \underset{(\hreff{NIM_prop_static_bdd1})}{\leq} \hat{C}~\kappa~ \int_U~\l|1-\det Du_{\kappa}\r|^{\frac{p}{n}}~dx \\ \underset{(\hreff{NIM_static_assumption})}{\leq} \frac{\hat{C}~\kappa}{c_3}~\int_U~w\l(\det Du_{\kappa}\r)~dx \underset{0\leq W^{\iso}}{\leq} \frac{\hat{C}}{c_3}~I_{\kappa}\l(u_{\kappa}\r) \\ \underset{(\hreff{NIM_static_lemma_bdd3})}{\leq}    \frac{\hat{C}}{c_3}~\int_U~W^{\iso}\l(D\widetilde{v}\r)~dx 
\end{gather*}
Set $C:=\frac{\hat{C}}{c_3}~\int_U~W^{\iso}\l(D\widetilde{v}\r)~dx$. Note that $C:=C\l(p,U,\widetilde{v},W,\hat{C}\r)=C\l(p,U,v,W\r)$.
\end{proof}
\subsection{Dynamic limits} \label{NIM_Dynamic}
Take $W^{\iso} \in C^{\infty}\l(\operatorname{Mat}^{n\times n},\R\r)$. Assume that   
\begin{equation} \begin{gathered} \label{NIM_condition_iso}
 W^{\iso}_{p^i_k p^j_l}(A) F^i_k F^j_l \geq \theta \l|F\r|^2 
\end{gathered} \end{equation}
for all $F \in \operatorname{Mat}^{n \times n}$ and $A \in N \subset \operatorname{Mat}^{n \times n}$. Here $N \ni I$ is an open set, and $0<\theta$. Take $w \in C^{\infty}\l(\R,\R\r)$. Assume that $W^{\dil}(A) = w\l(\det A\r)$ with     
\begin{equation} \begin{gathered} \label{NIM_condition_dil}
w''(1)>0 \text{ and } w'(1)=0=w(1) 
\end{gathered} \end{equation}
Set $\ell:=\frac{n}{2}+4$. Take $v(x;\kappa), ~\widetilde{v}(x;\kappa) \in W^{\ell,2}\l(\R^n,\R^n\r)$ with $v(x;\kappa)$ invertible and $\det Dv \equiv 1$. Assume that 
\begin{equation} \label{NIM_condition_initial}
\begin{gathered}
\underset{\kappa>0}{\sup}~\l|\l|v(x;\kappa) -x\r|\r|_{W^{\ell,2}\l(\R^n,\R^n\r)} + \l|\l|\widetilde{v}(x;\kappa)\r|\r|_{W^{\ell,2}\l(\R^n,\R^n\r)} < \infty \\ \text{ and } \DS \cup_{\kappa>0}\DS\cup_{x \in \R^n} Dv(x;\kappa) \subset \subset N \ .
\end{gathered}
\end{equation}
By assumptions (\hreff{NIM_condition_iso}), (\hreff{NIM_condition_dil}), and (\hreff{NIM_condition_initial}) we can determine $T:=T\l(v,\widetilde{v},W\r)$ and $\kappa_0:=\kappa_0\l(v,\widetilde{v},W\r)$ such that for any $\kappa>\kappa_0$ there exists $u(x,t;\kappa) \in C^2\l(\R^n \times [0,T],\R^n\r)$ satisfying (\hreff{BACK_eqn_motion}) \cite[p.~212]{Schochet}. Here $u\l(\cdot,t;\kappa\r)$ is injective with \newline $\underset{\kappa > \kappa_0}{\sup}\ \underset{t \in [0,T]}{\sup}\ \l|\l|u(x,t;\kappa)-x\r|\r| _{W^{\ell,2}\l(\R^n,\R^n\r)} < \infty$. 
\begin{proposition} \label{NIM_prop}
For any $\kappa>\kappa_1:=\kappa_1\l(U,v,\widetilde{v},W\r)$, there exist $\l\lb s(x,t;\kappa) \r\rb_{t \in [0,T]} \subset W^{\ell,2}\l(U,\R^n\r)$  with $\det Ds \equiv 1$ such that  $$\underset{t\in [0,T]}{\sup}~\int_{U}\l|u(x,t;\kappa) - s(x,t;\kappa)\r|^2\ dx \leq \frac{C}{\kappa} \ .$$ where $C:=C\l(U,v,\widetilde{v},W\r)$. 
\end{proposition}
\begin{proof}
For all $\kappa > \kappa_0$ and $t \in [0,T]$, we have 
\begin{gather*}
\l|\l|u(x,t;\kappa) - x\r|\r|_{W^{\ell,2}\l(\R^n,\R^n\r)} + \sqrt{\kappa} \l|\l|1 - \det Du(x,t;\kappa)\r|\r|_{W^{\ell-1,2}\l(\R^n,\R^n\r)}  \leq C'
\end{gather*}
where $C':=C'\l(v,\widetilde{v},W\r)$ \cite[p.~6]{Schochet}. Set $r:=\underset{x\in U}{\sup}~\l|x\r|$. Since $\ell>\frac{n}{2}+3$, this implies   
\begin{equation} \label{NIM_prop_bdd2} 
 \l|\l|u(x,t;\kappa) - x\r|\r|_{W^{1,\infty}\l(B_r(0),\R^n\r)} + \sqrt{\kappa} \l|\l|1 - \det Du(x,t;\kappa)\r|\r|_{L^{\infty}\l(B_r(0),\R^n\r)}  \leq \widetilde{C}
\end{equation}
where $\widetilde{C}:=\widetilde{C}\l(C',r\r) = \widetilde{C}\l(U,v,\widetilde{v},W\r)$. Therefore 
\begin{gather*}
 \l|u(y,t;\kappa)-u(z,t;\kappa)\r| - \l|y-z\r| \leq \l|\l(u(y,t;\kappa) - y\r)-\l(u(z,t;\kappa) - z\r)\r| \\ \leq \l|\l|u(x,t;\kappa) - x\r|\r|_{W^{1,\infty}\l(B_r(0),\R^n\r)} \l|y-z\r| \leq 2~r~ \widetilde{C}
\end{gather*}
for all $y,z \in B_r(0)$. This implies that 
\begin{equation} \begin{gathered} \label{NIM_prop_bdd3}
 \underset{t \in [0,T]}{\sup}~ \l|u(y,t;\kappa)-u(z,t;\kappa)\r| \leq 2~r \l(1+\widetilde{C}\r) 
\end{gathered} \end{equation}
for all $\kappa > \kappa_0$. Set $\kappa_1:=\kappa_0 + 4~\widetilde{C}^2$. We have 
\begin{equation} \label{NIM_prop_bdd4}
 \underset{t \in [0,T]}{\sup}~\l|\l|1 - \det Du(x,t;\kappa)\r|\r|_{L^{\infty}\l(B_r(0),\R^n\r)} \underset{(\hreff{NIM_prop_bdd2})}{\leq} \frac{1}{\sqrt{\kappa}} \widetilde{C} \leq \frac{1}{2}
\end{equation}
for all $\kappa > \kappa_1$. By Corollary \hreff{OTM_cor} with $\lambda =\frac{1}{2}$, $\Lambda = \frac{3}{2}$, and $d=2r\l(1+\widetilde{C}\r)$, there exists $s_1\l(x,t;\kappa\r) \in W^{1,\infty}\l(U,\R^{n}\r)$ essentially injective measure preserving such that 
\begin{equation} \label{NIM_prop_bdd1}
\int_U~ \l|u(x,t;\kappa)-s_1(x,t;\kappa)\r|^2\ dx \leq \hat{C} \int_U~ \l| 1 - \det Du(x,t;\kappa) \r|^2\ dx 
\end{equation}
where $\hat{C}:=\hat{C}\l(U,\widetilde{C}\r)=\hat{C}\l(U,v,\widetilde{v},W\r)$. If $\int_U~ \l| 1 - \det Du(x,t;\kappa) \r|^2\ dx = 0$, then set $s(x,t;\kappa):=u(x,t;\kappa)$. Otherwise, we can assume that $\int_U~ \l| 1 - \det Du(x,t;\kappa) \r|^2\ dx > 0$. By Lemma \hreff{OTM_approx_lemma}, there exists $s_2\l(x,t;\kappa\r) \in C^{\infty}_{\operatorname{diff}}\l(\R^n,\R^n\r)$ with $\det D s_2 \equiv 1$ such that 
\begin{equation} \label{NIM_prop_bdd5}
\int_U~ \l|s_1\l(x,t;\kappa\r)-s_2\l(x,t;\kappa\r)\r|^2\ dx \leq  \int_U~ \l| 1 - \det Du(x,t;\kappa) \r|^2\ dx
\end{equation}
Set $s(x,t;\kappa) := s_2\l(x,t;\kappa\r)$. For any $\kappa > \kappa_1$, we have 
\begin{equation}
\begin{gathered}
\int_U~\l|u(x,t;\kappa)-s(x,t;\kappa)\r|^2\ dx \\ \leq 4\int_U~\l|u(x,t;\kappa)-s_1\l(x,t;\kappa\r)\r|^2\ dx + 4\int_U~\l|s_1\l(x,t;\kappa\r)-s_2(x,t;\kappa)\r|^2\ dx \\ \underset{(\hreff{NIM_prop_bdd5})}{\leq} 4\int_U~\l|u(x,t;\kappa)-s_1\l(x,t;\kappa\r)\r|^2\ dx + 4 \int_U~ \l| 1 - \det Du(x,t;\kappa) \r|^2\ dx \\ \underset{(\hreff{NIM_prop_bdd1})}{\leq} \l(4+4~\hat{C}\r) \int_U~ \l| 1 - \det Du(x,t;\kappa) \r|^2\ dx \underset{(\hreff{NIM_prop_bdd4})}{\leq} \l(4+4~\hat{C}\r) ~\l| U \r|~ \frac{\widetilde{C}^2}{\kappa} 
\end{gathered}
\end{equation}
Set $C:=\l(4+4~\hat{C}\r) ~\l| U \r|~ \widetilde{C}^2$. Note that $C:=C\l(U,v,\widetilde{v},W\r)$ because $\widetilde{C}:=\widetilde{C}\l(U,v,\widetilde{v},W\r)$ and $\hat{C}:=\hat{C}\l(U,v,\widetilde{v},W\r)$.
\end{proof}
Note that the convexity assumption on $W^{\iso}$ is restrictive. However, the addition of a null Lagrangian to the energy function may yield (\hreff{NIM_condition_iso}) without affecting the dynamics \cite[p.~209]{Schochet}.

\vspace{-0.5cm}
\section*{Acknowledgements} 
Interest in the problem arose from discussions with J. Krishnan, M. Lewicka and M.R. Pakzad. The project benefited from suggestions of M. Christ, M. Fathi, and S. Govindjee. The research was supported through NSF RTG grant DMS-1344991 and NSF grant DMS-1301661.     

\vspace{-0.1cm}
\bibliographystyle{acm}
\bibliography{references_draft2.bib}

\end{document}